\documentclass{article}
\usepackage[francais]{babel}
\usepackage[latin1]{inputenc}
\usepackage[T1]{fontenc}
\usepackage{amsmath}
\usepackage{amssymb}
\usepackage{amsfonts}
\usepackage{amsthm}
\usepackage{caption}
\usepackage{graphicx}
\usepackage{makeidx}
\usepackage{flafter}
\usepackage{bm}
\usepackage[all]{xy}
\usepackage[left,modulo]{lineno}
\usepackage{multirow}
\usepackage{lscape}
\usepackage[ps2pdf]{hyperref} 	% ps2pdf needed for compilation via Latex -> dvips -> ps2pdf
\hypersetup{ colorlinks=true,linkcolor=blue, citecolor=blue, urlcolor=blue  }

\newtheorem{theorem}{\textsc{Th\'eor\`eme}}[section]

\newtheorem{definition}[theorem]{\textsc{D\'efinition}}
\newtheorem{example}[theorem]{Example}

\newtheorem{notation}[theorem]{\textsc{Notation}}
\newtheorem{proposition}[theorem]{\textsc{Proposition}}

\newtheorem{remark}[theorem]{\textsc{Remarque}}
\newtheorem{theorem-definition}[theorem]{Theorem and Definition}
\numberwithin{equation}{section}
\makeindex

\begin{document}

\title{Structures bihamiltoniennes partielles}

\author{Patrick Cabau \& Fernand Pelletier}
\date{}
\maketitle 

\begin{abstract}
Nous introduisons dans cet article trois notions de structures bihamiltoniennes partielles ($\operatorname{PQ}$, $\operatorname{PN}$ et $\operatorname{P\Omega}$) dans le cadre $c^\infty$-complet au sens de Fr\"{o}licher, Kriegl et Michor (ou cadre adapté), étudions des objets géométriques liés à ces structures et en donnons des exemples à la fois en dimensions finie et infinie. 
\end{abstract}
\medskip
\textbf{Mots clefs~:} 
Structures de Poisson partielles, variétés bistructurées partielles, distributions caractéristiques, structures adaptées, structures $c^\infty$-complètes.\\
\textbf{MSC 2020~:} 53D17, 46T05, 58A30, 18A30.

\tableofcontents
%!+ Ajout de remerciements
\section*{Remerciements}
Les auteurs remercient le \textit{referee} du \textit{Bulletin des Sciences Mathématiques} pour sa relecture minutieuse et ses remarques très pertinentes.

\section{Introduction}

En Mécanique ou en Physique Mathématique, les structures mises en jeu (variétés, fibrés, algébroïdes de Lie) peuvent être de dimension finie, de type Banach, Hilbert ou Fréchet, ou plus généralement des limites inverses ou directes de telles structures.

\bigskip 
Dans ces domaines, la notion de système hamiltonien apparaît comme fondamentale et est intiment liée à l'existence de structure de Poisson. En 1978, F. Magri introduit dans \cite{Mag78} le concept de structure bihamiltonienne\index{structure!bihamiltonienne} pour les systèmes intégrables (voir aussi le document de référence \cite{MaMo84} pour une étude développée émaillée de nombreux exemples). Du point de vue géométrique, ceci correspond à l'existence de deux structures de Poisson
\index{structure!de Poisson} liées par un opérateur de récursion permettant de générer une suite d'intégrales premières.

Cette structure apparaît bien entendu sur des espaces de configuration de dimension finie, notamment sur des groupes de Lie (correspondant aussi à des
groupes de symétrie).

Il s'avère que ce concept est aussi parfaitement adapté au cadre des équations d'évolution : la célèbre équation de Korteweg-de Vries (KdV)%
\index{KdV@$u_{t}=u_{xxx}+6uu_{x}$ (\'equation de KdV)}
\[
u_{t}=u_{xxx}+6uu_{x}
\]
peut être écrite sous forme bihamiltonienne dans le cadre de divers espaces de configuration définis \textit{via} des limites projectives~:

\begin{description}
\item[$\bullet$] groupe de Virasoro
\index{groupe!de Virasoro}\index{Virasoro!groupe}
$\operatorname{Diff}^{\infty}\left(  \mathbb{S}^{1}\right)  \times\mathbb{R}$
(cf. \cite{KhMi03}) où elle correspond à l'équation d'Euler décrivant le flot géodésique relativement à une métrique $L^{2}$ invariante à droite
\footnote{La notion d'équation d'Euler sur un groupe de Lie $G$ a été introduite
par V. Arnold dans \cite{Arn66}~:
\par
\begin{description}
\item[$\bullet$] 
$G=SO(3)$ pour un solide~;
\item[$\bullet$] 
$G=\operatorname{SDiff}(M)$ groupe des difféomorphismes
préservant le volume pour un fluide idéal remplissant un domaine $M$.
\end{description}
}, ce groupe apparaissant comme limite projective des groupes
$\operatorname{Diff}^{k} \left( \mathbb{S}^{1} \right)  \times \mathbb{R}$ (cf. \cite{Sch04})~;
\item[$\bullet$] 
tour hilbertienne d'espaces de Sobolev $\left(  \mathcal{H}^{n}\right)  _{n\in\mathbb{Z}}$ , où $\mathcal{H}^{n}=H^{n} \left( \mathbb{S}^{1},\mathbb{R} \right)  $ (cf. \cite{KaMa98})~;
\item[$\bullet$] 
algébroïde de Lie variationnelle (cf. \cite{KiVa11}) au-dessus du fibré des jets d'ordre infini de sections d'un fibré de rang fini au dessus d'une variété de dimension finie.
\end{description}

En dimension finie, la notion de structure de Poisson correspond à une variété lisse $M$ dont l'algèbre des fonctions lisses $C^\infty(M)$ est munie d'un crochet $\{.,.\}$, i.e. d'une application bilinéaire antisymétrique $C^\infty(M) \times C^\infty(M) \to C^\infty(M)$ qui satisfait à la fois les identités de Leibniz et de Jacobi. A toute fonction lisse $f$, est associée \textit{via} le crochet une dérivation $g \mapsto \{f,g\}$ qui correspond à un champ de vecteurs $X_f$ appelé champ de vecteurs hamiltonien associé à l'hamiltonien\index{hamiltonien} $f$.\\

%!+ restructuration 
Une extension d'une telle structure au contexte des espaces de Banach a été définie dans \cite{OdRa03} où les auteurs considèrent un crochet de Poisson sur l'algèbre $C^\infty(M)$ d'une variété $M$, prédual d'une $W^\ast$ algèbre (cf. \cite{OdRa08a}, \cite{OdRa08b} et \cite{Rat11} pour des développements). Cette notion a aussi été étudiée dans une série de papiers par divers auteurs (voir par exemple \cite{BRT07}, \cite{GoOd09} et \cite{GoTu24}).\\

La notion de structure de Poisson partielle dans le cadre de variétés de Banach a été définie par A.B. Tumpach dans \cite{Tum20}, Definition~3.7, désignée par l'expression 'generalized Poisson manifold' ainsi que dans \cite{BeOd19}, Definition~2.1.1 où elle est simplement nommée 'structure de Poisson'. Dans ces deux cas, le tenseur de Poisson n'est défini que sur un sous-fibré du fibré cotangent, sous-fibré séparant les points de l'espace tangent. \\
%Dans le cas d'une variété modelée sur un espace de Banach non réflexif, l'existence d'un champ hamiltonien $X_f$ requiert une condition supplémentaire sur le tenseur de Poisson $P$ associé au crochet.

L'extension de cette notion de structure de Poisson partielle au cadre localement convexe est due à K.-H. Neeb, H. Sahlmann et T. Thiemann dans le papier \cite{NST14}. Les auteurs considèrent une sous-algèbre $\mathcal{A}$ de $C^\infty(M)$ contenant les fonctions constantes et munie d'un crochet de Poisson. De plus, les différentielles des éléments de $\mathcal{A}$ séparent les vecteurs tangents.  Est aussi supposée l'existence de champs Hamiltoniens lisses. \\

Cette même notion de \emph{structure de Poisson partielle} dans le cadre des espaces adaptés ou $c^\infty$-complets au sens de Fr\"olicher, Kriegl et Michor est introduite dans \cite{PeCa19} où l'on a~:
\begin{description}
\item[$\bullet$]
l'algèbre $\mathfrak{A} (M)$ des fonctions lisses $f$ sur $M$ dont toutes les différentielles d'ordre supérieur $d^k_xf(.,u_2,\dots,u_k)$ ($k$ entier naturel non nul) appartiennent à $T_x^{\flat}M$ pour tout point $x$ de $M$, tout $(k-1)$-uplet $ \left( u_2,\dots, u_k \right) $ de  $T_xM$ où $T^\flat M$ est un sous-fibré faible du fibré cotangent cinématique $T'M$ de la variété $M$~;
\item[$\bullet$]
un morphisme de fibrés $P:T^{\flat}M \to TM$ tel que
\[
\{f,g\}_{P}=dg(P(df))
\]
définit un crochet de Poisson sur $\mathfrak{A} (M)$.
\end{description}
On s'intéresse alors dans cet article aux situations d'une variété adaptée 
\footnote{Ou $c^\infty$-complète.} 
 $M$ munie  
\begin{enumerate}
\item[$\bullet$]
soit de deux structures de Poisson  partielles compatibles $P$ et $Q$ définies sur le même sous-fibré $T^{\flat}M$, i.e. où $P+Q$ vérifie les mêmes conditions que $P$ et $Q$~;
\item[$\bullet$]
soit d'une structure de Poisson partielle $P$ définie sur $T^{\flat}M$ et d'un tenseur de Nijenhuis $N$ compatible avec $P$~;
\item[$\bullet$]
soit d'une structure de Poisson partielle $P$ définie sur $T^{\flat}M$ et une structure symplectique faible $\Omega$ à valeurs dans $T^{\flat}M$, ces deux structures vérifiant la condition de compatibilité $d (\Omega P \Omega ) = 0$.  
\end{enumerate}

\section{Cadre $c^\infty$-complet ou adapté}
\label{_CadreAdapte}

Le calcul différentiel en dimension infinie a été développé par Jacques Bernoulli, Leonhard Euler et Joseph-Louis Lagrange dans le cadre de la recherche de minimisation d'intégrales de fonctionnelles.\\

De nombreuses théories de la différentiation ont été proposées au delà du simple cadre d'espaces de Banach, principalement sur des certains espaces vectoriels topologiques localement convexes (e.v.t.l.c.). Apparaissent alors de nombreux problèmes~:
\begin{itemize}
\item[$\bullet$]
une équation différentielle ordinaire peut ne pas avoir de solutions et tant est que solution existe, on n'est pas assuré de son unicité~;
\item[$\bullet$]
il n'existe pas de théorème d'inversion locale naturel~;
\item[$\bullet$]
il n'existe pas de topologie naturelle sur le dual et, qui plus est, aucune de ces topologies n'est métrisable.
\end{itemize}

Ces diverses théories ont été développées de diverses manières en fonction du type de dérivée requise (cf. \cite{Mich38}, \cite{Bas64}, \cite{Kel74}, \cite{Lesl82}). On obtient alors divers calculs différentiels non équivalents sur des e.v.t.l.c. et par voie de conséquence sur des structures géométriques fondées sur de tels espaces (variétés, fibrés, groupes de Lie, algébroïdes).\\

Un point clé au sein de ces diverses approches est le cadre du calcul différentiel adapté 
\footnote{Ou calcul différentiel convenable selon la terminologie de \cite{Fau11}, eng. \textit{convenient differential calculus}.} 
introduit par A. Fr\"{o}licher et A. Kriegl
 (see \cite{FrKr88}).\\
\medskip
%!+ justification de l'utilisation du cadre adapté
Notre choix s'est alors porté sur ce type de calcul différentiel pour diverses raisons\footnote{On trouvera dans \cite{Schme21}, A.7.6 et A.7.7, une comparaison très intéressante de ces différents types de calcul différentiel.}~:
\begin{itemize}
\item[$\bullet$]
la catégorie des espaces vectoriels adaptés est cartésiennement fermée\footnote{cf. \cite{KrMi97}, Theorem~(3.12).}~:
%!++
\[
C^\infty(\mathbb{E} \times \mathbb{F},\mathbb{G}) \cong
C^\infty \left( \mathbb{E},C^\infty(\mathbb{F},\mathbb{G}) \right)
\]
\item[$\bullet$]
ce cadre englobe toutes les situations usuelles rencontrées (dimension finie, Banach, Hilbert et Fr\'echet)~;
\item[$\bullet$]
les limites directes de variétés de dimension finie voire de Hilbert ou de Banach peuvent être munies de structures adaptées ou $c^\infty$-complètes\footnote{Par exemple le groupe $GL(\infty, \mathbb{R})$ et tous ses sous groupes (cf. \cite{KrMi97}, 47) ou encore la limite directe d'espaces classifiants de fibr\'es principaux, (cf. \cite{Mor60}) peuvent être munis d'une structure de variété adaptée.}.   \\ 
\end{itemize}

\medskip
%!++
Notons par ailleurs que le cadre localement convexe général n'est pas toujours appropri\'e pour les problèmes de dimension infinie rencontrés en Physique et en mécanique des fluides (cf. \cite{SBZB17} et \cite{YoMa07}). En effet, le Lemme fondamental du calcul des variations qui est essentiel pour définir un lien entre les courbes extrémales d'une fonctionnelle et les équations d'Euler-Lagrange ne peut pas être utilisé dans ce cadre trop général. Ceci est dû au fait que  fonction évaluation $ev:\mathbb{E}\times \mathbb{E}^\ast\to\mathbb{R}$ entre un espace localement convexe $\mathbb{E}$ et son dual continu $\mathbb{E}^\ast$ n'est plus continue.
Par contre, pour la topologie $c^\infty$ d'espace adapté, l'application $ev:\mathbb{E}\times \mathbb{E}' \to\mathbb{R}$ entre un espace adapté $\mathbb{E}$ et  son dual bornologique $\mathbb{E}'$ est continue. On peut alors étendre les arguments usuels utilisés en dimension finie au cadre adapté pour obtenir un lien entre courbes singulières d'une fonctionnelle et solutions de l'équation d'Euler-Lagrange. \\
\smallskip

L'ouvrage de référence est ici le livre \cite{KrMi97}.
\begin{definition}
\label{D_cInfiniTopologie}
La \emph{$c^{\infty}$-topologie}\index{cinftopologie@$c^{\infty}$-topologie} sur un e.v.t.l.c. $\mathbb{E}$ est la topologie finale pour toutes les courbes infiniment différentiables (ou lisses)  $\mathbb{R}\to \mathbb{E}$. L'espace vectoriel $\mathbb{E}$ muni de cette topologie est noté $c^{\infty}\mathbb{E}$.\\
Un ouvert pour cette topologie est appelé \emph{$c^{\infty}$-ouvert}\index{cinfouvert@$c^{\infty}$-ouvert(ouvert pour la $c^{\infty}$-topologie)}.
\end{definition}

Autrement dit, la $c^{\infty}$-topologie est la plus fine topologie de  $\mathbb{E}$ pour laquelle  toutes les courbes infiniment différentiables demeurent continues. En général, la $c^{\infty}$-topologie est plus fine que toute topologie localement convexe ayant les mêmes ensembles bornés. $c^{\infty}\mathbb{E}$ n'est pas en général  un espace vectoriel topologique. Si $\mathbb{E}$ est un espace de Fr\'echet, alors la $c^{\infty}$-topologie coïncide avec la topologie donnée de l'e.v.t.\\
Le fait qu'une courbe soit lisse ne dépend pas de la topologie donnée sur $\mathbb{E}$ puisque toutes les topologies conduisant à un même dual donnent la même famille de courbes lisses. En fait, le fait qu'une courbe soit lisse dépend uniquement de la famille des parties bornées.

L'idée de J. Boman de tester le fait qu'une application soit lisse \textit{via} des courbes lisses de $\mathbb{R}$ vers $\mathbb{R}^n$ (cf. \cite{Bom67}) peut être étendue aux e.v.t.l.c., puisque cette notion \index{infiniment différentiable!courbe}\index{lisse!courbe} dans ce cadre est parfaitement définie.

\begin{definition}
\label{D_ApplicationLisseEntreEVTLC}
Soient $\mathbb{E}$ et $\mathbb{F}$ deux espaces vectoriels topologiques localement convexes. Une application $f:\mathbb{E} \to \mathbb{F}$ est dite \emph{convenablement lisse}\index{courbe!convenablement lisse} si l'image d'une courbe lisse est une courbe lisse, i.e. si $f\circ c\in C^{\infty} \left( \mathbb{R},\mathbb{F} \right)  $ pour toute courbe $c\in C^{\infty}\left ( \mathbb{R},\mathbb{E} \right) $.
\end{definition}

\begin{notation}
\label{N_DuauxAlgebriqueTopologiqueBornologique}
Si $\mathbb{E}$ est un espace vectoriel adapté, on note
\begin{description}
\item
$\mathbb{E}^\sharp$\index{Edual1@$\mathbb{E}^\sharp$ (dual algébrique)}\index{D1@dual!alg\'ebrique} le dual alg\'ebrique~;
\item
$\mathbb{E}^\star$\index{Edual2@$\mathbb{E}^\star$ (dual topologique)}\index{D2@dual!topologique}le dual topologique, i.e. l'espace des formes linéaires continues sur $E$~;
\item
$\mathbb{E}^\prime$\index{Edual3@$\mathbb{E}^\prime$ (dual bornologique)}\index{D3@dual!bornologique} le dual bornologique, i.e. l'espace des formes linéaires bornées sur $E$.
\end{description}
\end{notation}

\begin{definition}
\label{D_CourbeFaiblementLisse}
Soit $\mathbb{E}$ un espace adapté. Une courbe  $c : \mathbb{R} \to \mathbb{E}$ est dite \emph{faiblement lisse}\index{courbe!faiblement lisse} si pour toute forme linéaire continue $l$, $l \circ c$ est lisse~:
\[
\forall l \in E^\star,\; l \circ c \in C^\infty(\mathbb{R},\mathbb{R}).
\]
\end{definition}

On peut noter que le lien entre continuité et infinie différentiabilité en dimension infinie n'est pas aussi ténu qu'en dimension finie~: \textbf{il existe des courbes lisses qui ne sont pas continues pour la topologie donnée  de $\mathbb{E}$~!} Cette situation ne peut être évitée si l'on souhaite que la propriété de dérivation des fonctions composées 
\footnote{Eng. \textit{chain rule}.}
 soit vérifiée~: la fonction évaluation $ev:\mathbb{E}\times \mathbb{E}^{\star }\to\mathbb{R}$ de $\mathbb{E}$ est lisse mais n'est continue que si l'espace $\mathbb{E}$ est normable.

En analyse en dimension finie, la condition de Cauchy est une condition nécessaire de convergence d'une suite pour définir la complétude d'un espace. Dans le cadre adapté, c'est la notion plus faible de suite de Mackey-Cauchy qui est utilisée (cf. \cite{KrMi97},~2)~:
\begin{definition}
\label{D_SuiteDeMackeyCauchy}
Une suite $ \left( x_{n} \right) _{n \in \mathbb{N}}$ de $\mathbb{E}$ est dite de \emph{Mackey-Cauchy}\index{suite!Mackey-Cauchy}\index{Mackey-Cauchy!suite} s'il existe un sous-espace borné absolument convexe $B$ de $\mathbb{E}$ tel que $\left(  x_{n} \right)  $ est une suite de Cauchy dans l'espace vectoriel normé $\mathbb{E}_{B}$
%!+Définition de $\mathbb{E}_{B}$
\footnote{$\mathbb{E}_{B}$ est l'espace vectoriel engendré par $B$ dans $\mathbb{E}$ muni de la \emph{fonctionnelle de Minkowski}\index{fonctionnelle de Minkowski} $p_B$ définie pour tout $v$ de $\mathbb{E}$ par~:
\[
p_B(v)=\operatorname{inf} \{ \lambda > 0: v \in \lambda.B \}
\]}~:
\[
\forall \varepsilon > 0, 
\exists n_\varepsilon \in \mathbb{N},
\forall n>m>n_\varepsilon,\; x_m - x_n \in \varepsilon.B.
\]
\end{definition}

La définition suivante est fondamentale.

\begin{definition}
\label{D_EspaceVectorielAdapte}
Un espace vectoriel topologique localement convexe est dit  \emph{$c^{\infty}$-complet} ou \emph{adapté}\index{adapté!espace}\index{espace!adapté}\index{cinftycomplet@$c^{\infty}$-complet}\index{espace!$c^{\infty}$-complet} si toute suite de Mackey-Cauchy converge ($c^{\infty}$-complétude\index{cinftycompletude@$c^{\infty}$-complétude}\index{complétude}).
\end{definition}

La caractérisation suivante d'un espace adapté est tirée de \cite{KrMi97}, Theorem 2.14~:

\begin{proposition}
\label{P_CaracterisationEspaceAdapte}
Un espace vectoriel topologique localement convexe $\mathbb{E}$ est adapté si l'une des conditions équivalentes suivantes est satisfaite~:
\begin{description}
\item[\textbf{(EVA~1)}]
Toute courbe Lipschitzienne de $\mathbb{E}$ est localement Riemann intégrable.\index{courbe!localement Riemann intégrable}
\item[\textbf{(EVA~2)}]
Pour toute courbe $c\in C^{\infty} \left( \mathbb{R},\mathbb{E} \right) $, il existe une courbe $\gamma\in C^{\infty} \left( \mathbb{R},\mathbb{E} \right) $ telle que $\gamma^{\prime}=c$.
\item[\textbf{(EVA~3)}]
Toute courbe faiblement lisse est lisse.
\item[\textbf{(EVA~4)}]
Si $B$ est borné, fermé et absolument convexe, l'espace normé associé $\mathbb{E}_{B}$ est un espace de Banach.
%\item[\textbf{(EVA~5)}]
%Toute application linéaire continue d'un espace normé  dans $\mathbb{E}$ possède une extension continue au complété de l'espace normé.
\end{description}
\end{proposition}

Une application linéaire entre espaces adaptés est bornée
\footnote{
Soient $\mathbb{E}$ et $\mathbb{F}$ deux espaces vectoriels topologiques. Une application linéaire
 $f:\mathbb{E} \to \mathbb{F}$ est \emph{bornée}\index{bornée!application linéaire} ou bornologique\index{bornologique!application linéaire} si l'image d'une partie bornée de $\mathbb{E}$ est une partie bornée de  $\mathbb{F}$.
} si et seulement si elle est convenablement lisse.

\begin{notation}
\label{N_LinearMapsBetweenConvenientSpaces}
Si $\mathbb{E}_1$ et $\mathbb{E}_2$ sont des espaces adaptés, on note
\begin{description}
\item
$\operatorname{L} \left(  \mathbb{E}_{1},\mathbb{E}_{2} \right) $\index{LE1E2_1lineaire@$\operatorname{L} \left( \mathbb{E}_{1},\mathbb{E}_{2} \right) $ (espace des applications linéaires de $\mathbb{E}_1$ dans $\mathbb{E}_2$)} l'espace des applications linéaires de $\mathbb{E}_1$ dans $\mathbb{E}_2$;
\item
$L \left( \mathbb{E}_{1},\mathbb{E}_{2} \right) $\index{LE1E2_2borne@$L \left( \mathbb{E}_{1},\mathbb{E}_{2} \right) $ (espace des applications linéaires bornées de $\mathbb{E}_1$ dans $\mathbb{E}_2$)} l'espace des applications linéaires bornées de $\mathbb{E}_1$ dans $\mathbb{E}_2$.
\item
$\mathcal{L} \left( \mathbb{E}_{1},\mathbb{E}_{2} \right) $\index{LE1E2_3continue@$\mathcal{L} \left( \mathbb{E}_{1},\mathbb{E}_{2} \right) $ (espace des applications linéaires continues de $\mathbb{E}_1$ dans $\mathbb{E}_2$)} l'espace des applications linéaires continues de $\mathbb{E}_1$ dans $\mathbb{E}_2$.
\end{description}
\end{notation}
Si $\mathbb{E}_{1}=\mathbb{E}_{2}=\mathbb{E}$, on utilise respectivement les notations $\operatorname{L} \left( \mathbb{E} \right) $, $\mathcal{L} \left( \mathbb{E} \right) $ et $L \left(  \mathbb{E} \right) $.

\begin{example}
\label{Ex_R-infinity}
{\sf Espace des suites réelles finies.}\\ 
La somme directe \index{Rinfty@$\mathbb{R}^{\infty}$ (espace des suites réelles finies)}
$\mathbb{R}^{\infty}=\bigoplus\limits_{n\in\mathbb{N}^{\star }}\mathbb{R}$, noté aussi $\mathbb{R}^{ \left(  \mathbb{N}^\star \right) }$, est l'ensemble de toutes les suites de nombres réels nulles à partir d'un certain rang, espace qui peut être vu comme la réunion  $\bigcup\limits_{n\in\mathbb{N}^{\star }}\mathbb{R}^{n}$ où chaque $\mathbb{R}^{n}$ est identifié à l'hyperplan
$\mathbb{R}^{n}\times\left\{0\right\}$ de $\mathbb{R}^{n+1}$ \textit{via} l'inclusion $\iota_{n}^{n+1}:\mathbb{R}^{n}\to\mathbb{R}^{n+1}$. L'espace $\mathbb{R}^{\infty}$ est muni de la topologie limite directe, i.e. la plus fine topologie pour laquelle les inclusions $\iota_{n}:\mathbb{R}^{n}\to\mathbb{R}^{\infty}$ sont continues. Une partie $U$ de $\mathbb{R}^{\infty}$ est ouverte pour cette topologie si et seulement si 
$U\cap\mathbb{R}^{n}$ est ouvert dans $\mathbb{R}^{n}$ pour tout $n\in\mathbb{N}^{\star }$. \\
Munie de l'addition et la multiplication scalaire ponctuelles, $\mathbb{R}^{\infty}$ est un espace vectoriel localement convexe \textbf{non métrisable} (cf.
\cite{Spa14}, Example 3.10). \\
Les parties
\[
B_{0}^{\infty}\left(  \epsilon
_{n}\right)  =\left\{  \left(  x_{n}\right)  _{n\in\mathbb{N}^{\star }}\in\mathbb{R}^{\infty}:\left\vert x_{n}\right\vert <\epsilon_{n}\right\},
\]
où $ \left( \epsilon_{n} \right)  _{n\in\mathbb{N}^{\star }}$ est une suite de réels strictement positifs constitue une base de voisinages ouverts de $0$
(\cite{Jar81}, 4.1.4).\\
 $\mathbb{R}^{\infty}$ est un \emph{espace régulier}\index{régulier!espace}\index{espace!régulier}: toute partie borné $B$ de $\mathbb{R}^{\infty}$ est contenue dans un $\mathbb{R}^{n}$.\\
D'après \emph{\textbf{(EVA~4)}}, $\mathbb{R}^{\infty}$ est un espace vectoriel adapté.\\
$\mathbb{R}^{\infty}$ est aussi une limite directe stricte de plongements linéaires fermés des espaces de Banach $\mathbb{R}^{n}$.
\end{example}

\begin{theorem}
\label{T_ProprietesApplicationsDifferentiablesEntreEspacesAdaptes}
Soit $U$ un $c^{\infty}$-ouvert d'un espace vectoriel adapté $\mathbb{E}$ et soient $\mathbb{F}$ et $\mathbb{G}$ deux sous-espaces adaptés.
\begin{enumerate}
\item
L'espace $C^{\infty} \left( U,\mathbb{F} \right) $\index{CinftyUF@$C^{\infty} \left( U,\mathbb{F} \right) $} peut être muni d'une structure d'espace vectoriel adapté. Le sous-espace $L(\mathbb{E},\mathbb{F}) $ de toutes les applications bornées de $\mathbb{E}$ dans $\mathbb{F}$ est fermé dans $C^{\infty}(\mathbb{E},\mathbb{F}) $.
\item
La catégorie des espaces vectoriels adaptés est cartésiennement fermée, i.e. on a  l'identification naturelle~:
\[
C^{\infty}\left( \mathbb{E} \times \mathbb{F},\mathbb{G} \right) \simeq C^{\infty} \left( \mathbb{E},C^{\infty
} \left( \mathbb{F},\mathbb{G} \right) \right) .
\]
\item
L'opérateur différentiel 
\[
d :C^{\infty} \left( \mathbb{E},\mathbb{F} \right)  \to C^{\infty}\left( \mathbb{E},L\left(
\mathbb{E},\mathbb{F} \right) \right)
\]
\[
df(x)v =\underset{t\to0}{\lim}\dfrac{f\left(
x+tv\right)  -f\left(  x\right)  }{t}
\]
existe, est linéaire et lisse.
\item
La règle de la dérivation des fonctions composées est valide~:
\[
d \left( f \circ g \right)(x)  v
=
df \left( g(x) \right) dg(x) v.
\]
\end{enumerate}
\end{theorem}

Dans ce cadre adapté, les notions de variété, fibré et groupe de Lie se définissent de manière analogue au cadre de Banach (cf. \cite{KrMi97}, Chapitres VI et VIII).\\
Il est important de noter que, dans ce même cadre, toute dérivation de l'algèbre des fonctions de la variété n'est pas nécessairement un champ de vecteurs (cinématique) et que tout champ de vecteurs n'admet pas nécessairement un flot
\footnote{En fait, au delà du cadre des espaces de Banach, les résultats classiques sur l'existence et l'unicité de solutions d'équations différentielles, obtenus à partir de théorèmes résultant du théorème du point fixe, ne s'appliquent plus nécessairement.} (cf. \cite{KrMi97}, 32.12).

\section{Variétés de Poisson partielles}
\label{_VarietesPoissonPartielles}

Les structures de Poisson peuvent être définies à l'aide d'un morphisme antisymétrique  $P$ du fibré cotangent $T'M$ dans le fibré tangent $TM$. \\
Ce morphisme permet de définir un crochet $\{.,.\}$ sur l'ensemble $C^\infty(M)$ des fonctions lisses sur la variété, i.e. une application bilinéaire antisymétrique $C^\infty(M) \times C^\infty(M) \to C^\infty(M)$ qui satisfait à la fois les identités de Leibniz et de Jacobi.
\\

%!+
La notion de structure de Poisson partielle apparaît dès le cadre d'espaces de Banach non réflexifs où le morphisme $P$ peut n'être défini que sur un sous-fibré faible $T^\flat M$ de $T'M$ et le crochet de Poisson n'est alors défini que sur une sous-algèbre de  $C^\infty(M)$.

\subsection{Structures de Poisson linéaires partielles}
\label{___StructuresDePoissonLineairesPartielles}

Soit $\mathbb{E}$ un espace vectoriel adapté. On note $\mathbb{E}^\prime$ son dual bornologique et  $<.,.>$ le crochet de dualité. \\
Soit $\mathbb{E}^\flat$ un sous-espace vectoriel de $\mathbb{E}^\prime$ muni de sa propre structure d'espace adapté où l'inclusion $\mathbb{E}^\flat \to \mathbb{E}^\prime$ est bornée. 
%!+ 
Si $\mathbb{E}^\flat$ sépare les points de $\mathbb{E}$, le crochet de dualité $<.,.>$ entre $\mathbb{E}^\prime$ et $\mathbb{E}$ induit alors une forme bilinéaire bornée sur $\mathbb{E}^\flat\times\mathbb{E}$  qui est non dégénérée à gauche et qui est encore notée $<.,.>$ si aucune confusion n'est possible.\\

Un sous-espace $\mathbb{Y}$ d'un espace vectoriel adapté  $\mathbb{X}$ est dit \emph{fermé}\index{fermé} s'il est  $c^\infty$-fermé au sens de \cite{KrMi97}.\\

Plus généralement, on adopte systématiquement les notations et la terminogie utilisées dans \cite{KrMi97} et rappelées à la section \ref{_CadreAdapte}.

\begin{definition}
\label{D_StructureDePoissonLineairePartielle}
Une \emph{structure de Poisson linéaire partielle}\index{partielle!structure de Poisson linéaire}\index{structure!de Poisson linéaire partielle} sur $\mathbb{E}$ est la donnée de 
\begin{itemize}
\item[\emph{\textbf{(SPLp~1)}}] 
un sous-espace vectoriel $\mathbb{E}^\flat$ de $\mathbb{E}^\prime$  muni de sa propre structure d'espace adapté et tel que l'inclusion de $\mathbb{E}^\flat$ dans $\mathbb{E}^\prime$ est bornée~;
\item[\emph{\textbf{(SPLp~2)}}]
une application linéaire bornée 
$\mathtt{P}:\mathbb{E}^\flat\to \mathbb{E}$ 
telle que
\[
<\omega_1,\mathtt{P}\omega_2> = -<\omega_2,\mathtt{P}\omega_1>
\]
pour tous $\omega_1$ et $\omega_2$ de $\mathbb{E}^\flat$.\\
On dit alors que $ \left( \mathbb{E}^\flat,\mathbb{E},\mathtt{P} \right) $\index{Eflatmathbb@$ \left( \mathbb{E}^\flat,\mathbb{E},\mathtt{P} \right) $ (espace de Poisson linéaire partiel)} est un \emph{espace de Poisson linéaire partiel}\index{partiel!espace de Poisson linéaire}.
\end{itemize}
\end{definition}

Si $\mathtt{P}$ est un isomorphisme adapté, alors la $2$-forme $\Omega _\mathtt{P} \in L^2 \mathbb{E}^\prime$ définie par
\[
\Omega_\mathtt{P}(u,v)=<\mathtt{P}^{-1}(u), v>
\]
est une forme symplectique faible sur $\mathbb{E}$.\\ 
% Forme symplectique faible sur variété adaptée définie dans \cite{KrMi97}, 48.2
Quand  $\mathbb{E}^\flat=\mathbb{E}^\prime$, $\Omega _\mathtt{P}$ est une forme symplectique forte.\\
Si $\Omega$ est une forme symplectique faible sur $\mathbb{E}$  et si $\mathbb{E}^\flat$  est l'image de l'opérateur linéaire associé $\Omega^\flat : \mathbb{E} \to \mathbb{E}^\prime$ 
alors  
$\mathtt{P}= \left( \Omega^\flat \right) ^{-1}:\mathbb{E}^\flat \to \mathbb{E}$ 
est une structure de Poisson partielle linéaire qui est aussi un isomorphisme adapté. Ceci justifie la définition suivante~:
 
\begin{definition}
\label{D_StructureSymplectiqueLineairePartielle}  
Soit $\mathtt{P}:\mathbb{E}^\flat\to \mathbb{E}$ une structure de Poisson linéaire partielle.\\
Si $\mathtt{P}$ est un isomorphisme adapté,  on dit que $\mathtt{P}$ est une \emph{structure symplectique linéaire partielle}\index{partielle!structure symplectique linéaire} sur $\mathbb{E}$. Quand $\mathbb{E}^\flat=\mathbb{E}^\prime$ on dit simplement que $\mathtt{P}$ est \emph{symplectique linéaire}\index{symplectique linéaire}.
\end{definition}

\begin{remark}
\label{R_OP} 
Soit $\mathtt{P}:\mathbb{E}^\flat\to \mathbb{E}$ une structure de Poisson linéaire partielle telle que   $\ker \mathtt{P}$ est scindé et que $\mathtt{P}(\mathbb{E}^\flat)$ est fermé dans $\mathbb{E}$. Ceci implique que  $\mathtt{P}$  induit une application linéaire bijective  bornée $\widehat{\mathtt{P}}$ de $\mathbb{E}^\flat/\ker \mathtt{P}$ sur $\mathtt{P}(\mathbb{E})$ (cf. \cite{No05}).  Mais, {\bf dans le contexte adapté,  $\widehat{P}$ n'est pas un isomorphisme d'espaces adaptés}.\\
Dans la suite de cette remarque, {\bf on suppose   que  $\widehat{\mathtt{P}}$ est un isomorphisme d'espaces adaptés}. Notons que cette  situation se présente si $\mathbb{E}^\flat$  ou $\mathtt{P}(\mathbb{E})$ est un espace de Banach.  On a alors une décomposition $\mathbb{E}^\flat=\ker \mathtt{P} \oplus \mathbb{F}$ et si $q_{\mathbb{F}}$ est la  restriction de la projection associée $q$  de $\mathbb{E}^\flat$ sur  $\mathbb{F}$, alors  $q_\mathbb{F}$ est un  isomorphisme de $\mathbb{F}$ sur  $\mathbb{E}^\flat/\ker \mathtt{P}$  et par suite la restriction $\mathtt{P}_{\mathbb{F}}$ de $\mathtt{P}$ à $\mathbb{F}$ qui est égale à  $\hat{\mathtt{P}}\circ q_{\mathbb{F}}$ est un isomorphisme sur  $\mathtt{P}(\mathbb{E}^\flat)$. 
Puisque l'inclusion de $\mathbb{E}^\flat$ dans  $\mathbb{E}^{\prime}$ est bornée, il en est de même pour  l'inclusion de  $\mathbb{F}$ et le crochet de dualité  en restriction à $\mathbb{F}$ est une forme bilinéaire bornée sur $\mathbb{F} \times \mathbb{E}$.  
% and also on $\overline{\mathbb{E}^\flat}\times \mathbb{E}$ if  $\overline{\mathbb{E}^\flat}$  is the $c^\infty$ closure of $\mathbb{E}^\flat$ in $\mathbb{E}^{\prime}$. 
Par suite, on obtient une forme bilinéaire bornée sur  $\mathtt{P}(\mathbb{E})$ définie par
\[
\Omega_{\mathtt{P}}(u,v)
=<\beta,\mathtt{P}_{\mathbb{F}} \alpha> 
\; \textrm{ si }\; 
\mathtt{P}_\mathbb{F}(\alpha) =u,\;
 \mathtt{P}_\mathbb{F}(\beta)=v.
\]
D'une part,  $\Omega_P$ ne dépend pas du choix de  $\mathbb{F}$. D'autre part, l'espace vectoriel de Poisson partiel  $ \left( \mathbb{F},\mathbb{E}, \mathtt{P}_{\mathbb{F}} \right) $  donne lieu  à la même  structure linéaire de Poisson partielle  que  $ \left( \mathbb{E}^\flat,\mathbb{E},\mathtt{P} \right) $ dans le sens où,  pour chaque  $\alpha \in \mathbb{E}^\flat$ on a  $\mathtt{P}(\alpha)=\mathtt{P}_{\mathbb{F}}(\alpha_{\mathbb{F}})$ si $\alpha_{\mathbb{F}}=q_\mathbb{F}(\alpha)$.  Mais  $(\mathbb{F},\mathbb{E}, \mathtt{P}_{\mathbb{F}})$ est une structure symplectique partielle telle que $(\Omega_\mathtt{P} )^\flat=\mathtt{P}_{\mathbb{F}}^{-1}$ qui donne lieu à une forme symplectique faible $\Omega_\mathtt{P}$ sur  $\mathtt{P}(\mathbb{E})$.
\end{remark}

\begin{definition}
\label{D_Annulateur}
Pour tout sous-espace fermé $\mathbb{F}$ d'un espace vectoriel adapté $\mathbb{E}$,  \emph{l'annulateur}\index{annulateur}\index{Fannulateur@$\mathbb{F}^{\operatorname{ann}}$ (annulateur de $\mathbb{F}$)} de $\mathbb{F}$ est le sous espace  
\[
\mathbb{F}^{\operatorname{ann}}=
\{\alpha \in \mathbb{E}^\prime: \; 
\forall u \in \mathbb{F}, 
<\alpha,u>=0 \}.
\]
\end{definition} 
 
\begin{definition}
\label{D_Orthogonal_Isotropique} 
Soit $ \left( \mathbb{E}^\flat,\mathbb{E},\mathtt{P} \right) $ une structure  linéaire partielle.
\begin{enumerate}
\item 
Pour tout sous espace $\mathbb{A}$ de $\mathbb{E}^\flat$, l'espace vectoriel 
\[
\mathbb{A}^{\perp_{\mathtt{P}}}:=
\{\omega\in \mathbb{E}^\flat:\;
\forall \alpha \in \mathbb{A}\, \;
 <\omega, \mathtt{P} \alpha>=0\}
\] 
est appelé \emph{orthogonal de $\mathbb{A}$ relativement à $\mathtt{P}$}\index{AperpP@$\mathbb{A}^{\perp_{\mathtt{P}}}$ (orthogonal de $\mathbb{A}$ relativement à $\mathtt{P}$)}.\\
\item
L'espace  $\mathbb{A}$ est dit \emph{isotrope}\index{isotrope} si $\mathbb{A} \subset \mathbb{A}^{\mathtt{P}}$.
\end{enumerate}
\end{definition}
\begin{definition}
\label{D_Orthogonal_Coisotropic}
Soit $ \left( \mathbb{E}^\flat,\mathbb{E},\mathtt{P} \right) $ un espace de Poisson linéaire partiel.
\begin{enumerate}
\item  
Si $\mathbb{F}$ est un sous espace vectoriel adapté de $\mathbb{E}$ (i.e. si $\mathbb{F}$ est fermé), on pose\index{F0@$\mathbb{F}^0$}  
\[
\mathbb{F}^0:=\{\alpha \in \mathbb{E}^\flat:\;
\forall\; u\in \mathbb{F}, \; <\alpha,u>=0  \}.
\]
\emph{L'orthogonal}\index{Fperp@$\mathbb{F}^{\perp_{\mathtt{P}}}$} $\mathbb{F}^{\perp_{\mathtt{P}}}$ d'un sous espace adapté $\mathbb{F}$ de $\mathbb{E}$ est l'espace vectoriel
\[
\mathbb{F}^{\perp_{\mathtt{P}}} = \mathtt{P}(\mathbb{F}^0).
\]
\item
Le sous espace $\mathbb{F}$ est dit  \emph{co-isotrope}\index{co-isotrope} si $\mathbb{F}^0$ est isotrope dans $\mathbb{E}^\flat$.
\end{enumerate}
\end{definition}

En utilisant la terminolgie classique (cf. \cite{Wei88}), on introduit la notion essentielle de sous espace Lagrangien.

\begin{definition}
\label{D_SousEspaceLagrangien} 
Soit $ \left( \mathbb{E}^\flat,\mathbb{E},\mathtt{P} \right) $ un espace symplectique linéaire partiel.\\
Un sous espace adapté $\mathbb{F}$ de $\mathbb{E}$ est dit \emph{Lagrangien}\index{sous-espace Lagrangien}\index{Lagrangien!sous-espace} si 
$\mathtt{P} \left( \mathbb{F}^0 \right) = \mathbb{F}$.\\
On dit qu'un sous espace Lagrangien $\mathbb{F}$ est \emph{scindé} \index{sous espace Lagrangien!scindé}\index{scindé!sous-espace Lagrangien} s'il existe un supplémentaire Lagrangien de $\mathbb{F}$ dans $\mathbb{E}$.
\end{definition}
Notons que dans le cas d'une structure symplectique linéaire partielle, tout sous espace Lagrangien est aussi co-isotrope.\\
Cependant, en général, un sous espace Lagrangien n'admet pas nécessairement de supplémentaire.\\
% Even more, the reader will find in \cite{KaSw98} an example of linear symplectic Banach space in which any Lagrangian space is not supplemented.  \\
\begin{definition} 
%\label{D_LinearPoissonMap} 
\label{D_MorphismeDePoissonLineaire}
Soient $ \left( \mathbb{E}_1^{\flat}, \mathbb{E}_1, \mathtt{P}_1 \right)$ et $ \left( \mathbb{E}_2^{\flat}, \mathbb{E}_2, \mathtt{P}_2 \right)$ deux espaces de Poisson linéaires partiels\index{partiel!espace de Poisson linéaire}.\\
Une application linéaire bornée $\Phi:\mathbb{E}_1\to \mathbb{E}_2$ est appelée \emph{morphisme de Poisson (linéaire)}\index{morphisme!de Poisson (linéaire)}\index{Poisson!morphisme (linéaire)}  si l'adjoint\footnote{Rappelons que  l'adjoint d'un op\'erateur lin\'eaire born\'e  entre deux espaces vectoriels topologiques localement convexes est aussi born\'e (cf. \cite{PaSw94}).} $\Phi^{\star }$ de $\Phi$ vérifie
\begin{itemize}
\item[\emph{\textbf{(PMorphL~1)}}]
{\hfil $ \Phi^{\star }(\mathbb{E}_2^{\flat}) \subset \mathbb{E}_1^{\flat} $} 

\item[\emph{\textbf{(PMorphL~2)}}]
{\hfil
$\mathtt{P}_2 =	\Phi \circ \mathtt{P}_1 \circ \Phi^\star  $}
\end{itemize}
De plus, si chaque  $\mathtt{P}_i$ ($i \in \{1,2\}$) est symplectique partiel et est un morphisme de Poisson (linéaire),  $\Phi$ est alors appelé  \emph{morphisme symplectique partiel (linéaire)}\index{partiel!morphisme symplectique (linéaire)}.
\end{definition}

Nous introduisons la notion de polarisation d'un espace de Poisson linéaire (cf. \cite{AmAw04})~:

\begin{definition}
\label{D_PolarisationLinéaire} 
On appelle \emph{polarisation}\index{polarisation} d'un espace de Poisson linéaire partiel  $ \left( \mathbb{E}^{\flat}, \mathbb{E}, \mathtt{P} \right) $ la donnée d'un sous espace $\mathbb{E}^\tau$ fermé de $\mathbb{E}^\flat$ tel que $\mathtt{P}(\mathbb{E}^\tau \cap \mathbb{E}^\flat)$ est Lagrangien  dans $\mathtt{P}(\mathbb{E}^\flat)$.
\end{definition}
Notons que si  $\mathbb{E}^\tau$ est une polarisation de $ \left( \mathbb{E}^{\flat}, \mathbb{E}, \mathtt{P} \right)$ alors $\ker \mathtt{P}$ est contenu dans  $\mathbb{E}^\tau$.\\

Soient $ \left( \mathbb{E},\mathbb{E}^\flat, \mathtt{P} \right) $ une structure  symplectique partielle et $\Omega$ la forme symplectique associée. S'il existe un espace Lagrangien $\mathbb{L}$ de $\Omega$ alors $\mathbb{E}^\tau=\Omega^\sharp(\mathbb{L})$ est une polarisation de l'espace linéaire de Poisson partiel $ \left( \mathbb{E}^\flat,\mathbb{E}, \mathtt{P} \right) $. Cette situation se produit toujours si $\mathbb{E}$ est hilbertizable. \\
Plus généralement, si $\mathtt{P}:\mathbb{E}^\flat\to \mathbb{E}$ est une structure de Poisson linéaire partielle telle que  $\ker \mathtt{P}$ est scindé et que $\mathtt{P}(\mathbb{E}^\flat)$ est fermé dans $\mathbb{E}$ alors avec les hypothèses et notations  de la Remarque \ref{R_OP}, l'espace $\mathtt{P}(\mathbb{E}^\flat)$ peut être muni d'une  forme symplectique  $\Omega_\mathtt{P}$. Supposons qu'il existe un espace Lagrangien $\mathbb{L}$ pour $\Omega_{\mathtt{P}}$ alors $\mathbb{E}^\tau=P^{-1}(\mathbb{L})$ est une polarisation de $ \left( \mathbb{E}^\flat, \mathbb{E}, \mathtt{P} \right) $. 
  
\subsection{Variétés de Poisson partielles}
\label{___VarietesPoissonPartielles}

Soit $M$ une variété adaptée modelée sur l'espace vectoriel adapté $\mathbb{M}$.
On note $:p_{M}:TM\to M$ son fibré tangent cinématique (\cite{KrMi97}, 28.12) et $p_{M}^{\prime}:T^{\prime}M\to M$ son fibré cotangent cinématique (\cite{KrMi97}, 33.1).

\subsubsection{L'algèbre $\mathfrak{A} (U)$}
\label{___LAlgebremathfrakAU}

\begin{definition}
\label{D_SousFibreFaible}
Un sous-fibré $p^{\flat}:T^{\flat}M\to M$ de $p_{M}^{\prime}:T^{\prime}M \to M$ où $p^{\flat}:T^{\flat}M\to M$ est un fibré adapté, est un \emph{sous-fibré faible} de \index{faible!sous-fibré} $p_{M}^{\prime}:T^{\prime}M\to M$ si l'injection canonique $\iota:T^{\flat}M \to T^{\prime}M$ est un morphisme de fibrés adaptés.
\end{definition}

En se référant à \cite {KrMi97}, Définition~48.5, on introduit l'ensemble suivant.
\begin{definition}
\label{D_AU}
Pour tout ouvert $U$ de $M$, on considère l'ensemble $\mathfrak{A}(U)$\index{AmathfrakU@$\mathfrak{A}(U)$} des fonctions $f\in C^\infty(U)$ telle que, pour tout entier naturel non nul $k$ et tout $x$ de $U$, la dérivée d'ordre $k$ de $f$ en $x$, $d^{k}f(x)\in L_{\operatorname{sym}}^{k}(T_{x}M,\mathbb{R})$  vérifie~:
\begin{equation}
\label{eq_dkf}
\forall (u_2,\dots,u_k) \in (T_xM)^{k-1},\;
d^{k}_xf(.,u_{2},\dots,u_{k}) \in T_{x}^{\flat}M.
\end{equation}
\end{definition}

\begin{remark}
\label{R_AUEndomorphime}
Soit $f\in\mathfrak{A} (U)$, alors pour tout entier naturel $k$, tout $x\in U$et tout $(k-1)$-uplet $ \left( u_2,\dots,u_k \right) $ de $T_xM$, si $A$ est un endomorphisme de $T_xM$ alors l'application linéaire $u\mapsto d_x^kf(A(u),u_2,\dots,u_k)$ appartient à $T_x^\flat M$.
\end{remark}
%!+
On a alors le résultat suivant correspondant à  \cite{CaPe23}, Proposition~7.1~: 

\begin{proposition}
\label{P_AUalgebra}
Soit un ouvert donné $U$ de $M$.
\begin{enumerate}
\item
L'ensemble $\mathfrak{A} (U)$ est une sous-algèbre de $C^\infty(U)$.
%which is closed convenient space  of  the convenient space $C^\infty(U)$.
\item
Pour tout entier naturel $k$ et tous champs de vecteurs locaux $X_1,\dots ,X_k$ au-dessus de $U$, 
l'application $x\mapsto  d^kf(X_1,\dots,X_k)(x)$ appartient à $\mathfrak{A} (U)$.
\end{enumerate}
\end{proposition}

\subsubsection{Structures de Poisson partielles}
\label{____StructuresDePoissonPartielles}

On notera $<\;,\;>$ le crochet de dualité canonique entre  les fibrés $T^{\prime}M$ et $TM$.

\begin{definition}
\label{D_MorphismeAntisymetrique}
Un morphisme $P:T^{\flat}M \to TM$ est dit antisymétrique\index{morphisme!antisymétrique} s'il vérifie, pour tous $\xi$ et $\eta$ de $T_x^{\flat}M$, la relation 
\begin{equation}
\label{eq_Antisymetrie}
<\xi,P(\eta)>=-<\eta,P(\xi)>.
\end{equation}
On dit alors que $P$ est une \emph{quasi ancre de Poisson}\index{quasi ancre de Poisson}.
\end{definition}

Étant donné un tel morphisme $P$, on définit sur $\mathfrak{A} (U)$ le crochet $\{.,.\}_P$ par~:
\begin{equation}
\label{eq_CrochetDual}
\{f,g\}_P=-<df,P(dg)>.
\end{equation}
Dans ces conditions, la relation (\ref{eq_CrochetDual}) définit une application bilinéaire antisymétrique $\{.,.\}_P:\mathfrak{A} (U)\times\mathfrak{A} (U)\to
C^\infty (U)$.

\begin{proposition}
\label{P_PoissonLeibniz}
Le crochet $\{.,.\}_{P}$ est à valeurs dans 
$\mathfrak{A} (U)$ et vérifie la propriété de Leibniz~:
\[
\{f,gh\}_{P}=g\{f,h\}_{P}+h\{f,g\}_{P} \label{eq_ProprieteDeLeibnizSurA(M)}
\]
\end{proposition}

\begin{proof}
Puisque, pour tous $f$ et $g$ de $\mathfrak{A} (U)$, on a 
\[
\{f,g\}_P=-df(P(dg))
\] 
d'après la Proposition \ref{P_AUalgebra}, 2., il s'ensuit que $df(P(dg))$ appartient à $\mathfrak{A} (U)$.\\
Finalement, d'après la définition du crochet, on a
\[
\begin{array}{rcl}
\{f,gh\}_P	&=&	-<df, Pd(gh)> 				\\
			&=&	-g<df,P(dh)>-h<df,P(dg)>		\\
			&=&g\{f,h\}_P+h\{f,g\}_P
\end{array}
\]
ce qui établit la propriété de Leibniz.
\end{proof}

En se référant à \cite{KrMi97}, 29.5, on peut associer à deux fibrés vectoriels adaptés $p_E : E \to M$ et $p_F : F \to M$ d'autres fibrés vectoriels adaptés sur cette même base $M$, e.g. $E \oplus F$ et $L(E,F)$.\\
D'autre part, si $ \pi : E \to M$ est un fibré vectoriel adapté, on peut munir l'ensemble $\Gamma(E)$ des sections lisses de ce fibré d'une structure d'espace vectoriel adapté (cf. \cite{KrMi97}, 30).

\begin{notation}
\label{N_UnderlineP_Psharp}
A toute quasi ancre de Poisson $P : T^\flat M \to TM$, on associe
\begin{itemize}
\label{eq_PUnderscore}
\item[]{\hfil$ \underline{P}:\Gamma \left( T^\flat M \right)
\to \mathfrak{X}(M)$}
\item[]
\index{Ptilde@$\tilde{P}$}
ainsi que 
\[
\begin{array}{rccc}
\tilde{P}:	& T^\flat M \times_M T^\flat M &
			\to		&\mathbb{R}		\\
			& (\alpha,\beta)	&
			\mapsto	& <\beta,P \alpha>
\end{array}
\]
\item[]
et aussi le tenseur antisymétique de type $(2,0)$ (i.e. $2$-fois contravariant)\index{PsharpUnderline@$\underline{\tilde{P}}$ (tenseur $(2,0)$ associé à $P$)}
\[
\underline{\tilde{P}}:	
\Gamma \left( T^\flat M \times_M T^\flat M \right) \to C^\infty(M)
\] 
défini pour toutes sections $\eta$ et $\xi$ de $T^\flat M$ et pour tout $x$ de $M$ par
\[
\left( \underline{\tilde{P}}(\eta,\xi) \right) (x)
=
\left\langle \xi(x),P_x(\eta(x) \right\rangle
\]
\end{itemize}
\end{notation}
On pourra s'autoriser quelques abus de notation en écrivant $P$ en lieu et place de $\underline{P}$, $\tilde{P}$ ou $\underline{\tilde{P}}$ lorsqu'il n'y a pas d'ambiguïté  compte tenu du contexte. 
\begin{definition}
\label{D_PartialPoissonStructure}
Soit $p^{\flat}:T^{\flat}M\to M$ un sous-fibré faible de
$p_{M}^{\prime}:T^{\prime}M \to M$ et $P:T^{\flat}M \to TM$ une quasi ancre de Poisson.
\begin{enumerate}
\item
On dit que $ \left( T^\flat M,M,P,\{.,.\}_{P} \right) $\index{crochetPoissonpartiel@$\{.,.\}_{P}$ (crochet de Poisson partiel)}\index{partiel!crochet de Poisson} est une \emph{structure de Poisson partielle}\index{structure!de Poisson partielle} sur $M$  si le crochet $\{.,.\}_{P}$ vérifie l'identité de Jacobi\index{Jacobi!identité}\index{identité!de Jacobi}
\[
\{f,\{g,h\}_{P}\}_{P}+\{g,\{h,f\}_{P}\}_{P}+\{h,\{f,g\}_{P}\}_{P}=0.
\label{eq_IdentiteDeJacobiPartielPoisson}
\]
Dans ce cas, $P$ est appelée \emph{ancre de Poisson partielle}\index{ancre de Poisson partielle}.
\item
Soit $\mathfrak{A} $ une sous-algèbre de $\mathfrak{A} (M)$ telle que la restriction de $\{.,.\}_P$ à $\mathfrak{A} \times\mathfrak{A} $ est à valeurs dans $\mathfrak{A} $. On dit que $ \left( M,\mathfrak{A} ,\{.,.\}_{P} \right) $ est une \emph{variété de Poisson partielle}\index{variété!de Poisson partielle}. Si, de plus, $P$ est un isomorphisme, on dit que $ \left( M,\mathfrak{A} ,\{.,.\}_{P} \right) $ est une \emph{variété symplectique partielle}\index{variété! symplectique  partielle}.
\end{enumerate}
\end{definition}

On peut noter en particulier que si $ \left( T^\flat M,M,P,\{.,.\}_{P} \right) $ est une structure de Poisson partielle, alors $ \left( M,\mathfrak{A} (M),\{.,.\}_{P} \right) $  est toujours une variété de Poisson partielle.

\begin{notation}
$\mathfrak{P}_M$\index{Pmathfrak_M@$\mathfrak{P}_M$} désigne le faisceau des $\mathfrak{A} (U)$-modules engendré par l'ensemble $\{df,f\in\mathfrak{A} (U)\}$.
\end{notation}

Soit $ \left( T^\flat M,M,P,\{.,.\}_{P} \right) $ une structure de Poisson partielle. Comme en dimension finie, on associe au morphisme $P$ un tenseur antisymétrique $\underline{\tilde{P}}$ de type $(2,0)$ sur $T^\flat M$   (cf. Notation \ref{N_UnderlineP_Psharp}). On a alors pour toutes sections $\eta$ et $\xi$ de $T^\flat M$ et pour tout $x$ de $M$ par
\begin{eqnarray}
\label{eq_Lambda}
\left( \underline{\tilde{P}}(\eta,\xi) \right) (x)
=
\left\langle \xi(x),P_x(\eta(x) \right\rangle.
\end{eqnarray} 
%!+ $\mathfrak{P}(U)$
Pour tout ouvert $U$ de $M$, l'espace $\mathfrak{P}(U)$ est engendré par les différentielles de fonctions de $\mathfrak{A} (U)$. Pour tous éléments $\sigma_1$ et $\sigma_2$ de $\mathfrak{P}(U)$, on définit le crochet $[.,.]_P$ par
\[
\left[\sigma_1,\sigma_2\right]_P
=
L_{\underline{P}\sigma_1}\sigma_2
-L_{\underline{P}\sigma_2}\sigma_1
-d<\sigma_1,\underline{P}\sigma_2>
\]
L'application
 $[\underline{P},\underline{P}]:\mathfrak{P}(U)\times \mathfrak{P}(U)\to \mathfrak{X}(U)$
%définie par
\begin{equation}
\label{eq_PP}
[\underline{P},\underline{P}](\sigma_1,\sigma_2)
=
\underline{P}([\sigma_1,\sigma_2]_P)
-[\underline{P}\sigma_1,\underline{P}\sigma_2]
\end{equation}
est alors parfaitement définie~; elle est bilinéaire et sa valeur en  $T^\flat_x M$ dépend uniquement des valeurs de $\sigma_1$ et $\sigma_2$ en $x$. Ainsi $[\underline{P},\underline{P}]$ est un tenseur de type $(2,1)$ à valeurs dans $\mathfrak{X}(M)$.

On peut alors considérer le tenseur noté $[\underline{\tilde{P}},\underline{\tilde{P}}]$ de type $(3,0)$ sur $T^\flat M$ défini, pour tout triplet $ \left( \sigma_1,\sigma_2,\sigma_3 \right) $  d'éléments de $\mathfrak{P}(U)$,  par~:  
\begin{eqnarray}
\label{eq_Lambda}
[\underline{\tilde{P}},\underline{\tilde{P}}](\sigma_1,\sigma_2,\sigma_3)
=
<\sigma_3,[\underline{P},\underline{P}](\sigma_2,\sigma_1)>.
\end{eqnarray}
En adaptant les résultats de \cite{MaMo84}, Appendix~B.1 à notre contexte (cf. \cite{CaPe23}, Chaptitre~7, Lemme~7.27), on a les propriétés suivantes~:
\begin{enumerate}
\item
pour tout triplet $ \left( \sigma_1,\sigma_2,\sigma_3 \right) $ d'éléments de $\mathfrak{P}(U)$,
\[
[\underline{\tilde{P}},\underline{\tilde{P}}](\sigma_1,\sigma_2,\sigma_3)
=<L_{\underline{P}\sigma_1}\sigma_2,\underline{P}\sigma_3>
+<L_{\underline{P}\sigma_2}\sigma_3,\underline{P}\sigma_1>
+<L_{\underline{P}\sigma_3}\sigma_1,\underline{P}\sigma_2>
\]
\item
Si $\sigma_i=df_i$ où $f_i\in\mathfrak{A} (U)$, $i \in \left\{ 1,2,3 \right\} $,
\begin{equation}\label{eq_Schouten}
[\underline{\tilde{P}},\underline{\tilde{P}}](df_1,df_2,df_3)
=\{f_1,\{f_2,f_3\}_P\}_P+\{f_2,\{f_3,f_1\}_P\}_P+\{f_3,\{f_1,f_2\}_P\}_P.
\end{equation}
\end{enumerate}

Le tenseur $2[\underline{\tilde{P}},\underline{\tilde{P}}]$ est appelé \emph{crochet de Schouten}\index{crochet!de Schouten}\index{Schouten!crochet} de $\underline{\tilde{P}}$ (cf. \cite{CaPe23}, Chapitre~7, Definitions~7.28 et 7.50).\\

On a alors le résultat suivant correspondant à \cite{CaPe23}, Chapitre~7, Theorem~7.29:
\begin{theorem}
\label{T_CaracterisationStructurePartielleDePoisson} 
Soit $P:T^\flat M\to TM$ une quasi ancre de Poisson et $\underline{\tilde{P}}$ le tenseur antisymétrique associé. \\ 
$P$ définit une structure de Poisson partielle sur $M$ si et seulement si le crochet de Schouten de $\underline{\tilde{P}}$ est identiquement nul.
\end{theorem}

De manière classique, si l'on se donne une variété de Poisson partielle $ \left( M,\mathfrak{A} ,\{.,.\}_{P} \right) $, toute fonction $f$ de $\mathfrak{A} $ est appelée \emph{Hamiltonien}\index{Hamiltonien} et,  
%!+ 
si les différentielles des fonctions de $\mathfrak{A}(U)$ séparent les points dans $TM$ le champ de vecteurs associé $X_{f}=P(df)$ est appelé \emph{champ de vecteurs Hamiltonien}\index{champ de vecteurs Hamiltonien}.\\
On a alors $\{f,g\}=X_{f}(g)$ ainsi que 
\begin{eqnarray}
\label{eq_CrochetChampsHamiltoniens}
\left[ X_f,X_g \right] = X_{\{f,g\}}
\end{eqnarray}
d'après \cite{NST14}, ce qui est équivalent à
\begin{equation}
\label{eq_Pdfdg}
P(d\{f,g\})=[P(df),P(dg)]. 
\end{equation}

\subsubsection{Morphismes et applications de Poisson}
\label{____MorphismesEtApplicationsDePoisson}

\begin{definition}
\label{D_MorphismeDePoisson} 
Soient deux structures de Poisson partielles $(T^\flat M_i,TM_i, P_i,\{\;,\;\}_{P_i})$  pour $i \in \{1,2\}$.  
%where $T^\flat M_1$ is a closed in $T^\prime M_1$. 
Une application lisse $\phi:M_1 \to M_2$  est appelée \emph{morphisme de Poisson} si, pour tout $x$ de $M_1$, on a~:
\begin{itemize}
\item[\emph{\textbf{(PMorph~1)}}]
\hfil{}$\left( T_{\phi(x)}\phi \right) ^{\star } (T^\flat _{\phi(x)}M_1)\subset T^\flat_x M_1$\\ 
où  $\left( T_{\phi(x)}\phi \right) ^{\star }$  désigne l'adjoint de l'application tangente $T_x\phi$;
\item[\emph{\textbf{(PMorph~2)}}] 
$T_x \phi : T_x M_1 \to T_{\phi(x)} M_2$ est un morphisme de Poisson (linéaire) de $ \left( T^\flat_x M_1,T_x M_1, \left( P_1 \right) _x \right) $ dans $ \left( T_{\phi(x)} M_2,T_{\phi(x)} M_2, \left( P_2 \right) _{\phi(x)} \right) $.
\end{itemize}   
\end{definition}
Notons que l'application tangente $T_x \phi : T_x M_1 \to T_{\phi(x)} M_2$ est un morphisme de Poisson (linéaire) signifie précisément que, pour tout $x$ de $M_1$, on retrouve la relation \textbf{(PMorphL~2)} adapté à notre contexte
\begin{eqnarray}
\label{eq_PMorphL2_EntreFibres}
(P_2)_{\phi(x)}
=
T_x\phi\circ (P_1)_x\circ (T_x\phi )^{\star }.
\end{eqnarray}
\begin{definition}
\label{D_ApplicationDePoisson} 
Pour $i \in \{1,2\}$, soit $\{(\mathcal{E}_i (U_i), \{\;,\;\}_{U_i}),\;  \textrm{ pour tout ouvert } U_i  \textrm{ de } M_i\}$ un faisceau d'algèbres de Lie-Poisson sur $M_i$.\\
Une application $\phi: M_1\to M_2$ est une \emph{application de Poisson}\index{application de Poisson} si, pour tout ouvert $U_i$ de $M_i$ tel que $\phi(U_1)\subset U_2$, l'application induite $\phi^{\star }:\mathcal{C}^{\infty}(U_2)
\to \mathcal{C}^{\infty}(U_1) $,  définie par $\phi^{\star }(f):=f\circ \phi$, est telle que
\begin{itemize}
\item[\emph{\textbf{(PApp~1)}}]
$\hfil{}
\phi^{\star } \left( \mathcal{E}(U_2) \right) 
\subset \mathcal{E}(U_1)$
\item[\emph{\textbf{(PApp~2)}}]
$\hfil{}
\forall (f,g) \in  \mathcal{E}_2(U_2)^2,\;\{\phi^{\star }(f),\phi^{\star }(g)\}_{P_1}
=
\phi^{\star }(\{f,g\}_{P_2}).
$
\end{itemize}
\end{definition}

On obtient alors le résultat suivant établi dans le chapitre 7 de \cite{CaPe23}~:
\begin{theorem} 
\label{T_LienEntreMorphismesDePoissonEtApplicationsDePoisson} 
Soient deux structures de Poisson partielles $ \left( T^\flat M_i,TM_i, P_i,\{\;,\;\}_{P_i} \right) $,  pour $i \in \{1,2\}$, telles que  $T^\flat M_i$ est fermé dans $T^\prime M_i$ et soit $\phi:M_1\to M_2$ une application lisse.
\begin{enumerate}
\item 
Si $\phi$ est un morphisme de Poisson, alors $\phi$ est aussi une application de Poisson.
\item 
Supposons que l'on ait 
\begin{center}
$\forall y\in  \phi(M_1),\;
 \left( T_y M_2 \right) ^{\operatorname{ann}}  \cap T^\flat_y M_2=\{0\}$
\end{center}
où $(T_y M_2)^{\operatorname{ann}}$ est l'annulateur de  $T_y M_2$.\\
Si $\phi$ est une application de Poisson, alors $\phi$ est aussi un morphisme de Poisson.
\end{enumerate}
\end{theorem}

\subsubsection{Exemples de variétés de Poisson partielles}
\label{___ExemplesVarietesDePoissonPartielles}

\begin{example}
\label{Ex_VarieteDePoissonDeDimensionFinie}
{\sf Variété de Poisson de dimension finie.}\\ 
Une variété de Poisson de dimension finie $ \left( M,\mathcal{C}^{\infty}(M),\{.,.\} \right) $ (cf. \cite{Marl83}) est un cas particulier de variété de Poisson partielle. En effet, à chaque fonction $f$ sur $M$ est associé un champ de vecteurs Hamiltonien\index{champ de vecteurs Hamiltonien} $X_{f}$. Puisque $T^{\flat}M=T^\star M$ est localement engendré par les différentielles de fonctions, on peut étendre l'application $df \mapsto X_{f}$ à un unique morphisme de fibrés antisymétrique $P:T^{\star}M\to TM$ tel que $P(df)=X_{f}$.
\end{example}

\begin{example}
\label{Ex_VarieteDeBanachPoisson}
{\sf Variétés de Banach-Poisson.}\\ 
Soit $M$ une variété de Banach. La notion de variété de Banach-Poisson a été définie et développée dans les articles \cite{OdRa03} et \cite{Rat11} où ces auteurs supposent qu'il existe un crochet de Poisson $\{.,.\}$ sur $C^{\infty}(M)$  tel qu'à chaque fonctionnelle linéaire $\xi$ sur $M$ est associée une section $\xi^{\prime}$ du bidual $T^{\prime\prime}M$ qui appartient en fait à $TM\subset T^{\prime\prime}M$.\\ 
%erreur ! Ceci génère un morphisme antisymétrique $P:T^{\prime}M\to TM$. 
%Si  l'on se donne un morphisme antisymétrique $P:T^{\prime}M\to TM$, on obtient un crochet $\{.,.\}$ sur $C^{\infty}(M)$ comme donné par la relation (\ref{eq_CrochetDual}). Ainsi, si ce crochet satisfait l'identité de Jacobi, on obtient la notion de variété de Banach-Lie Poisson (cf. \cite{Pel12}).  $M$ est alors munie d'une structure de Poisson partielle.
Notons que la donnée d'un crochet de Poisson ne génère pas nécessairement une application linéaire du dual dans le bidual à cause de l'existence de \textit{queer operational tangent vectors} (cf. \cite{KrMi97}) et de \textit{queer Poisson brackets} (cf. \cite{BGT18}).
\end{example}

\begin{example}
\label{Ex_VarietesDeBanachSymplectiquesFaibles}
{\sf Variétés de Banach symplectiques faibles.}
% cf. \cite{OdRa08}.\\
Une \emph{variété symplectique faible}\index{variété!symplectique faible}\index{symplectique faible!variété} est une variété adaptée $M$ munie d'une $2$-forme fermée ${\omega}$ telle que le morphisme associé
\[
\begin{array}
[c]{cccc}
{\omega}^{\flat}:   & TM    & \to       & T^{\prime}M\\
                    & X     & \mapsto   & {\omega}(X,.)
\end{array}
\]
%!+ oubli de injectif
est injectif.
\end{example}

\begin{example}
\label{Ex_VarieteSymplectiqueFaible_linfty-times-l1-omega}
{\sf La variété symplectique faible $ \left( l^\infty \times l^1,\omega \right)$.}\\
Cet exemple s'inscrit dans le cadre des variétés de Banach faiblement symplectiques comme définies dans l'exemple \ref{Ex_VarietesDeBanachSymplectiquesFaibles} emprunté à \cite{OdRa08a}.\\
Puisque l'application $\omega^\flat_x$ est seulement injective, le crochet de Poisson  $\{f,g\}$ ne peut pas être défini pour tout couple $(f,g)$ d'éléments de $C^\infty(M)$. Afin de pouvoir définir le champ de vecteurs Hamiltonien associé à une fonction $f$ par $\iota_{X_f}\omega = df$, il est nécessaire que pour chaque $x$ de $M$, $df(x)$ appartienne à $\omega^\flat(T_x M)$.
\[
\mathfrak{A}  = \left\lbrace f \in C^\infty(M):\;
\forall x \in M,\;df(x) \in \omega^\flat(T_x M) \right\rbrace
\]
est alors une algèbre.\\
Si $(f,g) \in\mathfrak{A} ^2$, les champs de vecteurs Hamiltoniens  $X_f$ et $X_g$ existent et l'on peut alors définir un crochet de Poisson par 
\[
\{f,g\}_\omega = \omega  \left( X_f,X_g \right) . 
\] 
$\mathfrak{A} $ est alors une \emph{algèbre de Poisson}\index{Poisson!algèbre}, i.e. une algèbre relativement à la multiplication des fonctions et une algèbre de Lie relativement au crochet de Poisson, crochet vérifiant l'identité de Leibniz.\\
On considère l'espace de Banach $\ell^\infty \times \ell^1$\index{linftyxl1@$\ell^\infty \times \ell^1$} où
\[
\ell^\infty = \left\lbrace \mathbf{q}= \left( q_n \right) _{n \in \mathbb{N}}: \; \| \mathbf{q} \| _{\infty} = \displaystyle\sup_{n\in \mathbb{N}} |x_n| < +\infty \right\rbrace
\]
est l'espace de Banach des suites réelles bornées et où
\[
\ell^1 = \left\lbrace \mathbf{p}= \left( p_n \right) _{n \in \mathbb{N}}: \; \| \mathbf{p} \|_1 = \displaystyle\sum_{n=0}^{+\infty} |x_n| < +\infty \right\rbrace
\] 
est l'espace de Banach des suites réelles absolument convergentes.\\
Le crochet de dualité fortement non dégénéré 
\[
\forall (\mathbf{q},\mathbf{p}) \in \ell^\infty \times \ell^1,\;
\left\langle \mathbf{q},\mathbf{p} \right\rangle
= \displaystyle\sum_{n=0}^{+\infty} q_n p_n
\] 
correspond à l'isomorphisme d'espaces de Banach $ \left(  \ell^1 \right) ^\star = \ell^\infty$.\\
La forme symplectique faible $\omega$ est la forme canonique donnée par  
\[
\forall (\mathbf{q},\mathbf{q'},\mathbf{p},\mathbf{p'})
\in \left( \ell^\infty \right) ^2 
\times
\left( \ell^1 \right) ^2 ,\;
\omega  \left( (\mathbf{q},\mathbf{p}),(\mathbf{q'},\mathbf{p'}) \right)
= \left\langle \mathbf{q},\mathbf{p'} \right\rangle 
- \left\langle \mathbf{q'},\mathbf{p} \right\rangle
\]
On a alors
\[
\mathfrak{A}  = \left\lbrace f \in C^\infty  \left( \ell^\infty \times \ell^1 \right) :\;  \left( \dfrac{\partial f}{\partial q_n} \right) _{n \in \mathbb{N}} \in \ell^1 \right\rbrace
\]
et le crochet de Poisson est donné par
\[
\{f,g\}_\omega = \displaystyle\sum_{n=0}^{+\infty}
 \left( \dfrac{\partial f}{\partial q_k} \dfrac{\partial g}{\partial p_k}
 - \dfrac{\partial g}{\partial q_k} \dfrac{\partial f}{\partial f_k} \right) . 
\]
Ce crochet vérifie bien l'identité de Leibniz et est associé à un tenseur $2$ fois contravariant antisymétrique.
\end{example}

\begin{example}
\label{Ex_AlgebroideDeBanachLiePoisson}
{\sf Algébroïde de Banach Lie et structure de Poisson partielle.}\\
Un algébroïde de Banach-Lie est un fibré vectoriel de Banach $\pi:E\to M$ muni d'un morphisme $\rho:E\to TM$, appelé \emph{ancre}\index{ancre},
et d'un crochet de Lie $[.,.]_{E}$ qui est une application bilinéaire antisymétrique
$\Gamma(E)\times\Gamma(E)\to\Gamma(E)$ telle que
\[
[X,fY]_{E}=df(\rho(X))Y+f[X,Y]_{E}
\]
où $\Gamma(E)$ désigne l'ensemble des sections de $\pi:E\to M$ qui satisfont l'identité de Jacobi (cf. \cite{Ana11} et \cite{CaPe12}). \\
Si l'on note $\pi^\star:E^\star \to M$ le fibré dual du fibré $\pi:E\to M$, il existe un sous-fibré de Banach $T^{\flat}E^\star $ de $T^{\prime}E^\star $ défini comme suit.\\
Pour toute section (locale) $s: U\to E_{| U}$, on note $\Phi_s$ l'application linéaire sur $E^\star_{|U}$ définie par
\[
\Phi_s(\xi)=<\xi, s\circ \pi^\prime(\xi)>
\]
Ainsi, pour tout $\sigma\in E^\star$, $T^\flat_\sigma E^\star$ est engendré par l'ensemble
\[
\{d(\Phi_s+f\circ\pi^\star),\; 
s \in \Gamma \left( E_{| U} \right),\;
 f\in C^\infty(U),\; U \textrm{ voisinage ouvert de } \pi^\star(\xi) \}.
\]
Pour chaque ouvert $U$ de $M$, l'algèbre $\mathfrak{A}  \left( E^\star_{|U} \right) $ est aussi l'algèbre   des fonctions lisses $f:E^\star_{| U}\to\mathbb{R}$ dont la différentielle induit une section de $E^\star_{| U}$.
Soit $\mathfrak{A} _L(E^\star_U)$ l'ensemble des fonctions lisses $f:E^\star_{| U}:=E^\star_U\to\mathbb{R}$  dont la restriction à chaque fibre est linéaire. Il apparaît clairement que  $\mathfrak{A} _L(E^\star_{U})$ est un sous espace vectoriel  de $\mathfrak{A} (E^\star_U)$.\\
Par ailleurs, il existe un morphisme $P:T^{\flat}E^\star \to TE^\star $ défini de la manière suivante:\\
pour tout  $(\eta,\mathsf{w})\in T_\sigma^\flat E^\star$,  il existe un  germe en $x$ de fonction $f$ sur $M$  et un  germe en  $x$ de section $s$  de $E$ tels que (en coordonnées locales) 
\[
\left\{
\begin{array}
[c]{c}
\eta=d_\sigma(f\circ\pi^\prime))+\displaystyle\frac{\partial \Phi_s}{\partial \mathsf{x}}(\sigma)  \\
\mathsf{w}
=d_\sigma{\Phi_s}_{|V_\sigma E^\star}
\equiv s \left( \pi^\star(\sigma) \right)
\end{array}
.\right.
\]
où $V_\sigma E^\star$ est le fibré vertical de  $T_\sigma^\star E^\star$  et où $\Phi_s(\sigma)=<\sigma, s>$  et   $d_\sigma{\Phi_s}_{|V_\sigma E^\star}$ s'identifie à  $s\circ\pi_\star$.\\
Dans ces conditions, en coordonnées locales si $\sigma=(x,\xi)$,  le crochet de Lie s'écrit (cf.  \cite{CaPe12})
$$[s,s']_E(x)=ds'(\rho_x(s)-ds(\rho_x(s'))+C_x(s,s')$$
et  on a  (cf. \cite{CaPe23}, Chapitre~7, preuve du Théorème 7.18)
%\begin{equation}\label{eq_ValueP2}
%\begin{matrix}
%P_\sigma(\eta,\mathsf{w})&=\Phi_{[\mathfrak{a},.]_\rho}(\sigma)+L_{\rho(\mathfrak{a})(x)}(.)-L_{\rho(.)}(f)\circ\pi^\prime(\sigma)\hfill{}\\
%&=-\eta\circ\rho_x(.) -<\sigma, C(\mathsf{w},.)>+\rho(\mathsf{w})\hfill{}\\
%&=-T_\sigma\pi^\prime (d_\sigma(\Phi_\mathfrak{a}+f\circ \pi^\prime)\circ \rho -<\xi, C_x(\mathfrak{a}(x),.>+\rho(\mathfrak{a}(x))\hfill{}\\
%\end{matrix}
%\end{equation}
\begin{equation}
\label{eq_ValueP2}
\begin{matrix}
P_\sigma(\eta,\mathsf{w})&=\left(\rho_x(\mathsf{w}),-\eta\circ\rho_x(.) +<\xi, C_x(\mathsf{w},.)>\right)\hfill\\
&=\left(\rho_x(s), \Phi_{[s,.]_\rho}(\sigma )+ d_xf\circ \rho\right).\hfill{}
\end{matrix}
\end{equation}
\`A $P$ est associé un crochet $\{.,.\}_{P}$ sur $\mathfrak{A} (E^\star_U)\times\mathfrak{A} (E^\star_U)$ 
 dont la restriction à $\mathfrak{A} _L(E^\star_U)\times\mathfrak{A} _L(E^\star_U)$ est à valeurs dans $\mathfrak{A} _L(E^\star_U)$.  Puisque le crochet de Lie $[.,.]_{E}$ vérifie l'identité de Jacobi, ceci implique que l'application bilinéaire $\{.,.\}_P $ vérifie aussi l'identité de Jacobi. En prenant $U=M$,  on obtient ainsi une variété de Poisson $(E^\star,\mathfrak{A} (E^\star),\{.,.\}_{P})$ (pour plus de détails cf. \cite{CaPe23}, Chapitre~7). A noter que la structure canonique symplectique  faible sur $T^\star M$ est précisément la structure de Poisson sur $T^\star M$ associée à la structure naturelle d'algebroide de Lie sur $T^\star M$.
\end{example}

\begin{example}
\label{Ex_CrochetsSurDualAlgebreLieBanachNonNecessairementReflexive}
{\sf Crochets de Poisson sur le dual de  l'algèbre de Banach-Lie d'un groupe de Lie  non nécessairement réflexive.} \\
Cet exemple correspond à un cas particulier du Théorème 3.14 de \cite{Tum20}.\\
Soit $G$ un groupe de Lie-Banach dont l'algèbre de Lie $\mathfrak{g}$ n'est pas nécessairement réflexive et soit $\mathfrak{g}^\star$ le dual de $\mathfrak{g}$.\\
Afin de définir un crochet, il est nécessaire de se restreindre à l'ensemble $\mathfrak{A} $ des fonctions régulières, i.e. aux fonctions $f \in C^\infty \left( \mathfrak{g}^\star \right)$ dont les dérivées $df(\alpha): \mathfrak{g}^\star \to \mathbb{R}$ appartiennent au sous-espace $\mathfrak{g}$ de $\mathfrak{g}^{\star\star}$ pour tout $\alpha\in \mathfrak{g}$.\\ 
On peut alors définir divers crochets classiques sur cet espace :
\begin{enumerate}
\item 
Remarquons  d'une part  que $\mathfrak{g}^\star\times \mathfrak{g}$ est un sous fibré fermé de $ \mathfrak{g}^\star\times \mathfrak{g}^{\star\star}=T^\star\mathfrak{g}^\star$  et l'application $ad^\star$ définit un morphism $P:\mathfrak{g}^\star\times \mathfrak{g} \to   \mathfrak{g}^\star\times \mathfrak{g}^{\star}$ par 
$P_\alpha(X)=\operatorname{ad}^\star_X(\alpha)$. Notons d'autre part que si $\left\langle .,. \right\rangle$ est l'appariement naturel $\mathfrak{g}^\star \times \mathfrak{g} \to \mathbb{R}$, on a
\[
 \left\langle \alpha,[X(\alpha),Y(\alpha)] \right\rangle 
 =\left\langle \operatorname{ad}_X^\star(\alpha),Y(\alpha) \right\rangle=\left\langle P_\alpha(X),Y(\alpha) \right\rangle
 \]
pour toutes sections  $X$ et $Y$  de $\mathfrak{g}^\star\times \mathfrak{g}$ en particulier, $P$ est une quasi ancre de Poisson. Maintenant, si $f$ et $g$ appartiennent  à $\mathfrak{A}$ alors $df $ et $dg$ sont des sections de  $\mathfrak{g}^\star \times \mathfrak{g}$ et par suite est associé le crochet:
\index{Lie-Poisson!crochet}\index{structure!de Lie-Poisson}
\[
\{f,g\}_{\operatorname{LP}}(\alpha)
= \left\langle \alpha,[df(\alpha),dg(\alpha)] \right\rangle. 
\]
Ce crochet satisfait à la fois l'identité de Leibniz et celle de Jacobi. Ainsi, $ \left( \mathfrak{g}^\star,\mathfrak{A} ,\{.,.\}_{\operatorname{LP}} \right) $ est une variété de Poisson partielle.  \\
 En fait, par construction, le crochet de Poisson $\{.,;\}_{\operatorname{LP}}$ n'est rien d'autre que le crochet de Poisson sur $\mathfrak{g}^\star$ associé à la structure d'algébroide de Lie sur $T\mathfrak{g}^\star$ defini par le crochet de Lie sur $\mathfrak{g}$
 \footnote{
Cette structure de Poisson est  linéaire, i.e. le crochet de deux fonctionnelles linéaire est encore une fonctionnelle linéaire.\\
On rencontre ce type de structures dans le processus de réduction d'un système Hamiltonien où le crochet et l'Hamiltonien transformés dépendent d'un plus petit nombre de variables, e.g. pour le crochet de Lie-Poisson pour un solide ou pour l'équation d'Euler d'un fluide idéal en 2D.} (cf. Example \ref{Ex_AlgebroideDeBanachLiePoisson}).\\
Notons que si $G$ est un groupe de Lie de dimension finie, cette structure de Poisson n'est rien d'autre que la structure classique de Kirillow-Kostant-Souriau\index{structure!de Kirillow-Kostant-Souriau} sur $\mathfrak{g}^\star$.
\item
On associe à tout élément fixé $\alpha_0$ de $\mathfrak{g}^\star$ le crochet\index{structure!de Poisson constante} défini par
\[\forall (f,g) \in\mathfrak{A} ^2,
\forall \alpha \in \mathfrak{g}^\star,\;
\{f,g\}_{\alpha_0}(\alpha)= \left\langle \alpha_0,[df(\alpha),dg(\alpha)] \right\rangle 
\]
qui satisfait lui aussi les identités de Leibniz et de Jacobi. La structure de Poisson est appelée structure de Poisson \emph{constante} ou \emph{figée} en $\alpha_0$.
\item
Soit $\omega \in \Lambda^2 \mathfrak{g}^\star$ une $2$-forme antisymétrique sur $\mathfrak{g}^\star$.\\ 
Le crochet $\{f,g\}^{\omega}$ défini par
\[\forall (f,g) \in\mathfrak{A} ^2,
\forall \alpha \in \mathfrak{g}^\star,\;
\{f,g\}^{\omega}
= \left\langle \alpha,[df(\alpha),dg(\alpha)] \right\rangle 
+ \omega \left( df(\alpha),dg(\alpha) \right)
\]
est un crochet de Poisson si et seulement si $\omega$ est un $2$-cocycle\index{cocycle}, i.e. vérifie, pour tout triplet $(x,y,z)$ d'éléments de $\mathfrak{g}$, la condition
\[
\omega \left( [x,y],z \right)
+ \omega \left( [y,z],x \right)
+ \omega \left( [z,x],y \right)
=0.
\]
Ce crochet est appelé \emph{crochet de Lie-Poisson modifié par $\omega$} (cf. \cite{Ben10}, 2.1.4 pour le cas de la dimension finie).\\
De plus, cette structure est isomorphe à la structure de Lie-Poisson si et seulement si $\omega$ est un cobord\index{cobord}, i.e. il existe $\beta \in \mathfrak{g}^\star$ tel que
\[
\forall (x,y) \in {\mathfrak{g}^\star}^2,\;
\omega(x,y) = \left\langle \beta,[x,y] \right\rangle.
\]
La translation par $\beta$ fournit l'isomorphisme entre $\left( \mathfrak{g}^\star,\{.,.\}_{\operatorname{LP}} \right) $ et $ \left( \mathfrak{g}^\star,\{.,.\}^\omega \right) $. 
\end{enumerate}
\end{example}

\begin{example}
\label{Ex_CrochetsSurPreDualAlgebreLieBanachNonNecessairementReflexive}
{\sf Crochets de Poisson sur le pré-dual de  l'algèbre de Banach-Lie d'un groupe de Lie  non nécessairement réflexive.}(cf. \cite{OdRa03}).\\
Soit $G$ un groupe de Lie-Banach dont l'algèbre de Lie $\mathfrak{g}$ n'est pas nécessairement réflexive et soit $\mathfrak{g}^\star$ le dual de $\mathfrak{g}$. Supposons que $\mathfrak{g}$ possède un pré-dual $\mathfrak{g}_\star $\footnote{i.e. le dual de $\mathfrak{g}_\star$ est égal à $\mathfrak{g}$.}. Puisque $\mathfrak{g}_\star\subset(\mathfrak{g}_\star)^{\star\star}=\mathfrak{g}^\star$ la condition suivante a un sens~:
\[
\forall g\in G, \; 
\operatorname{Ad}^\star_g \left( \mathfrak{g}_\star \right)
\subset \mathfrak{g}_\star.
\] 
Dans ce cas, le fibré $\mathfrak{g}_\star\times \mathfrak{g}$ est le fibre cotangent  à $\mathfrak{g}_\star$ et par suite 
$\operatorname{ad}^\star_{X(\alpha)} $ appartient à $\mathfrak{g}_\star$ pour toute section $X$ de ce fibré. Il en résulte que, $P_\alpha(X)= \operatorname{ad}^\star_{X(\alpha)} $ définit une ancre de Poisson  partielle sur $\mathfrak{g}_\star$.
\footnote{Lorsque qu'on  a une ancre de Poisson partielle $P:T^\star M\to TM$, on parlera souvent  simplement d'ancre de Poisson.}
\end{example}

\begin{example}
\label{Ex_LimiteDirecteDeStructuresPartiellesDePoissonCompatiblesSurDesVarietesDEBanach}
{\sf Limite directe de structures partielles de Poisson sur des variétés de Banach.}\\
Cet exemple est fondé sur la section \textit{Direct limits of partial Poisson Banach manifolds} du livre \cite{CaPe23}, Chapitre~7.\\
On considère une suite directe $\left( T^\flat M_i,P_i,\{.,.\}_i \right) _{i \in \mathbb{N}}$ de structures de Banach-Poisson partielles~:\\
$\left( M_i \right) _{i \in \mathbb{N}}$ est une croissante de variétés de Banach lisses où $M_i$ est modelée sur l'espace de Banach $\mathbb{M}_i$ avec $\mathbb{M}_i$ sous-espace de Banach supplémenté dans $\mathbb{M}_{i+1}$ et où $ \left( M_i,\varepsilon_i^{i+1} \right) $ est une sous-variété faible de $M_{i+1}$, $\varepsilon_i^{i+1}:M_i \to M_{i+1}$ désignant l'injection canonique. On a de plus les propriétés suivantes~:
\begin{enumerate}
\item[\emph{\textbf{(1)}}]
$T^{\star}\varepsilon_{i}^{i+1}(T^{\flat}M_{i+1})\subset T^{\flat}M_{i}$;
\item[\emph{\textbf{(2)}}]
$P_{i+1}=T\varepsilon_{i}^{i+1}\circ P_{i}\circ T^{\star}\varepsilon_{i}^{i+1}$;
\item[\emph{\textbf{(3)}}]
Au voisinage de chaque point $x\in M$, il existe une suite de cartes locales $\left( U_{i},\phi_{i}\right)  _{i\in\mathbb{N}}$ telle que
$ \left( U=\underrightarrow{\lim}(U_{i}),\phi=\underrightarrow{\lim}(\phi_{i}) \right) $ est une carte locale en $x$ dans $M$, de telle manière que les cartes $ \left( U_{i},\phi_{i} \right) $ et $ \left( U_{i+1},\phi_{i+1} \right) $ sont compatible avec la propriété \emph{\textbf{(1)}}.
\end{enumerate}
Il existe alors un sous fibré faible $p^{\flat}:T^{\flat} M\to M$ de $p'_M : T'M \to M$, où $M = \underrightarrow{\lim} M_i$ est une variété adaptée, et un morphisme $P:T^{\flat
}M\to TM$ tel que $(M,\mathfrak{A} _P(M),\{.,.\}_{P})$ est une structure de Poisson partielle.\\
Si, pour $i$ entier naturel quelconque, ${\varepsilon}_{i}:M_{i}\to M$ est l'injection canonique alors, qui plus est, ${\varepsilon}_{i}$ est un morphisme de Poisson de $ \left( M_{i},\mathfrak{A} (M_{i}),\{.,.\}_{P_{i}} \right) $ dans $ \left( M,\mathfrak{A} (M),\{.,.\}_{P} \right) $ et
\[
\mathfrak{A} _P(M)=\underrightarrow{\lim}\mathfrak{A} _{P_i}(M_i)
\]
\[
\{.,.\}_{P}=\underrightarrow{\lim}\{.,.\}_{P_{i}}.
\]  
\end{example}

%!+ exemple essentiel ajouté pour justifier la nécessité d'utiliser le cadre convenient
\begin{example}
\label{Ex_StructureDePoissonPartielleSurPinftyM}
{\sf Structure de Poisson partielle sur l'espace $P^\infty(M)$.}\\
Dans cette exemple, on se réfère essentiellement à \cite{Lot18}.\\
On considère une variété Riemannienne fermée connexe de dimension finie $(M,g)$ et on note $\mathsf{dvol}_M$ la forme volume sur $M$ associée à $g$. On considère de plus l'ensemble\footnote{Cet ensemble est aussi utilisé en Géométrie de l'information où il fournit un cadre à la version infinie dimensionnelle de la \emph{matrice d'information de Fisher} (cf. \cite{Laf88}).} (cf. \cite{Lot18}) 
\[
P^\infty(M)=\{ \rho\;\mathsf{dvol}\;:\; \;\rho\in C^\infty(M),\; \rho>0,\; \displaystyle\int_M \rho\;\mathsf{dvol}_M=1\}.
\]
$P^\infty(M)$ peut être munie d'une structure de variété adaptée (\cite{Lot18}).\\
D'autre part, à toute fonction $\phi\in C^\infty(M)$, on associe la fonction $F_\phi\in C^\infty (P^\infty(M))$ donnée par~:
\[
F_\phi(\rho\;\mathsf{dvol}_M)
=\displaystyle \int_M \phi\rho\;\mathsf{dvol}_M.
\]
Le gradient de  $ \phi\in C^\infty(M)$  est noté $\nabla\phi$ et est associé au champ de vecteurs $V_\phi$ sur $P^\infty(M)$ défini par~:
 \begin{equation}
\label{eq_Xphi}
V_\phi(F)(\rho\;\mathsf{dvol}_M)
=\displaystyle\frac{d}{d\tau}_{| \tau=0} F(\rho\;\mathsf{dvol}_M -\tau \mathsf{div}(\rho\nabla\phi)\mathsf{dvol}_M)
\end{equation}
pour tout $F\in C^\infty(P^\infty(M))$.\\
Par ailleurs, on définit une métrique $\bar{g}$ sur $P^\infty(M)$ (cf. \cite{Ot01}) par
\[
\bar{g}_{\rho\;\mathsf{dvol}_M}(V_\phi,V_\psi)
=\displaystyle\int_M g(\nabla\phi,\nabla\psi)\rho\;\mathsf{dvol}_M.
\]
Si $T'P^\infty(M)$ désigne le fibré cotangent cinématique de $P^\infty(M)$, l'application $X \mapsto \bar{g}(X,.)$ est un morphisme injectif $\bar{g}^\flat:TP^\infty(M)\to T'P^\infty(M)$. Il s'ensuit que l'image $T^\flat P^\infty(M)$ 
de $\bar{g}^\flat$ est un sous-fibré de $T'P^\infty(M)$.  En particulier, pour toute fonction $\phi\in C^\infty(M)$, la différentielle $dF_\phi$ est la restriction à $T_{\rho\mathsf{dvol}_M}P^\infty(M)$  de l'application linéaire $\phi\mapsto \displaystyle\int_M\phi\;\mathsf{dvol}_M$ et ainsi $dF_\phi$ est une forme différentielle sur $P^\infty(M)$. De plus, $dF_\phi$ est une section de $T^\flat P^\infty(M)$ pour tout $\phi\in C^\infty(M)$.\\
On note $\mathcal{A}(P^\infty(M))$ l'espace vectoriel  $\{F_\phi,\;\phi\in C^\infty (M)\}$ et on le munit d'une structure d'algèbre pour le produit donné par  
$F_\phi.F_\psi=F_{\phi \psi}$.\\
Ainsi
$ 
\begin{array}[c]{ccc}
C^\infty(M) 	& \to 		& \mathcal{A}(M)		\\
\phi     		& \mapsto 	& F_\phi
\end{array}
$
 est un morphisme d'algèbres (cf. \cite{Lot18}). \\
Supposons maintenant, que la variété $M$ soit munie d'une structure de Poisson définie par un bivecteur $\Lambda$. Cette structure définit sur $C^\infty M$ un crochet de Poisson $\{.,.\}_\Lambda$. On obtient alors sur $\mathcal{A}(P^\infty(M))$ une algèbre de Poisson relativement au crochet de Poisson défini par~: 
\begin{equation}\label{eq_PoissonbracketPinfty}
\{F_\phi,F_\psi\}_{P^\infty(M)}(\rho)
=\displaystyle\int_M \{\phi,\psi\}_\Lambda \;\rho\;\mathsf{dvol}_M
\end{equation}
où $\{\phi,\psi\}_\Lambda$ est le crochet de Poisson sur $C^\infty(M)$ (cf. \cite{Lot18}). Notons que l'application $\phi \mapsto F_\phi $ est un morphisme d'algèbre de Poisson de  $C^\infty(M )$ dans $\mathcal{A}(P^\infty(M))$.
\end{example}

\subsubsection{Distribution caractéristique}
\label{____DistributionCaracteristique}

\textbf{Feuilletages et distributions.}\\
\label{____FeuilletagesEtDistributions}
%La théorie des feuilletages a permis de% résoudre des problèmes dans le cadre des systèmes dynamiques.\\
Un \emph{feuilletage}\index{feuilletage} d'une variété adaptée $M$ est une partition de $M$ en sous-variétés faibles, appelées \emph{feuilles}\index{feuille} du feuilletage.\\
D'un point de vue infinitésimal, pour une variété feuilletée $M$,  on désigne pour tout point $x$ de $M$ par $\mathcal{D}_x$ l'espace tangent à la feuille passant par $x$.  Le feuilletage est dit  \emph{régulier}\index{feuilletage!régulier} si $\mathcal{D}:=\displaystyle\cup_{x\in M} \mathcal{D}_x$ est un sous fibré fermé de $TM$. Sinon le feuilletage est dit \emph{singulier}\index{feuilletage!singulier}. Pour un feuilletage régulier,  pour tout couple de champs de vecteurs locaux tangents à  $\mathcal{D}$, leur crochet de Lie est encore tangent à $\mathcal{D}$.\\

Une \emph{distribution}\index{distribution} sur une variété adaptée $M$ est une application $\mathcal{D}: x \mapsto \mathcal{D}_x\subset T_xM$ sur $M$, où $\mathcal{D}_x$\index{Dmathcal@$\mathcal{D}_x$ (distribution)} est un sous-espace vectoriel de $T_xM$.
Une distribution $\mathcal{D}$ est dite
\begin{enumerate}
\item
\emph{lisse}\index{distribution!lisse}\index{lisse!distribution} si, pour tout $x\in M$, $\mathcal{D}_x$ est engendrée par les champs de vecteurs locaux qui sont tangents à  $\mathcal{D}$ au voisinage de $x$;
\item
\emph{integrable}\index{distribution!integrable}\index{integrable!distribution} si, pour tout $x\in M$, il existe une sous-variété immergée faible  $f:L\to M$ telle que
\[
\forall z\in L,\; T_z f( T_z L )= \mathcal{D}_{f(z)};
\]
\item
\emph{involutive}\index{distribution!involutive}\index{involutive!distribution} si pour tous champs de vecteurs locaux $X$ et $Y$ sur $M$ qui sont tangents à $\mathcal{D}$, le crochet de Lie $[X,Y]$ est aussi tangent à $\mathcal{D}$.\\
%\item
%\emph{invariant}\index{distribution!invariant}\index{invariant!distribution} if, for any local vector field $X$ tangent to $\mathcal{D}$, the local flow $\operatorname{Fl}^X_t$ leaves $\mathcal{D}$ invariant for any $t$ for which $\operatorname{Fl}^X_t$ is defined.
\end{enumerate}

\textbf{Distribution caractéristique.}\\
\label{____DistributionCaracteristique}

On associe à une structure de Poisson partielle $ \left( T^\flat M, TM, P \right) $, la famille de sous-espaces vectoriels 
\[
\{\mathcal{D}_x=P(T_x^\flat M)\subset T_xM\}_{x\in M}.
\]
La réunion $\mathcal{D}:=\displaystyle\bigcup_{x\in M} \mathcal{D}_x$ est une distribution sur $M$ appelée \emph{distribution caractéristique}\index{Dmathcal@$\mathcal{D}$ (distribution)}\index{distribution!caractéristique}.\\
En dimension finie, quand $T^{\flat} M = T^{\prime}M$, il est bien connu que  la distribution caractéristique donne naissance à un feuilletage de Stefan-Sussman\index{feuilletage!de Stefan-Sussman} \index{Stefan-Sussman!feuilletage}\index{feuilletage} et que chaque \emph{feuille}\index{feuille} est une sous-variété symplectique immergée (cf. \cite{Wei94}).\\ 
Dans le cadre adapté, ceci n'est plus vrai en général. En fait, le problème apparaît dès le cadre de Banach (cf. Remark \ref {R_Feuille}). On a cependant les conditions suffisantes suivantes pour lesquelles un tel résultat est vrai (cf. \cite{PeCa19}).

\begin{theorem}
\label{T_FeuilletageSurUneVarieteDeBanachPoissonPartielle} 
Soit $ \left( T^\flat M, TM, P\right) $ une variété de Poisson partielle pour laquelle le noyau de $P$ est scindé dans chaque fibre $T^{\flat}_{x}M$ de $T^{\flat}M$ et où $P(T^{\flat}M)$ est une distribution fermée. On  alors~:
\begin{enumerate}
\item 
$\mathcal{D}=P(T^{\flat}M)$ est intégrable et le feuilletage défini par $\mathcal{D}$ est un feuilletage  symplectique faible, i.e. tel que sur chaque feuille $N$ de ce feuilletage, on a une 2-forme fermée  non dégénérée $\omega_N$ sur $TN$ telle que pour tout $u$ et tout $v$ de $T_xN$, on a $\omega_N(u, v)= < \alpha, P(\beta)>$ pour tout couple $(\alpha, \beta)$ d'éléments de $T_x^\flat M$ où $P(\alpha)=u$ et $P(\beta)=v$.
\item 
Sur chaque feuille maximale $N$, on considère la forme naturelle faible $\omega_N$. On note $\mathcal{E}_{\omega_N} (V)$ le sous-ensemble des fonctions $f\in C^\infty(V)$ telles que $df$ appartienne à $\omega_N^\flat(TV)$.  La restriction $f_N$ à $U\cap N$  de n'importe quelle fonction $f\in \mathcal{E}(U)$ appartient à $\mathcal{E}_{\omega_N}(U\cap N)$ et on a, pour tout $f$ et tout $g$ de $\mathcal{E}(U)$,
\[
{\{f_{|N},g_{|N}\}_{P}}_{|N}=\{f_{N},g_{N}\}_{\omega_{N}}.
\]
\end{enumerate}
\end{theorem}

\begin{remark}
\label{R_Feuille}
En général, dans le cadre adapté, la distribution caractéristique n'est pas intégrable. Notons que  dans le cadre de Banach, lorsque le noyau de $P$ est scindé mais  si la distribution caractéristique $\mathcal{D}$ n'est pas fermée  elle peut être intégrable (cf. Exemple \ref{Ex_DistributionsCaracteristiquesSurDualAlgebreLieBanachNonNecessairementReflexive} et \ref{Ex_DistibutionCaractéristiqueSurPeDualAlgebreLieBanachNonNecessairementReflexive}). Comme ces exemples sont des cas particuliers,  nous pensons que, en général,  la distribution caractéristique n'est pas intégrable même si  le noyau de $P$ est scindé car, en dehors du cadre de Banach,  il n'existe pas de théorème d'existence de flot ou des fonctions implicites. Malheureusement, nous n'avons  pas d'exemple pouvant étayer une telle conjecture.
\end{remark}

Cependant, plus généralement, on a aussi:

\begin{theorem}
\label{T_LimiteDirecteDeVarietesDePoisson}
On considère une suite directe $\left( T^\flat M_i,P_i,\{.,.\}_i \right) _{i \in \mathbb{N}}$ de structures de Banach-Poisson partielles\footnote{cf. Exemple~\ref{Ex_LimiteDirecteDeStructuresPartiellesDePoissonCompatiblesSurDesVarietesDEBanach}.} ayant les propriétés suivantes
\begin{itemize}
\item[(1)] Chaque variété de Poisson partielle  $\left( T^\flat M_i,P_i,\{.,.\}_i \right) _{i \in \mathbb{N}}$ vérifie les hypothèses du Théorème \ref{T_FeuilletageSurUneVarieteDeBanachPoissonPartielle} pour tout entier $i$.
\item[(2)] pour tout  $x\in M$,il existe un voisinage ouvert $U$  de $x=\underrightarrow{\lim}x_i$ et une suite ascendante de voisinage ouvert $U_i$ simplement connexe de $x_i$ tels
$U=\underrightarrow{\lim}U_{i},\;$ ${T^\flat M_i}_{|U_i}$ est  supplémenté dans $ {T^\flat M_{i+1}}_{|U_i}$ 
\end{itemize}
Alors la distribution caractéristique $\mathcal{D}=P(T^\flat M)$ est intégrable et la  feuille $N$ contenant $x$ est la limite directe de la suite des feuilles $N_i$ contenant $x_i$. De plus, $L$ possède une structure de variété symplectique faible définie par une forme symplectique $\omega=\underrightarrow{\lim}\omega_i$ si $\omega_i$ est la forme symplectique faible canonique sur $N_i$.
\end{theorem}

\begin{proof}[Idée de la démonstration]  
Tout d'abord, il résulte 
%de l'hypothèse (2) du Théorème et 
 de l'Exemple 
\ref{Ex_LimiteDirecteDeStructuresPartiellesDePoissonCompatiblesSurDesVarietesDEBanach} que la distribution caractéristique $P(T^\flat M)$ est une limite directe locale de fibrés Banachiques de Koszul puisque chaque ouvert $U_i$ est simplement connexe et donc ${T^\flat M_i }_{|U_i}$ est muni d'une connection canonique linéaire plate et donc compatible 
avec la structure canonique plate sur ${T^\flat M_{i+1}}_{|U_{i+1}}$ (cf. \cite{CaPe23}, Chapitre~8, Définition~8.6). 
%!++ Théorème 8.66 remplacé par Théorème 8.12 
Il suffit  alors d'appliquer le  Théorème 8.12 de \cite{CaPe23}\footnote{En fait, il faut remplacer l'hypothèse "$(E_n, \pi_n, U_n, \rho_n, [., .]_n)$" est un algébroïde de Banach par "$(E_n, \pi_n, U_n, \rho_n, [., .]_n)$" est un algebroïde de Banach partiel  fort et la démonstration du Théorème  8.12 de \cite{CaPe23} fonctionne de la même manière.}. 
Maintenant concernant le résultat $\omega=\underrightarrow{\lim}
\omega_i$, 
%!++ Proposition 5.56 remplacée par Proposition 5.20
il suffit d'appliquer la Proposition 5.20 du Chapitre~5 de \cite{CaPe23}.
\end{proof}

\begin{example}
\label{Ex_DistributionsCaracteristiquesSurDualAlgebreLieBanachNonNecessairementReflexive}
{\sf Distributions caractéristiques sur le dual d'une algèbre de Lie-Banach d'un groupe de Lie non nécessairement réflexive.}\\
On considère un groupe de Lie-Banach $G$ dont l'algèbre de Lie $\mathfrak{g}$ n'est pas nécessairement réflexive et on désigne par $\mathfrak{g}^\star$ le dual de $\mathfrak{g}$ (cf. Exemple~ \ref{Ex_CrochetsSurDualAlgebreLieBanachNonNecessairementReflexive}).\\
Soit l'algèbre des fonctions régulières
\[
\mathfrak{A} 
=
\{
f \in C^\infty \left( \mathfrak{g}^\star \right),\;
\forall \alpha \in \mathfrak{g}^\star,\;
\left( df(\alpha): \mathfrak{g}^\star \to \mathbb{R} \right)
\in \mathfrak{g} \subset  \mathfrak{g}^{\star\star}
\}.
\]
Le crochet de Lie-Poisson peut être écrit pour tout couple $(f,g)$ de fonctions de $\mathfrak{A} $ et tout élément $\alpha \in \mathfrak{g}^\star$,
\[
\begin{array}{rcl}
\{f,g\}_{\operatorname{LP}}(\alpha) 
		&=&
	\left\langle \alpha,[df(\alpha),dg(\alpha)] \right\rangle	\\
		&=&	
	\left\langle \alpha,\operatorname{ad}_{df}(\alpha) \left( dg(\alpha) \right) \right\rangle						\\				
		&=&
	\left\langle  \operatorname{ad}^\star_{df}(\alpha) ,  dg(\alpha) 
	 \right\rangle \\
	& = &\left\langle  P_\alpha(df) ,  dg(\alpha) 
	 \right\rangle.
\end{array}
\]
Ainsi le champ hamiltonien associé à l'hamiltonien $f$ est 
\[
X_f^{\operatorname{LP}}(\alpha) = \operatorname{ad}^\star_{df}(\alpha)=P_\alpha(df).
\]
Dans le cadre Banachique, si pour chaque $\alpha\in \mathfrak{g}^\star$, le noyau $\mathfrak{g}_\alpha$ de $P$ est scindé dans $\mathfrak{g}$, alors il existe un  groupe de Lie caractérisé par 
\[
G_\alpha 
= \{g\in G : \operatorname{Ad}^\star_g(\alpha)=\alpha\}
\]  
dont l'algèbre de Lie est $\ker P_\alpha$ (groupe d'isotropie de $\alpha$). Par suite, le quotient $G/G_\alpha$ a une structure de variété Banachique quotient  $G/G_\alpha$. 
(cf. \cite{Bou72}, Chapter~III, §6.4, Corollary~1). 
 Soit $\iota_\alpha$  l'application quotient  de l'application $g\mapsto \operatorname{Ad}^\star_g(\alpha)$ par $G_\alpha$. Alors $\iota_\alpha$ est une  application injective $C^\infty$ de $G/G_\alpha$ dans $\mathfrak{g}^\star$. Il en résulte que son image (qui est l'orbite  coadjointe $\mathcal{O}_\alpha$ de $\alpha$)  est munie de cette structure de variété banachique et l'espace tangent à cette orbite est l'image de $P_\alpha$. Cette image n'est pas fermée en général  si cet espace est de dimension infinie.  Ainsi  la distribution caractéristique est intégrable  puisque  les feuilles de la structure de Lie-Poisson sont alors les orbites coadjointes de l'action du groupe sur $\mathfrak{g}^\star$ et elle n'est pas nécessairement fermée~; elle n'est fermée que si ces orbites sont de dimensions finies (cf. \cite{OdRa08b} par exemple).  Cette situation généralise la situation classique de    Kirillow-Kostant-Souriau sur $\mathfrak {g}^\star$ (cf.\cite{Kir81}).\\

%Par ailleurs, le plan tangent à l'orbite coadjointe passant par le point $\alpha_0$ ainsi que tous les plans parallèles à celui-ci sont les feuilles symplectiques de la structure de  Poisson constante en $\alpha_0$.\\
%On trouvera dans \cite{KhMi03} une classification des orbites de codimension $2$ sur le dual $\operatorname{vir}^\star$ de l'algèbre de Virasoro\index{vir@$\operatorname{vir}$ (algèbre de Virasoro)} correspondant à $3$ structures gelées et associées aux EDP Korteweg-de Vries
%\footnote{
%(KdV) $u_t = c u_{xxx} - 3u u_x$
%}, 
%Camassa-Holm 
%\footnote{
%(CH) $u_t - u_{txx} = -3u u_x + 2u_x u_{xx} 
%+ u u_{xxx} + c u_{xxx}$
%}
% et Hunter-Saxton
%\footnote{
%(HS) $u_{txx} = -2u_x u_{xx} - u u_{xxx}$
%}. 
%\footnote{L'algèbre de Virasoro est une extension de l'algèbre de Lie $\operatorname{Vect} \left( \mathbb{S}^1 \right) $ des champs de vecteurs sur le cercle $\mathbb{S}^1$:
%\[
%\operatorname{vir} = \operatorname{Vect} \left( \mathbb{S}^1 \right) \oplus \mathbb{R}
%\] 
%}
\end{example}

\begin{example}
\label{Ex_DistibutionCaractéristiqueSurPeDualAlgebreLieBanachNonNecessairementReflexive}
{\sf Distribution caractéristique sur le pré-dual de  l'algèbre de Banach-Lie d'un groupe de Lie  non nécessairement réflexive.}(\cite{OdRa03}).\\
Soit $G$ un groupe de Lie-Banach dont l'algèbre de Lie $\mathfrak{g}$ n'est pas nécessairement réflexive et supposons que $\mathfrak{g}$ possède un pré-dual $\mathfrak{g}_*$ tel que  pour tout   $g\in G$, on a $Ad^\star_g(\mathfrak{g}_\star)\subset \mathfrak{g}_\star$ (cf. Exemple \ref{Ex_CrochetsSurPreDualAlgebreLieBanachNonNecessairementReflexive}). Comme dans l'Exemple \ref{Ex_DistributionsCaracteristiquesSurDualAlgebreLieBanachNonNecessairementReflexive}  on suppose de plus que pour tout $\alpha\in \mathfrak{g}_*$ le noyau de $P_\alpha:T^\star_\alpha\mathfrak{g}_*\to T_\alpha \mathfrak{g}_*$ est scindé\footnote{Cette condition est équivalente à la condition (iii) du Théorème 7.3 de \cite{OdRa03}.}. Les mêmes arguments que dans l'exemple précédent permettent de montrer que la distribution caractéristique est intégrable et que chaque feuille est difféomorphe au quotient $G/G_\alpha$ où $G_\alpha$ est le groupe d'isotropie de $\alpha$. En particulier une fois de plus  si la distribution caractéristique n'est pas de dimension finie, elle n'est pas nécessairement  fermée  (cf.  \cite{OdRa08b} par exemple).
\end{example}

\begin{example}
\label{Ex_AlgebroideDeBanachLiePoissonDistributionCaractéristique}
{\sf Distribution caractéristique d'une  structure de Poisson partielle  associée à  un algébroïde de Lie-Banach}\\
On considère un algébroïde  Lie-Banach $ \left( E,M,\rho,[.,.]_E \right) $. Supposons  que $\ker\rho_x$ soit scindé dans la fibre $E_x$ pour tout $x\in M$ et que l'image de $\rho_x$ soit fermée.
Alors la distribution $\rho(E)$ est intégrable (cf \cite{Pel12} ou \cite{CaPe23}, Théorème~8.47). Soit $P:T^\flat E^\star\to TE^\star$ l'ancre de Poisson associée à cet algébroïde (cf. Exemple 
\ref{Ex_AlgebroideDeBanachLiePoisson}). Par construction, on a $\mathcal{D}(\sigma)=P(T_{\sigma}^\flat E^\star)$. % Comme $T_\sigma E^\star$ est engendré par les différentielle en $\sigma$ 
%de tous les germes de fonctions du type $\Phi_s+f\circ\pi^\star$, $D(\sigma)$ est donc engendré par les champs  Hamiltoniens 
%$$\left(\rho(s\circ\pi^\prime(\sigma), \Phi_{[s,.]_\rho}(\sigma )+ d_\sigma(f\circ\pi^\prime)\circ \rho\right)\in T_\sigma E^\star$$
%(cf. Remark 7.21 Chapter 7 \cite{CaPe23}). 
Rappelons que  si $\sigma=(x,\xi)$, en coordonnées locales,  on a comme indiqué dans l'exemple~ \ref{Ex_AlgebroideDeBanachLiePoisson}, (\ref{eq_ValueP2})~:
\begin{equation}
\label{eq_ValueP3}
\begin{matrix}
P_\sigma(\eta,\mathsf{w})&=\left(\rho_x(\mathsf{w}),-\eta\circ\rho_x(.) +<\xi, C_x(\mathsf{w},.)>\right)\hfill\\
&=\left(\rho_x(s), \Phi_{[s,.]_\rho}(\sigma )+ d_xf\circ \rho\right).\hfill{}
\end{matrix}
\end{equation}
Si l'algebroïde de Lie-Banach $ \left( E,M,\pi,\rho,[.,.]_\rho \right) $
 est tel que l'image de $\rho_x$ soit fermée et son noyau scindé, 
 %alors la distribution $P(TE)$ est intégrable (cf. \cite{CaPe23}, Chapitre~8, Théorème~8.47 ).
%Sous ces hypothèses,
 l'image de $P_\sigma$ est fermée dans $T_\sigma E^\star$. En effet, si $F$ est un supplémentaire de $K := \ker\rho_x$  dans la fibre $E_x$, la restriction de $\rho_x$  à $F$ est 
un isomorphisme sur $\rho_x(E_x)$. D'autre part  le noyau de $\rho^\star_x$ est l'annulateur $(\operatorname{im}\rho_x)^{\operatorname{ann}}$ de $\operatorname{im}\rho_x$ et $\rho^\star_x$ induit un isomorphisme $\widehat{\rho_x^\star}$ de $T_x^\star M/ (\operatorname{im}\rho_x)^{\operatorname{ann}} $ sur $\operatorname{im}\rho_x^\star$ qui est égal à  l'annulateur $K^{\operatorname{ann}} $~; par conséquent, $\operatorname{im}\rho_x^\star$ est fermé.  Considérons  $ (X, \theta)\in T_\sigma E^\star$ qui appartient à la 
fermeture de $\operatorname{im}(P_\sigma)$. il existe une suite $(\eta_n,\mathsf{w}_n)\in F\times T_x^\star M$ telle que  $P_\sigma(\eta_n,\mathsf{w}_n)$ converge vers $(X,\theta)$. Comme ${\rho_x}_{| F}$ est un 
isomorphisme, $\mathsf{w}_n$ converge vers $\mathsf{w}\in F$ et on a $\rho_x(\mathsf{w})=X$. De même $\xi\circ C_x(\mathsf{w}_n,.)$ converge vers $\xi\circ C_x(\mathsf{w},.)$  et comme  la 
composante verticale de $P_\sigma(\eta_n,\mathsf{w}_n)$  converge vers $\theta $, il en résulte que $\bar{\eta}:=\theta-\xi\circ C_x(\mathsf{w},.)$ appartient à l'image de $\rho^\star$ puisque cette image est fermée  et isomorphe à $K^{\operatorname{ann}} $.  Il en résulte que  la suite de classes d'équivalences de $[\eta_n]\in T_x^\star M/ (\operatorname{im}\rho_x)^{\operatorname{ann}} $ converge vers $[ \eta]=(\widehat{\rho^\star_x}_{| K^{\operatorname{ann}} })^{-1}(\bar{\eta})$ d'où le résultat. \\

Par ailleurs, on a 
\[
\ker P_\sigma=\{(\mathsf{w},\eta)\in K\times T^\star M,\; \rho_x^\star\eta=\xi\circ C_x(\mathsf{w},.)\}.
\]
Par suite si  l'espace $H_x:= \ker P_\sigma$
est  scindé 
\footnote{Cette situation est toujours  vraie si la fibre type $ \mathbb{E}$ du fibré $E$ est un espace de Hilbert ou si $\rho_x$ est un opérateur  semi-Fredholm}.  
 la distribution caractéristique est intégrable (cf. Théorème \ref{T_FeuilletageSurUneVarieteDeBanachPoissonPartielle}). De plus, chaque feuille caractéristique  contenant $(x,\xi)\in E^\star$ est alors un fibré  Banachique au-dessus de la feuille contenant $x$  du feuilletage sur $M$ associé à $\rho$  et dont la fibre est isomorphe à $E^\star_x/H_x$.
\end{example}

%\begin{example}
%\label{Ex_CrochetsSurDualAlgebreLieBanachNonNecessairementReflexive}
%{\sf Crochets de Poisson sur le pré-dual de  l'algèbre de Banach-Lie d'un groupe de Lie  non nécessairement réflexive.}(\cite{OdRa03}) Soit $G$ un groupe de Lie-Banach dont l'algèbre de Lie $\mathfrak{g}$ n'est pas nécessairement réflexive et soit $\mathfrak{g}^\star$ le dual de $\mathfrak{g}$. Supposons que $\mathfrak{g}$ possède un pré-dual $\mathfrak{g}_*$\footnote{ c'est-à-dire   le dual de $\mathfrak{g}_*$ est égal à $\mathfrak{g}$} . Puisque $\mathfrak{g}_*\subset(\mathfrak{g}_*)^{**}=\mathfrak{g}^\star$ la condition suivante a un sens:
% pour tout   $g\in G$, on a $Ad^\star_g(\mathfrak{g}_*\subset \mathfrak{g}_*$. Dans ce cas, si de plus on suppose que
 
% Pour tout $\alpha\in  the coadjoint isotropy subgroup G? := {g ? G | Ad?g ? = ?}, which is closed in G, is a Lie subgroup of G (in the sense that it is a submanifold of G and not just injectively immersed).

\begin{example}
\label{Ex_DistributionsCaracteristiquesLimitesDirecteVariétésPoissodimension finie}
{\sf Limite directe de variétés de Poisson partielles de dimension finie.}\\
Si  $\left( T^\flat M_i,P_i,\{.,.\}_i \right) _{i \in \mathbb{N}}$ est une suite croissante de variétés de Poisson de dimension finies, toutes les hypothèses du Théorème \ref{T_LimiteDirecteDeVarietesDePoisson} sont satisfaites. Par suite la limite directe $M=\underrightarrow{\lim}M_i$ possède une structure de variété adaptée de Poisson partielle  dont la distribution caractéristique est intégrable. En particulier, toute limite directe de variétés symplectiques de dimensions finies est une variété adapté faiblement symplectique. \\
Toute fonction $f\in \mathfrak{A} _P(M)$ est une limite directe $f=\underrightarrow{\lim}f_i$,où $f_i\in\mathfrak{A} _{P_i}(M_i)$ (cf.  Example \ref{Ex_LimiteDirecteDeStructuresPartiellesDePoissonCompatiblesSurDesVarietesDEBanach}) et par suite possède un champ  Hamiltonien  $X_f=\underrightarrow{\lim}X_{f_i}$.
\end{example}

\begin{example}
\label{Ex_DistributionsCaracteristiquesLimitesDirecteVariétésPoissondimension hilbertienne}
{\sf Limite directe de variétés de Poisson partielles de Hilbert.}\\
Si  $ \left( T^\flat M_i,P_i,\{.,.\}_i \right) _{i \in \mathbb{N}}$ est une suite croissante de variétés hilbertiennes de Poisson partielles vérifiant l'hypothèse (1) du Théorème \ref{T_LimiteDirecteDeVarietesDePoisson},   l'hypothèse (2) est automatiquement vérifiée. Par suite la limite directe $M=\underrightarrow{\lim}M_i$ possède une structure de variété de Poisson partielle adaptée dont la distribution caractéristique est intégrable.  En particulier, toute limite directe de variétés hilbertiennes symplectiques  est une variété adaptée faiblement symplectique.\\ 
Comme dans l'Exemple \ref{Ex_DistributionsCaracteristiquesLimitesDirecteVariétésPoissodimension finie}, toute fonction $f\in \mathfrak{A} _P(M)$ est une limite directe $f=\underrightarrow{\lim}f_i$, où $f_i\in\mathfrak{A} _{P_i}(M_i)$ et, par suite, possède un champ  Hamiltonien  $X_f=\underrightarrow{\lim}X_{f_i}$.
\end{example}

Ces résultats et exemples justifient la définition suivante.

\begin{definition}
\label{D_LocalSymplectique}  
Soit  $ \left( M,\mathfrak{A} _M \{.,.\}_P \right) $ une structure de Poisson partielle  sur une variété adaptée $M$ associée à l'ancre de Poisson $P: T^\flat M \to TM$. 
 On dit que la distribution caractéristique $\mathcal{D}$ est  \emph{fermée scindée intégrable}\index{distribution!fermée scindée intégrable} (\emph{\textbf{FSI}} en abrégé) si
\begin{itemize}
\item[$\bullet$] 
pour chaque  $x\in M$, le sous espace vectoriel  $P(T_x^\flat M)$ est fermé et $\ker P_x$ est scindé dans $T_x^\flat M$~;
\item[$\bullet$] 
l'application  linéaire bijective canonique  $\hat{P}_x$   entre $T_x^\flat M/ \operatorname{ker}P_x$ et $\mathcal{D}(x)=P(T_x^\flat M)$ associée  à $P_x$ est un isomorphisme~;
\item[$\bullet$]  
$\mathcal{D}$ est intégrable.
\end{itemize}
\end{definition}

A noter que  dans  le contexte des Théorèmes \ref{T_FeuilletageSurUneVarieteDeBanachPoissonPartielle}  et \ref{T_LimiteDirecteDeVarietesDePoisson} la distribution caractéristique est \textbf{FSI}.

Il résulte de la Remarque \ref{R_OP}
que si la distribution caractéristique est \textbf{FSI} alors l'espace caractéristique  $\mathcal{D}(x)$ est muni d'une forme bilinéaire  bornée symplectique faible  $\Omega_x$ canonique.  
On introduit alors la notion suivante (cf. \cite{PeCa19}).

\begin{definition}
\label{D_SymplecticLeaf}
 Soit $ \left( M,\mathfrak{A} _M \{.,.\}_P \right) $ une structure de Poisson partielle sur une variété adaptée $M$ associée à l'ancre de Poisson $P: T^\flat M \to TM$ telle que la distribution caractéristique est intégrable. Une feuille symplectique faible est une sous variété faible $(N, \iota_N)$ où $N\subset M$ et $\iota_N:N\to M$  est l'inclusion naturelle possédant les propriétés suivantes~:
\begin{enumerate}
\item
$(N,\iota_N)$ est une feuille de  $\mathcal{D}$~;
\item
Sur ${N}$, il existe une forme symplectique faible  $\omega_{N}$ telle que $(\omega_{N})_x=\Omega_x$  pour tout $x\in \mathcal{N}$.
\end{enumerate}
\end{definition}

Comme en dimension finie, on a le résultat suivant établi dans \cite{PeCa19}~:

\begin{theorem} 
\label{T_IntegrabilityCharacteristicDistributionWeakSymplecticLeaf} Soit $ \left(T^\flat M,M, P, \{.,.\}_P \right) $ une structure de Poisson partielle sur une variété adaptée $M$ associée à l'ancre de Poisson $P: T^\flat M \to TM$.
Si la distribution caractéristique est \emph{\textbf{FSI}}, alors chaque feuille a une structure naturelle de feuille symplectique faible.
\end{theorem}

%Bien sûr ce théorème s'applique en particulier dans le contexte des Théorèmes \ref{T_FeuilletageSurUneVarieteDeBanach-PoissonPartielle}  et \ref{T_LimiteDirecteDe variétédePoisson} la distribution caractéristique est \textbf{FSI}.

\begin{definition}
\label{D_DistributionCaracteristiqueMaximale} 
Soit $ \left( M,\mathfrak{A} _M, \{.,.\}_P \right) $ une structure de Poisson partielle sur une variété adaptée $M$ associée à l'ancre de Poisson $P: T^\flat M \to TM$. On dit que la distribution caractéristique $\mathcal{D}$ est \emph{maximale}\index{distribution caractéristique!maximale} en $x \in M$ s'il existe un voisinage ouvert $U$ autour de $x$ tel que la restriction de  $\mathcal{D}$ à $U$ soit un sous fibré de $TM_{|U}$.\\
Si $\mathcal{D}$ est intégrable, une feuille $N$ telle que $\mathcal{D}$ est maximale pour tout $x\in N$ est appelée une feuille symplectique maximale\index{feuille symplectique!maximale}.
\end{definition}  

L'ensemble $\Sigma$ des points $x$ où $\mathcal{D}$  est maximale est un ouvert et, de plus, on a~:
\begin{proposition}
\label{P_Sigma} 
L'ensemble $\Sigma$ est un ouvert dense de $M$ et la restriction de $\mathcal{D}$ à chaque composante connexe de $\Sigma$ est un sous-fibré adapté fermé de $TM$.\\
En particulier, si  $\mathcal{D}$ est intégrable, toute feuille symplectique faible  maximale est contenue dans l'une des composantes connexes de $\Sigma$ et $\mathcal{D}$ définit un feuilletage régulier de $\Sigma$ en variétés symplectiques faibles.
\end{proposition}
%Notons que  si une structure de Poisson partielle $\left( M,\mathfrak{A} _M \{.,.\}_P \right) $ l'ancre associée $P: T^\flat M \to TM$ est un isomorphisme c'est à dire est une variété partielle symplectique alors les seules les fonctions Casimir sont les fonctions constantes sur chaque composante connexe de $M$. En particulier, toutes les fonctions de Casimir sont globales. Par contre si codimension de de la distribution caractéristique

\subsubsection{Fonctions de Casimir}
\label{____FonctionsDeCasimirPourVarieteDePoissonPartielle}

Les structures symplectiques correspondent à des structures de Poisson non dégénérées. Dans ce cas,  le crochet de Poisson satisfait la condition supplémentaire
\begin{center}
$\forall g \in C^\infty(M), \{f,g\} = 0$ 
\end{center}
si et seulement si $f$ est une fonction constante.\\
Pour des structures de Poisson dégénérées il existe des fonctions non constantes vérifiant cette propriété appelées fonctions de Casimir et correspondent à des constantes du mouvement pour tout champ de vecteurs Hamiltonien.\\ 
On étend alors cette notion aux structures de Poisson partielles.

\begin{definition}
\label{D_FonctionDeCasimir}
Soit $ \left( T^\flat M,M, P, \{.,.\}_P \right) $ une structure de Poisson partielle sur une variété adaptée $M$.\\
Une fonction $C$ de $\mathfrak{A} (U)$ est dite \emph{fonction de Casimir}\index{fonction!de Casimir}\index{Casimir} si 
\[
\forall f \in\mathfrak{A} (U), \; \{ C,f \}_P = 0.
\]
\end{definition}

\begin{proposition}
\label{P_StructureAlgebreSurEnsembleDesFonctionsDeCasimir}
L'ensemble $\operatorname{Cas}_P(U)$\index{CasP@$\operatorname{Cas}(P)$ (ensemble des fonctions de Casimir locales de $P$)} des fonctions de Casimir locales associées à l'ancre de Poisson $P$ a une structure d'algèbre  abélienne qui est le centre de l'algèbre $\mathfrak{A} (U)$.
%\footnote{$\operatorname{Cas}_P(U)$ est le centre\index{centre! de l'algèbre $\mathfrak{A}(U)$} de l'algèbre $\mathfrak{A}(U)$, i.e.
%l'ensemble des éléments qui commutent avec tout élément de $\mathfrak{A}(U)$.}. 
\end{proposition}
\begin{proof}
Il est clair que $\operatorname{Cas}_P(U)$ est un sous-espace vectoriel de $\mathfrak{A} (U)$.\\
Soient $C_1$ et $C_2$, deux fonctions de Casimir et soit $g$ un élément quelconque de $\mathfrak{A}(U)$~; on a donc $\{C_1,g\}=0$ et $\{C_2,g\}=0$. Puisque le crochet $\{.,.\}_P$ est une bi-dérivation de $\mathfrak{A}(M)$, la règle de Leibniz donne
\[
\{C_1.C_2,g\}_P = C_1.\{C_2,g\} + C_2.\{C_1,g\} = 0
\] 
et ainsi le produit $C_1.C_2$ est encore une fonction de Casimir locale. Les autres propriétés découlent naturellement de la définition d'une fonction de Casimir.
\end{proof}

\begin{proposition}
\label{P_FonctionDeCasimirConstanteSurCourbeIntegraleDUnChampHamiltonien} Soit $N$ une sous variété fermée de $M$. Si $N$ est tangente  à la distribution caractéristique
\footnote{i.e. $ \forall x \in N, \; T_x N\subset P(T_x^\flat M)$.} 
alors toute fonction de Casimir $C\in \operatorname{Cas}_P(U)$ est constante sur $U\cap N$.\\
En particulier, si la distribution est intégrable et si chaque feuille est une sous-variété fermée alors la restriction de toute fonction de Casimir locale est constante sur cette feuille.\\
Si un  champ Hamiltonien $X_H = P.dH$ admet des courbes intégrales
\footnote{On rappelle que dans le cadre adapté (ou $c^\infty$-complet.), un champ de vecteurs cinématique lisse n'admet pas nécessairement de courbes intégrales (cf. \cite{KrMi97}, 32.12).}
, 
alors toute fonction de Casimir est constante sur toute les courbes intégrales de $X_H$.
\end{proposition}
\begin{proof}
Soit $C$ une fonction de Casimir et soit $N$ une sous variété fermée  connexe de $M$ tangente  à la distribution caractéristique.  Notons que puisque $N$ est une sous variété fermée, la restriction de $C$ à $N$ est aussi lisse. Pour tout $x\in N$, il existe un ouvert $U$ autour de $x$ dans $M$ tel que l'image $P(T_x^\flat M)$ est engendrée par l'ensemble $\{d_xf,\; f\in\mathfrak{A} (U)\}$. Par suite, si $u\in T_xN$, il existe une fonction $f\in \mathfrak{A}(U)$ telle que $u=P(d_xf)$. Mais $\{f,C\}_P(x)=-<d_xC, u>=0$.  Ce résultat est vrai pour tout $x$ de $N$  et dans le domaine de définition de $C$ et tout $u\in T_xN$. Il en résulte que $C$ est constante sur $N$. Les autres propriétés en découlent clairement.
\end{proof}

Considérons  une structure de Poisson partielle $\left( M,\mathfrak{A} _M \{.,.\}_P \right) $  associée à l'ancre $P: T^\flat M \to TM$. Si $P$ est un   un isomorphisme, c'est-à-dire une variété partielle symplectique alors la distribution $\mathcal{D}$ est trivialement {\bf FSI} et  les seules fonctions Casimir sont les fonctions constantes sur chaque composante connexe de $M$. En particulier, toutes les fonctions de Casimir sont globales. Par contre si la distribution  $\mathcal{D}$ est {\bf FSI} et de codimension strictement positive, alors pour tout ouvert $U$ connexe contenu dans $\Sigma$, la distribution caractéristique est un sous-fibré fermé $\mathcal{D}_U$ de $TM_{ U}$ et par suite l'annulateur $\mathcal{D}^{\operatorname{ann}} _U$ de $\mathcal{D}_U$, i.e. le noyau de $P$ au-dessus de $U$ est aussi un sous-fibré fermé de $T^\flat M_U$. Par conséquent, l'algèbre des fonctions  de Casimir sur $U$ est précisément l'algèbre des fonctions $\operatorname{Cas}_P(U)\subset\mathfrak{A}(U)$  dont la différentielle appartient à  $\mathcal{D}^{\operatorname{ann}}_U$.

\begin{example}
\label{Ex_FonctionDeCasimirSurLeDualDUneAlgebreDeLieDeDimensionn}
{\sf Fonctions de Casimir sur le dual d'une algèbre de Lie de dimension finie.}\\
Pour cet exemple, on pourra consulter  \cite{BiMo91} (voir aussi \cite{SBH11}).\\
Soit $\mathfrak{g}$ une algèbre de Lie semi-simple de dimension $n$ muni du crochet $[.,.]$ et soit $ \left( e_1,\dots,e_n \right) $ une base de $\mathfrak{g}$. Les constantes de structure $ \left( c_{ij}^k \right) _{1 \leqslant i,j,k \leqslant n} $ de l'algèbre de Lie $\mathfrak{g}$ sont définies par
\[
\left[ e_i,e_j \right] = \displaystyle\sum_{k=1}^n c_{ij}^k e_k. 
\]
Si $ \left( x_i \right) _{i \in \{ 1,\dots,n \}}$ sont les coordonnées de l'espace dual $\mathfrak{g}^\star $, le crochet de Lie-Poisson\index{crochet de Lie-Poisson} est donné par
\[
\forall \left( F_1,F_2 \right) \in \left( C^\infty \left( \mathfrak{g}^\star \right)\right) ^2,\;
\{ F_1,F_2\}_{\operatorname{LP}} =
\displaystyle \sum_{i,j,k=1}^n c_{ij}^k x_k \dfrac{\partial F_1}{\partial x_i} \dfrac{\partial F_2}{\partial x_j}.
\]
Tout élément $x$ de $\mathfrak{g}$ définit \emph{l'endomorphisme adjoint} $\operatorname{ad}_x$ de $\mathfrak{g}$ où
\[
\operatorname{ad}_x(y) = [x,y].
\]
La trace de la composition de deux tels endomorphismes définit une forme bilinéaire symétrique $K$ par
\[
K(x,y) = \operatorname{trace} \left( \operatorname{ad}_x  \circ \operatorname{ad}_y  \right)
\]
et appelée \emph{forme de Killing}\index{forme!Killing}\index{Killing!forme}\index{K@$K$ (forme de Killing)}.\\
Son expression en coordonnées locales est
\[
K_{ij} 
= \displaystyle\sum_{k,l=1}^n c_{ik}^j c_{jl}^k.
\]
Puisque l'algèbre de Lie $\mathfrak{g}$ est semi-simple, la forme de Killing est inversible (critère de semi-simplicité de Cartan)\index{Cartan!critère de semi-simplicité} et l'expression de son inverse en coordonnées est noté  $K^{ij}$.\\
La fonction $C:(x_i,x_j) \mapsto \dfrac{1}{2} \displaystyle \sum_{i,j=1}^n K^{ij} x_i x_j$ est une fonction de Casimir \index{fonction de Casimir} pour le crochet de Lie-Poisson, i.e. $\{C,. \}_{\operatorname{LP}}=0$.
\end{example}

\begin{example}
\label{Ex_DistributionCaracteristiquesLimitesDirecteVariétésPoissondimensionHilbertiennesCasimir}{\sf Fonction de Casimir sur une limite directe de variétés de Poisson partielles de Hilbert}\\
Si  $\left( T^\flat M_i,P_i,\{.,.\}_i \right) _{i \in \mathbb{N}}$ est une suite croissante de variétés hilbertiennes de Poisson vérifiant l'hypothèse (1) du Théorème \ref{T_LimiteDirecteDeVarietesDePoisson}, alors la distribution caratéristique $\mathcal{D}$ sur   $M=\underrightarrow{\lim}M_i$ est {\bf FSI}. Par suite, pour tout ouvert $U=\underrightarrow{\lim}U_i$ connexe contenu dans $\Sigma$, on a 
\[
\operatorname{Cas}_P(U)
=\underrightarrow{\lim}\operatorname{Cas}_{P_i} \left( U_i \right) .
\]
\end{example}

\begin{example}
\label{Ex_AlgebroideDeBanachLiePoissonDistributionCaractéristiqueCasimir}
{\sf Fonctions de Casimir  d'une  structure de Poisson partielle  associée à  un algébroïde de Lie-Banach}\\
On considère un algébroïde de Lie-Banach $(E,M,\rho,[.,.]_E)$. Supposons  que $\ker\rho_x$ soit scindé dans la fibre $E_x$ pour tout $x\in M$ et que l'image de $\rho_x$ soit fermée. 
Soit $P:T^\flat E^\star\to TE^\star$ l'ancre de Poisson associée à cet algébroïde (cf. Exemple 
\ref{Ex_AlgebroideDeBanachLiePoisson}). Si, pour tout $x\in M$,   l'espace de Banach $K_x:=\{\xi\circ C_x(\mathsf{w},.),\; \mathsf{w}\in \ker \rho_x\}$ est scindé  dans $E_x^\star$ alors  la distribution caractéristique sur $E^\star$ associée à $P$ est {\bf FSI} (cf. Exemple \ref{Ex_AlgebroideDeBanachLiePoissonDistributionCaractéristique}). Dans ce cas, si $U$ est un ouvert connexe contenu dans $\Sigma$ alors $\operatorname{Cas}_P(E^\star_U)$ est engendrée par $\pi^\star(\operatorname{Cas}_P(U))$ et par l'ensemble des fonctions\footnote{\cite{CaPe23}, 6.3.3.}  
$\{ \Phi_{\mathfrak{a}}\;:\; \mathfrak{a}\in \ker \rho\}$ où
%!+
\[
\Phi_{\mathfrak{a}}(\sigma)
=<\sigma,\mathfrak{a} \circ pi'(\sigma)>
\]
\end{example}

\subsubsection{Structure de Poisson partielles polarisées}
\label{___Structure de Poisson partielles polarisées}
Étant donné une variété Poisson $ M$ connexe  de dimension $n$,  définie par une ancre $P$, dont les feuilles symplectiques maximales sont de dimension $2r$. Un système $\mathbf{F}:=(f_1, \dots,f_s):M\to\mathbb{R}^s$  est \emph{intégrable au sens de Liouville} (cf. \cite{LMV11}) si les conditions suivantes sont satisfaites:
\begin{enumerate}
\item[(1)] 
$f_1,\dots,f_s$ sont indépendantes sur un ouvert dense de $M$~;
\item[(2)] $f_1,\dots,f_s$ sont en involution.
\end{enumerate}
Soient  $\Sigma_{\bf F}$  l'ouvert dense sur lequel sont $f_1,\dots,f_s$ sont indépendantes et $\Sigma$ l'ouvert sur lequel le rang de $P$ est $2r$. Sur l'ouvert dense $
 \Sigma\cap\Sigma_{\bf F}$ les fonctions $f_1,\dots,f_s$ définissent un feuilletage régulier  $\mathcal{T}$ et l'intersection de $\mathcal{T}$ avec chaque feuille symplectique définit 
 un feuilletage Lagrangien (cf. \cite{LMV11}). En fait un tel feuilletage $\mathcal{T}$ est une  polarisation de  variété de Poisson sur $\Sigma\cap\Sigma_{\bf F}$,   c'est-à-dire la donnée d'un feuilletage régulier $\mathcal{T}$ tel que le feuilletage induit sur chaque feuille symplectique est Lagrangien (cf. \cite{AmAw04}).  En effet,  si l'on se donne un feuilletage régulier $\mathcal{T}$ sur $\Sigma\cap\Sigma_{\bf F}$  qui induit sur chaque feuille un feuilletage Lagrangien,  alors, conformément au Theorème de Darboux généralisé (cf. \cite{AMV04}), localement, il existe des fonctions  $f_1,\dots, f_s$ qui définissent ce feuilletage et qui satisfont aux conditions (1),  (2) et (3).
 
En dimension infinie, une  généralisation de l'intégrabilité au sens de Liouville pose de sérieux problèmes. En dimension finie cette notion est fortement liée au Théorème de Darboux et coordonnées  "action-angle". En général dans un contexte de Banach,  le Théorème de Darboux est vrai sur un espace de Banach réflexif pour une forme symplectique forte et, sous des hypothèses fortes, pour une forme symplectique faible (cf. \cite{Bam99}).  Par ailleurs, même dans ce cas strictement Banachique,  la forme symplectique  n'admet toujours pas localement de feuilletages Lagrangiens transverses et n'admet pas de coordonnées locales de type "action-angle". Néanmoins, le lecteur intéressé trouvera dans \cite{OdRa08b} une formulation analytique fonctionnelle  pour l'intégrabilité des système dynamiques semi-infinis de Toda.\\
 
Par contre, en dimension infinie,  la notion de polarisation  peut se définir comme suit:
 
\begin{definition}
\label{D_Polarization}  
Soit $ \left( T^\flat M,M, P, \{.,.\}_P \right) $ une structure de Poisson partielle sur une variété adaptée $M$ dont la distribution caractéristique est $\mathbf{FSI}$. On appelle \emph{polarisation locale}\index{polarisation locale} de cette structure,  sur un ouvert $U$, la donnée d'un  sous fibré intégrable $T^\tau M_U$ de $T^\flat M$ tel que $T^\tau_x M_U$ est une polarisation de l'espace linéaire de Poisson partiel $ \left( T_x M,T_x^\flat M, P_x \right) $ pour tout $x\in U$ (cf. Definition \ref{D_PolarisationLinéaire}). Une polarisation sur $\Sigma$ sera appelée simplement une polarisation de  $ \left( T^\flat M,M, P, \{.,.\}_P \right) $.  
 \end{definition}
 
\begin{example}
\label{Ex_VariétéDimensionFinie}
{\sf Variétés de Poisson de dimension finie}.\\
Tout système Hamiltonien intégrable sur une variété de dimension finie  donne lieu à un feuilletage Lagrangien et donc à une polarisation. Plus généralement, conformément à l'introduction, tout système Hamitonien intégrable sur une variété de  Poisson de dimension finie donne lieu à une polarisation de cette variété de Poisson.
\end{example}
 
\begin{example}\label{Ex_PoissonPartielHilbertien}
{\sf Variété Hilbertienne symplectique partielle.}\\ 
On considère une structure de variété Hilbertienne de Poisson symplectique partielle  $ \left( T^\flat M,P,\{.,.\} \right) $ dont la distribution caractéristique est fermée. Alors cette distribution est {\bf FSI} et soit $\Sigma$ l'ensemble de ses points réguliers. Chaque composante connexe $N$ de  $ \Sigma$, est munie d'une forme symplectique faible $\Omega$  et vérifie donc localement les hypothèses du théorème de Darboux-Bambusi (cf. \cite {Bam99}, \cite{Pel18}). Par suite, au voisinage de chaque point $x\in N$, il existe une carte $(U,\psi)$ telle que $(\psi^{-1})^\star\Omega$ est la restriction à $\Psi(U)$ de la  forme symplectique linéaire $\omega:=(\psi^{-1})^\star(\Omega_x)$ faible sur un espace de Hilbert $\mathbb{M}$. Mais comme $P_x(T^\flat M) $ est fermé dans $T_xM$,  cela implique que $\omega$ est une forme linéaire symplectique forte. D'après \cite{Wei71}, nous avons une décomposition $\mathbb{M}=\mathbb{L}\oplus\mathbb{L}^\star$ où $\mathbb{L}$ est un espace Lagrangien pour $\Omega$. Comme $T\psi(TM_{| U})=U\times \mathbb{ M}$, cela implique qu'il existe un feuilletage Lagrangien $\mathcal{T}$ sur $U$. Par suite $\mathcal{T}$ est une polarisation locale de $N$. De plus, supposons  que  $\mathbb{M}$ est séparable. Il en est de même pour $\mathbb{L}$. Considérons une base orthogonale $\{e_n\}_{n\in \mathbb{N}}$ et si $\{e_n^\star\} _{n\in \mathbb{N}}$ est la base duale, notons $\phi_n(t,u)=e^\star(u)$ la fonction linéaire associée à $e_n^\star$ sur $U$. Si $f_n:=\psi^\star\phi_n$, le fibré tangent à $\mathcal{T}$ est alors l'intersection des noyaux $\bigcap_{n\in \mathbb{N}}\ker df_n$ et on a évidemment $\{f_n,f_m\}_P=0$.  On peut considérer qu'on obtient une généralisation des conditions (1), (2), (3) sur $U$ à ce contexte.
\end{example} 

\begin{example}\label{Ex_PolarisationStructurePoissonPartielleAlgébroïdeLie-Banach} 
{\sf Polarisation d'une  structure de Poisson partielle  associée à  un algébroïde de Lie-Banach}.\\
On considère un algébroïde  Lie-Banach $ \left( E,M,\rho,[.,.]_E \right)$. Supposons  que $\ker\rho_x$ soit scindé dans la fibre $E_x$ pour tout $x\in M$ et que l'image de $\rho_x$ soit fermée.\\
Soit $P:T^\flat E^\star\to TE^\star$ l'ancre de Poisson associée à cet algébroïde (cf. Exemple 
\ref{Ex_AlgebroideDeBanachLiePoisson}). Si pour tout $x\in M$, l'espace de Banach $K_x:=\{\xi\circ C_x(\mathsf{w},.),\; \mathsf{w}\in \ker \rho_x\}$ est scindé  dans $E_x^\star$ alors  la distribution caractéristique sur $E^\star$ associée à $P$ est {\bf FSI} (cf. Exemple \ref{Ex_AlgebroideDeBanachLiePoissonDistributionCaractéristique}). Comme pour tout couple $(f,g)$ de fonctions de $ C^\infty(U)$, on a $\{f\circ \pi^\star,g\circ \pi^\star\}_P=0$  et que pour $U$ assez petit, nous avons 
\[
T_\sigma E^\star\cap P \left( \{d_\sigma (f\circ\pi^\star), f\in C^\infty(U)\} \right) 
=T_\sigma E^\star\cap P(T_\sigma^\flat M)
\]
il en résulte que le feuilletage de $E^\star$ défini par la fibration $\pi^\star$ est une polarisation de la structure  de Poisson  $ \left( T^\flat E^\star, E^\star, P,\{.,.\}_P \right) $.
\end{example}

\subsubsection{Processus de restriction}
\label{____ProcessusDeRestrictionVarietesPoissonPartielles}

Le problème de restriction pour les variétés de Poisson partielle consiste en la recherche de sous-variétés immergées $S$ de la variété de Banach $M$ qui héritent d'une structure de variété de Poisson partielle.\\
On rappelle le résultat obtenu dans \cite{MaMo84} et rappelé dans \cite{MMR85} où l'on impose des conditions de compatibilité entre $S$ et le tenseur de Poisson $P$ qui lui confèrent une structure de variété de Poisson.\\
Pour ce faire, on introduit quelques notations~;
\begin{description}
\item[$\bullet$]
$\mathfrak{X}(S,M)$~:
~algèbre des champs de vecteurs définis sur $S$ à valeurs dans $TM$~;
\item[$\bullet$]
$\mathfrak{X}^\star (S,M)$~:
~espace vectoriel des $1$-formes différentielles définies sur $S$ à valeurs dans $T^\star  M$~;
\item[$\bullet$]
$\mathfrak{X}(S)$~:
~algèbre des champs de vecteurs tangent à $S$~;
\item[$\bullet$]
$\mathfrak{X}(S)^{\operatorname{ann}} $~;
~annulateur de $\mathfrak{X}(S)$, i.e. ensemble des $1$-formes appartenant à $\mathfrak{X}^\star (S,M)$ qui s'annulent sur $\mathfrak{X}(S)$~;
\item[$\bullet$]
$\mathfrak{X}^\star _P(S)$~:
~sous-espace des $1$-formes appartenant à $\mathfrak{X}^\star (S,M)$  dont l'image par $P$ est un vecteur tangent à $S$.
\end{description}

Le couple $( \check{M}, \imath : \check{M} \to M )$ où $\imath$ est une immersion injective constitue une paramétrisation de la sous-variété immergée $S = \imath \left( \check{M} \right) $ de $M$.\\
Cette paramétrisation met en relation les champs de vecteurs et les $1$-formes différentielles définies sur $\check{M}$ avec les objets de même nature définis sur $S$~:
\begin{center}
$ 
X \left( \imath \left( \check{m} \right) \right)
= d \imath \left( \check{m} \right) . \check{X} \left( \check{m} \right)	
$					
\end{center}
\begin{center}
$
\check{\alpha} \left( \check{m} \right) 
= \delta \imath \left( \check{m} \right) . \alpha 
\left( \imath \left(   \check{m} \right) \right).
$
\end{center}
On obtient alors deux applications
\[
d \imath : \mathfrak{X} \left( \check{M} \right)
			\to \mathfrak{X}(S,M)
\textrm{  et  }
\delta \imath : \mathfrak{X}^\star (S,M) 
			\to \mathfrak{X}^\star  \left( \check{M} \right).
\]
On notera que 
\begin{description}
\item[$\bullet$]
$d \imath$ est une application injective dont l'image est l'algèbre $\mathfrak{X}(S)$~;
\item[$\bullet$]
$\delta \imath$ est une application surjective telle que
$
\operatorname{ker} \delta \imath 
= \mathfrak{X}(S)^{\operatorname{ann}} 
$. 
\end{description}
On obtient alors le théorème suivant donnant des conditions suffisantes pour obtenir sur la sous-variété $S$ une structure de variété de Poisson (cf. \cite{MaMo84}, Lemma 6.1 et \cite{MMR85}, Theorem 4.2 pour le cadre des variétés $\operatorname{PN}$).

\begin{theorem}
\label{T_RestrictionVarietePoissonASousVarieteImmergee}
Soit $S$ une sous-variété immergée de la variété de Poisson $M$ où $S$ est paramétrée par $( \check{M}, \imath : \check{M} \to M )$.\\
Supposons que l'on ait
\begin{description}
\item[\textbf{(RP~1)}]
{\hfil
$\mathfrak{X}^\star _P(S) 
+ \mathfrak{X}(S)^{\operatorname{ann}} 
= \mathfrak{X}^\star (S,M)$
}
\item[\textbf{(RP~2)}]
{\hfil
$\mathfrak{X}^\star _P(S)
\cap \mathfrak{X}(S)^{\operatorname{ann}} 
\subset \operatorname{ker}P$
}
\end{description}
$S$ hérite alors une structure de variété de Poisson pour le tenseur
\begin{description}
\item[]
{\hfil
$\check{P} = d\imath^{-1} \circ P \circ \delta \imath|_{\mathfrak{X}^\star _P(S)}^{-1}$
}
\end{description}
\end{theorem}

\begin{proof}
Pour $\alpha \in \mathfrak{X}(S)^{\operatorname{ann}} $ et $\check{X} \in \mathfrak{X} \left( \check{M} \right)$, on a
\[
< \delta\imath (\alpha),\check{X}>
= <\alpha,d\imath \left( \check{X} \right)
=0
\]
puisque $d\imath \left( \check{X} \right) \in \mathfrak{X}(S)$.\\
Ainsi 
\[
\delta \imath \left( \mathfrak{X}^\star _P(S) \right)
= \delta \imath \left( \mathfrak{X}^\star _P(S) 
+ \mathfrak{X}(S)^{\operatorname{ann}}  \right)
\]
D'après \textbf{(RP~1)}, on en déduit que
\[
\begin{array}{rcl}
\delta \imath \left( \mathfrak{X}^\star _P(S) \right) 
		&=&
	\delta \imath \left( \mathfrak{X}^\star (S,M)	\right)	\\
		&=&
	\mathfrak{X}^\star  \left( \check{M} \right)
\end{array}
\]
la dernière égalité résultant de la surjectivité de $\delta \imath$. On obtient alors
\[
\delta \imath \left( \mathfrak{X}^\star _P(S) \right) 
= \mathfrak{X}^\star  \left( \check{M} \right).
\]
Soient maintenant $\alpha_1 \in \mathfrak{X}^\star _P(S)$ et $\alpha_2 \in \mathfrak{X}^\star _P(S)$ telles que
\[
\delta \imath \left( \alpha_1 \right) 
= \delta \imath \left( \alpha_2 \right)
= \check{\alpha} \in \mathfrak{X}^\star  \left( \check{M} \right).
\]
On en déduit trivialement que
\[
\left\{
\begin{array}
[c]{c}
\alpha_2 - \alpha_1 \in \mathfrak{X}^\star _P(S)  \\
\alpha_2 - \alpha_1 \in \operatorname{ker} \delta \imath.
\end{array}
\right.
\]
Puisque $\operatorname{ker} \delta \imath = \mathfrak{X}(S)^{\operatorname{ann}}  $, on en déduit d'après \textbf{(RP~2)} que $\alpha_2 - \alpha_1 \in \operatorname{ker} P$ et ainsi $P \left( \alpha_2 \right) = P \left( \alpha_1 \right) = \check{X} \in \mathfrak{X} \left( \check{M} \right) $.\\

%\begin{figure}[h]
%    \centering
%	 \includegraphics[scale=0.8]{ProcessusRestrictionVarietePoissonSurSousVarieteImmergee.eps}
%    \caption{Processus de restriction d'une variété de Poisson}
%\end{figure}

On définit alors sans ambiguïté le tenseur restreint $\check{P}$ où 
$\check{P} \left( \check{\alpha} \right) = \check{X}$ par 
\[
\check{P} = d\imath^{-1} \circ P \circ \delta \imath|_{\mathfrak{X}^\star _P(S)}^{-1}.
\]
On vérifie que le tenseur $\check{P}$ est un tenseur de Poisson. 
\footnote{\cite{MaMo84}, (6.14)} 
\end{proof}

Afin de généraliser le processus de réduction aux variétés de Poisson partielles, on considère une variété $ \left( M,\mathfrak{A}(M),\{.,.\}_{P} \right) $ de ce type associée au morphisme $P : T^\flat M \to TM$ et on introduit les notations suivantes~:
\begin{description}
\item[$\bullet$]
$\mathfrak{X}^\flat(S,M)$~:
~espace vectoriel des $1$-formes différentielles définies sur $S$ à valeurs dans $T^\flat M$~;
\item[$\bullet$]
$\mathfrak{X}^\flat(S)^{\operatorname{ann}} $~;
~annulateur partiel de $\mathfrak{X}(S)$, i.e. ensemble des $1$-formes appartenant à $\mathfrak{X}^\flat(S,M)$ qui s'annulent sur $\mathfrak{X}(S)$~;
\item[$\bullet$]
$\mathfrak{X}^\flat_P(S)$~:
~sous-espace des $1$-formes appartenant à $\mathfrak{X}^\flat(S,M)$  dont l'image par $P$ est un vecteur tangent à $S$.
\end{description}
En adaptant la démonstration du théorème \ref{T_RestrictionVarietePoissonASousVarieteImmergee} à notre situation, on obtient

\begin{theorem}
\label{T_RestrictionVarietePoissonPartielleASousVarieteImmergee}
Soit $S$ une sous-variété immergée de la variété de Poisson partielle $M$ où $S$ est paramétrée par $( \check{M}, \imath : \check{M} \to M )$.\\
Supposons que l'on ait
\begin{description}
\item[\textbf{(RpP~1)}]
{\hfil
$\mathfrak{X}^\flat_P(S) 
+ \mathfrak{X}^\flat(S)^{\operatorname{ann}} 
= \mathfrak{X}^\flat(S,M)$
}
\item[\textbf{(RpP~2)}]
{\hfil
$\mathfrak{X}^\flat_P(S)
\cap \mathfrak{X}^\flat(S)^{\operatorname{ann}} 
\subset \operatorname{ker}P$
}
\end{description}
$S$ hérite alors une structure de variété de Poisson partielle pour le tenseur
\[
\check{P} = d\imath^{-1} \circ P \circ \delta \imath|_{\mathfrak{X}^\flat_P(S)}^{-1}
\]
\end{theorem}

\subsubsection{Processus de projection}
\label{____ProcessusDeProjectionVarietesPoissonPartielles}
Le problème de projection est de déterminer les distributions intégrables $\mathcal{D}$ d'une variété de Poisson $M$ pour lesquelles la variété quotient $M/\mathcal{D}$
hérite  d'une structure de variété de Poisson.\\
On se réfère à nouveau au document \cite{MaMo84} et à l'article \cite{MMR85}. \\
Soit $ \left( \check{M}, \jmath : M \to \check{M} \right) $ une paramétrisation de $M/\mathcal{D}$ où $\check{M}$ est la variété de paramétrisation et où $\jmath : M \to \check{M}$ est une submersion surjective telles que les feuilles de la distribution sont les fibres de la submersion. \\
Un champ de vecteurs $X$ est dit \emph{projetable}\index{projetable!champ de vecteurs} si pour tout couple $ \left( x_1,x_2 \right) $ de points de la même feuille, on a~:
\[
d \jmath \left( x_1 \right) . X \left( x_1 \right) 
= d \jmath \left( x_2 \right) . X \left( x_2 \right)
\]
%Pour des feuilles connexes, un champ de vecteurs est projetable si sa dérivée de Lie le long d'un vecteur tangent aux feuilles est encore tangent aux feuilles.\\
Un champ de vecteurs projetable définit alors un unique champ de vecteurs $\check{X}$ sur $\check{M}$ par
\[
\check{X} \left( \jmath(x) \right) 
= d \jmath (x) . X(x)
\]
appelé \emph{projeté}\index{projet\'e!champ de vecteurs} de $X$ sur $\check{M}$.\\
Une forme différentielle $\alpha$ est dite \emph{projetable}\index{projetable!forme différentielle} s'il existe une forme différentielle $\check{\alpha} \in \mathfrak{X}^\star  (M)$ telle que pour point $x$ de $M$, on ait~:
\[
\alpha(x) 
= \delta \jmath (x) . \check{\alpha} \left( \jmath(x) \right).
\]
%Pour des feuilles connexes, une forme différentielle est projetable si elle est constante le long des feuilles et si elle s'annule sur l'espace tangent à la feuille.\\
On introduit alors les notations
\begin{description}
\item[$\bullet$]
$\mathfrak{X}_{\mathcal{D}}(M)$~:~sous-espace vectoriel des champs de vecteurs projetables~;
\item[$\bullet$]
$\mathfrak{X}^\star _{\mathcal{D}}(M)$~:~sous-espace vectoriel des formes différentielles projetables.
\end{description}
On obtient alors le résultat suivant (cf. \cite{MaMo84}, Lemma~7.1).
\begin{theorem}
\label{T_ProjectionVarietePNLeLongDistributionIntegrable}
Soit $\mathcal{D}$ une distribution intégrable d'une variété de Poisson partielle $M$. Supposons que les feuilles soient connexes et que l'espace quotient $M/\mathcal{D}$ soit une variété quotient $\check{M}$ où la projection canonique $\jmath : M \to \check{M}$ est une submersion surjective. Si
\begin{description}
\item[\textbf{(PpP)}]
{\hfil
$ P \left( \mathfrak{X}^\star _D(M) \right) 
\subset \mathfrak{X}_D(M) $
}
\end{description}
alors $\check{M}$ hérite de $M$ d'une structure de variété de Poisson partielle  définie par le tenseur
\[
\check{P} = d\jmath \circ P \circ \delta \jmath.
\]
\end{theorem}
 
\section{Couple de variétés hamiltoniennes partielles compatibles}
\label{__CoupleDeVarietesHamiltoniennesPartiellesCompatibles}

La notion fondamentale de structure bihamiltonienne\index{structure!bihamiltonienne} pour les systèmes dynamiques a été introduite en 1978 par F. Magri  dans \cite{Mag78}. Du point de vue géométrique, ceci correspond à l'existence de deux structures de Poisson\index{structure!de Poisson} liées par un opérateur de récursion permettant de générer une suite d'intégrales premières.\\

Les structures de Poisson mises en jeu peuvent être définies à l'aide d'un morphisme antisymétrique $P$ du fibré cotangent $T'M$ dans le fibré tangent $TM$. Ce morphisme est associé à un crochet $\{.,.\}$ sur l'ensemble $C^\infty(M)$ des fonctions lisses sur la variété, i.e. une application bilinéaire antisymétrique $C^\infty(M) \times C^\infty(M) \to C^\infty(M)$ qui satisfait à la fois les identités de Jacobi et de Leibniz.
\\
%!+ Modification : retrait de l'erreur relative au travail de Ratiu qui est réalisé sur l'algèbre $C^\infty(M)$ et non pas sur une sous-algèbre

Comme rappelé en introduction, une généralisation au cas où le tenseur $P$ n'est défini que sur un sous-fibré du fibré cotangent dans le cadre de structures de Banach conduit à la notion de variété de Poisson partielle.

%Dès le cadre d'espaces de Banach non réflexifs, le morphisme $P$ n'est défini que sur un sous-fibré faible $T^\flat M$ de $T'M$, cette situation étant liée, comme introduit par A.~Odzijewicz et T.~Ratiu dans \cite{OdRa03}, à un crochet sur une sous-algèbre de $C^\infty(M)$. On est alors en présence de structures de Poisson partielles. 		\\
 
On est alors amené naturellement à étudier les notions de structures bihamiltoniennes partielles. On se focalise ici, dans le cadre $c^\infty$-complet sur des variétés~:
\begin{itemize}
\item[$\bullet$]
de type $\operatorname{PQ}$ correspondant à un couple $(P,Q)$ de deux structures de Poisson partielles, définies sur le même sous-fibré $T^\flat M$, qui sont compatibles, i.e. pour lesquelle $P+Q$ est encore une structure de même type~;
\item[$\bullet$]
de type $\operatorname{PN}$ correspondant à la donnée d'un tenseur de Poisson $P : T^\flat M \to TM$ et d'un opérateur de Nijenhuis $N$ vérifiant une condition de compatibilité avec $P$~;
\item[$\bullet$]
$\operatorname{P\Omega}$ correspondant à la donnée d'un tenseur de Poisson $P : T\flat M \to TM$ et d'une forme symplectique faible $\Omega : TM \to T'M$ à valeurs dans $T^\flat M$ vérifiant la condition de compatibilité $d(\Omega P \Omega)=0$.  
\end{itemize}

Ces notions sont illustrées par de nombreux exemples  tant en dimension finie qu'infinie.\\ 

$M$ désigne une variété adaptée (ou $c^\infty$-complète) modelée sur l'espace vectoriel adapté $\mathbb{M}$, $p_{M}:TM\to M$ son fibré tangent cinématique et $p_{M}^{\prime}:T^{\prime}M\to M$ son fibré cotangent cinématique. 

\subsection{Variétés $\operatorname{PQ}$ partielles}
\label{__VarietesPQPartielles}

\begin{definition}
\label{D_AncresPoissonPartiellesCompatibles}
Soit $p^{\flat}:T^{\flat}M\to M$ un sous-fibré faible de
$p_{M}^{\prime}:T^{\prime}M\to M$.\\
Soient deux ancres de Poisson partielles $P : T^{\flat}M \to TM$ et $Q : T^{\flat}M \to TM$. $M$ est dite \emph{variété $\operatorname{PQ}$ partielle}\index{PQ@$\operatorname{PQ}$!variété}\index{variété!$\operatorname{PQ}$ partielle} si la somme des ancres $P+Q : T^{\flat} \to TM$  est encore une ancre de Poisson partielle. Dans ce cas, les deux ancres de Poisson sont dites \emph{compatibles}\index{ancres de Poisson partielles compatibles}\index{compatibles!ancres de Poisson partielles}.
\end{definition}

Soient  $\underline{\tilde{P}}$ et $\underline{\tilde{Q}}$ les tenseurs antisymétriques associés à des quasi-ancres de Poisson  $\operatorname{P}$ et $\operatorname{Q}$ respectivement. Alors le crochet de Schouten  $[\underline{\tilde{P}},\underline{\tilde{Q}}]$ est bien défini (cf. \cite{CaPe23}, Chapitre~7, Définition~7.50) et  comme que le degré de $\underline{\tilde{P}}$ et  $\underline{\tilde{Q}}$ est $2$, d'après le Théorème~7.51 (2), le crochet de Schouten $[\underline{\tilde{P}},\underline{\tilde{Q}}]$ est symétrique.

\begin{theorem}
\label{T_CaracterisationCompatibilitePQAvecNulliteCrochet}
Deux ancres de Poisson partielles $P : T^{\flat}M \to TM$ et $Q : T^{\flat}M \to TM$ sont compatibles si et seulement si
le crochet de Schouten $[\underline{\tilde{P}},\underline{\tilde{Q}}]$ est nul.
\end{theorem}
\begin{proof}
Soient $P : T^{\flat}M \to TM$ et $Q : T^{\flat}M \to TM$ deux ancres de Poisson. \\%compatibles.\\
Puisque le crochet de Schouten est bilinéaire symétrique en degré $2$ (comme on l'a déjà vu avant l'énoncé du Théorème), on a:
\begin{eqnarray}
\label{eq_CrochetDeLaSommeDePEtDeQ}
[\underline{\tilde{P}}+\underline{\tilde{Q}},\underline{\tilde{P}}+\underline{\tilde{Q}}]=[\underline{\tilde{P}},\underline{\tilde{P}}]+2[\underline{\tilde{P}},\underline{\tilde{Q}}]+[\underline{\tilde{Q}},\underline{\tilde{Q}}].
\end{eqnarray}
Le fait que $P$ et $Q$ soient des ancres de Poisson se traduit, d'après la relation (\ref{eq_Schouten})
et le théorème
 \ref{T_CaracterisationStructurePartielleDePoisson} on a  $[\underline{\tilde{P}},\underline{\tilde{P}}]=[\underline{\tilde{Q}},\underline{\tilde{Q}}]=0$. La compatibilité de $P$ et de $Q$ se traduit pour les mêmes raisons par $[\underline{\tilde{P}}+\underline{\tilde{Q}},\underline{\tilde{P}}+\underline{\tilde{Q}}]=0$. Le résultat découle alors de la formule (\ref{eq_CrochetDeLaSommeDePEtDeQ}).\\
\end{proof}

\subsubsection{Exemples de variétés de Poisson partielles compatibles}
\label{___ExemplesDeVarietesDePoissonPartiellesCompatibles}

\begin{example}
\label{Ex_CrochetsCompatiblesSurDualAlgebreLieBanachNonNecessairementReflexive}
{\sf Crochets de Poisson compatibles sur le dual d'une algèbre de Banach Lie-Poisson non nécessairement réflexive}\\
Les crochets $\{.,.\}_{\operatorname{LP}}$ (structure linéaire de Lie-Poisson) et $\{.,.\}_{m_0}$ sur le dual $\mathfrak{g}^\star$ de l'algèbre de Lie de Banach non nécessairement réflexive $\mathfrak{g}$ donnés à l'exemple \ref{Ex_CrochetsSurDualAlgebreLieBanachNonNecessairementReflexive}
sont compatibles car leur somme n'est autre que le crochet $\{.,.\}_{\operatorname{LP}}$ translaté de l'origine à $-m_0$.  \\
$\mathfrak{g}^\star$ munies des ancres associées à ces deux crochets est une variété $\operatorname{PQ}$ partielle.
\end{example}

%!+ exemple sur une variété adaptée (non Fréchet)
% d'où l'intérêt d'utiliser le cadre adapté (convenient)

\begin{example}
\label{Ex_StructuresDePoissonPartielleCompatibleSurPinftyM}
On reprend l'espace $P^\infty(M)$ défini à l'Exemple~\ref{Ex_StructureDePoissonPartielleSurPinftyM}.\\
On munit la variété $M$ de deux structures de Poisson compatibles données par des tenseurs $\Lambda_1$ et $\Lambda_2$, ce qui donne en particulier que $\Lambda_1 + \Lambda_2$ est encore de Poisson.\\
On obtient alors sur la variété adaptée $P^\infty(M)$ deux structures de Poisson partielles qui restent compatibles puisque leur somme est encore une structure de même type. 
\end{example}

\begin{example}
\label{Ex_LimiteDirecteDeStructuresDeBanachPoissonCompatibles}
{\sf Limite directe de structures de Banach-Poisson compatibles.}\\
On se réfère à l'exemple \ref{Ex_LimiteDirecteDeStructuresPartiellesDePoissonCompatiblesSurDesVarietesDEBanach}
 et on considère deux suite directes $\left( T^\flat M_i,P_i,\{.,.\}_{P_i} \right) _{i \in \mathbb{N}}$ et $\left( T^\flat M_i,Q_i,\{.,.\}_{Q_i} \right) _{i \in \mathbb{N}}$ de structures de Banach-Poisson partielles.\\
Les relations \textbf{(1)} et \textbf{(3)} sont trivialement vérifiées.
En ce qui concerne la relation \textbf{(2)}, pour $\alpha_i \in T_{x_i}^\flat M_i$, on pose $X_i=P_i \left( \alpha_i \right) $ et $Y_i=Q_i \left( \alpha_i \right) $~; on a par linéarité $T\varepsilon_i^{i+1} \left( X_i+Y_i \right) = T\varepsilon_i^{i+1} \left( X_i \right) + T\varepsilon_i^{i+1} \left( Y_i \right)$ ou encore 
\[
T\varepsilon_i^{i+1} \circ \left( P_i \left( \alpha_i \right) + Q_i \left( \alpha_i \right) \right) 
= T\varepsilon_i^{i+1} \circ P_i \left( \alpha_i \right) + T\varepsilon_i^{i+1}  \circ Q_i \left( \alpha_i \right).
\]
Puisque $\alpha_i = T^\star \varepsilon_i^{i+1} \left( \alpha_{i+1} \right) $, on en déduit que
\[
\begin{array}{cl}
	& T\varepsilon_i^{i+1} \circ \left( P_i \left( T^\star \varepsilon_i^{i+1} \left( \alpha_{i+1} \right)  \right) 
+ Q_i \left( T^\star \varepsilon_i^{i+1} \left( \alpha_{i+1} \right)  \right) \right) \\ 
=	& T\varepsilon_i^{i+1} \circ P_i \left( T^\star \varepsilon_i^{i+1} \left( \alpha_{i+1} \right)  \right) 
+ T\varepsilon_i^{i+1}  \circ Q_i \left( T^\star \varepsilon_i^{i+1} \left( \alpha_{i+1} \right)  \right)
\end{array}
\]
ce qui donne
\[
\begin{array}{cl}
	& T\varepsilon_i^{i+1} \circ 
	\left( 
	\left( P_i + Q_i \right) 
	\left( T^\star \varepsilon_i^{i+1} 
	\left( \alpha_{i+1} \right)  
	 \right)
	 \right) \\ 
=	& T\varepsilon_i^{i+1} \circ P_i \left( T^\star \varepsilon_i^{i+1} \left( \alpha_{i+1} \right)  \right) 
+ T\varepsilon_i^{i+1}  \circ Q_i \left( T^\star \varepsilon_i^{i+1} \left( \alpha_{i+1} \right)  \right)
\end{array}
\]
ou encore
\[
T\varepsilon_i^{i+1} \circ \left( P_i+Q_i \right) \circ T^\star \varepsilon_i^{i+1}
=
T\varepsilon_i^{i+1} \circ P_i \circ T^\star \varepsilon_i^{i+1}
+
T\varepsilon_i^{i+1} \circ Q_i \circ T^\star \varepsilon_i^{i+1}.
\]
On obtient finalement en utilisant la relation \textbf{(2)} pour $P_i$ et $Q_i$ au second membre~:
\[
T\varepsilon_i^{i+1} \circ \left( P_i+Q_i \right) \circ T^\star \varepsilon_i^{i+1}
=
P_{i+1}+Q_{i+1}.
\]
%%Figure
%\begin{figure}[!h]
%    \centering
%	 \includegraphics[scale=0.8]{SuiteDirecteStructuresDePoissonPartiellesCompatibles.eps}
%	\captionsetup{labelformat=empty}
%  	\caption{Suite directe de structures partielles de Banach-Poisson compatibles}
%\end{figure}
Il existe alors un sous fibré faible $p^{\flat}:T^{\flat} M\to M$ de $p'_M : T'M \to M$, où $M = \underrightarrow{\lim} M_i$ est une variété adaptée, et un morphisme $P+Q:T^{\flat
}M\to TM$ (où $P$ et $Q$ sont obtenus ainsi qu'à l'exemple \ref{Ex_LimiteDirecteDeStructuresDeBanachPoissonCompatibles}) tel que $ \left( M,\mathfrak{A}_{P+Q}(M),\{.,.\}_{P+Q} \right) $ est une structure de Poisson partielle et ainsi $M$ est une variété $\operatorname{PQ}$ partielle. 
\end{example}

\subsubsection{Famille de crochets d'une variété $\operatorname{PQ}$ partielle}
\label{___FamilleDeCrochetsDUneVarietePQPartielle}

Soient $P : T^{\flat}M \to TM$ et $Q : T^{\flat}M \to TM$ deux ancres de Poisson compatibles. 
On considère pour $\lambda \in [0,+\infty[ \cup \{\infty\}$ l'application $P_\lambda = P + \lambda Q : T^{\flat}M \to TM $ où $P_\infty = Q$. 
Pour $\lambda$ quelconque dans $[0,+\infty[ \cup \{\infty\}$, le morphisme $P_\lambda = P + \lambda Q$ est aussi une ancre de Poisson.

Par le même type de raisonnement que dans la preuve du Théorème
 \ref{T_CaracterisationCompatibilitePQAvecNulliteCrochet} on obtient:

\begin{proposition}\label{P_FamilleDeCrochetPQPartiellePoisson} Soient $P : T^{\flat}M \to TM$ et $Q : T^{\flat}M \to TM$ deux ancres de Poisson compatibles. Alors  $P_\lambda$ est une ancre de Poisson pour tout $\lambda\in [0,+\infty[ \cup \{\infty\}$.\\
\end{proposition}

\subsubsection{Distributions caractéristiques associées}
\label{____DistributionsCaracteristiquesAssocieesPQ}

Considérons  $p^{\flat}:T^{\flat}M\to M$ un sous-fibré faible de
$p_{M}^{\prime}:T^{\prime}M\to M$ au dessus de la variété de Banach $M$.\\
Soient deux ancres de Poisson partielles $P : T^{\flat}M \to TM$ et $Q : T^{\flat}M \to TM$ compatibles.\\
Soit pour $\lambda \in [0,+\infty[ \cup \{\infty\}$  on pose  $P_\lambda = P + \lambda Q : T^{\flat}M \to TM $ où $P_\infty = Q$.  Même si la distribution caractéristique de $P$ et $Q$ vérifie la propriété {\bf FSI},   en général, la distribution caractéristique de $P_\lambda$ ne vérifie  la propriété  {\bf FSI} même dans le contexte Banachique. En effet, par example si $T$ et $S$ sont deux opérateurs linéaires  d'un espace Banach $\mathbb{E}$ dans un espace de Banach $\mathbb{F}$  avec  noyau  scindé, en général le noyau de $T+\lambda S$, pour tout $\lambda\in [0,+\infty[$ n'est pas scindé. Il en résulte que  les hypothèses du Theorème \ref{T_FeuilletageSurUneVarieteDeBanachPoissonPartielle}  ne sont pas satisfaites en général  pour $P_\lambda$ pour  $\lambda\in [0,+\infty[$. \\
Soit  $\mathfrak{F}(\mathbb{E},\mathbb{F})$ \index{FmathfrakEF@$\mathfrak{F}(\mathbb{E},\mathbb{F})$} (resp. $\mathfrak{F}_+(\mathbb{E},\mathbb{F})$\index{FmathfrakPlusEF@$\mathfrak{F}_+(\mathbb{E},\mathbb{F})$}) l'ensemble  des opérateurs  Fredholm
(resp. semi-Fredholm supérieurs) d'un espace de Banach $\mathbb{E}$ dans un espace de Banach $\mathbb{F}$
\footnote{Un opérateur continu $T:E\to F$  est un opérateur \emph{semi-Fredholm supérieur} si son noyau est de dimension finie  et si son image est fermée. Si, de plus, son image est de codimension finie alors $T$ est \emph{Fredholm}. En particulier, $\mathfrak{F}(\mathbb{E},\mathbb{F})\subset \mathfrak{F}_+(\mathbb{E},\mathbb{F})$.\\
L'\emph{indice} d'un opérateur de Fredholm $T$ est défini par 
$\operatorname{ind}T
= \dim (\ker T) - \operatorname{codim} (\operatorname{im} T).$.
}. 
Rappelons que l'ensemble  $\mathfrak{F}(\mathbb{E},\mathbb{F})$  (resp.  $\mathfrak{F}_+(\mathbb{E},\mathbb{F})$) est ouvert dans l'ensemble $\mathcal{L}(\mathbb{E},\mathbb{F})$ des opérateurs continus de $\mathbb{E}$ dans $\mathbb{F}$ (cf. \cite{Mul07},  Théorème 16.17 et Corollaire 18.2). Par suite, si $T$ est un opérateur de Fredholm (resp. semi-Fredholm supérieur), pour tout opérateur continu  $S$,  il existe $\varepsilon >0$ tel que,  pour tout $\lambda\in [0,\varepsilon[ $, $T+\lambda S$ est  Fredholm  (resp. semi-Fredholm supérieur). Par suite, si $P$ est une  ancre de Poisson partielle qui est un opérateur de  Fredholm   (resp. semi-Fredholm superieur) en restriction à chaque fibre alors les hypothèses du Théorème 
\ref{T_FeuilletageSurUneVarieteDeBanachPoissonPartielle} sont satisfaites pour $\lambda $ assez petit. {\it Malheureusement cet ensemble n'est pas fermé} et on ne peut pas utiliser un argument de connexité dans ce cadre. Néanmoins, nous avons:
\begin{theorem}
\label{T_IntegrabiliteDistributionAssocieeAFaisceauStructuresPQ}
Soient deux ancres de Poisson partielles $P : T^{\flat}M \to TM$ et $Q : T^{\flat}M \to TM$ compatibles sur une variété adaptée $M$. Alors  la distribution caractéristique $\mathcal{D}_{P_\lambda}$ est {\bf FSI} dans les situations suivantes
\begin{enumerate}
\item[(1)]  
$M$ est une variété Banachique et $P_x: T_x^{\flat}M \to T_xM$  appartient à $\mathfrak{F} \left( T_x^\flat M, T_x M \right) $ et $Q_x$ est un opérateur de rang fini  pour tout $x\in M$.
\item[(2)] 
$M$ est une variété Banachique et $P_x: T_x^{\flat}M \to T_xM$ appartient à  $ \mathfrak{F}_+ \left( T_x^\flat M, T_xM \right) \setminus\mathfrak{F} \left( T_x^\flat M, T_x M \right) $ pour tout  $x\in M$
\item[(3)]  
$M = \underrightarrow{\lim} M_i$
\footnote{Dans la contexte de l'Exemple \ref{Ex_LimiteDirecteDeStructuresDeBanachPoissonCompatibles}.} où  $\left( T^\flat M_i,P_i,\{.,.\}_{P_i} \right) _{i \in \mathbb{N}}$ et $\left( T^\flat M_i,Q_i,\{.,.\}_{Q_i} \right) _{i \in \mathbb{N}}$ sont deux suites directes de structures de Banach-Poisson partielles telles que, pour tout $i\in \mathbb{N}$,  $P_i$  et $Q_i$ vérifient  les hypothèses de la situation (1) ou  de la situation (2).
\end{enumerate}
\end{theorem}

\begin{proof} 
 Dans la situation (1), d'après le Théorème de Kato (cf. \cite{Kat58}), pour tout opérateur compact $K: \mathbb{E}\to \mathbb{F}$,  si $T\in \mathfrak{F}(\mathbb{E},\mathbb{F})$, alors  $T+K \in \mathfrak{F}(\mathbb{E},\mathbb{F})$. Comme $\lambda K$ est un opérateur compact pour tout $\lambda\in [0,\infty[$, il s'ensuit que $T+\lambda K \in \mathfrak{F}(\mathbb{E},\mathbb{F})$. Mais l'image d'un opérateur compact n'est pas fermée  s'il n'est pas de rang fini. Pour que les hypothèses du Théorème \ref{T_FeuilletageSurUneVarieteDeBanachPoissonPartielle} soient vérifiées pour $\lambda=\infty$, il faut donc que $Q_x$ soit de rang fini. Dans ce cas ce théorème s'applique  à chaque $P_\lambda$ pour $\lambda\in [0,\infty[\cup\{\infty\}$ ce qui prouve (1).

Dans la situation (2) bien que $\mathfrak{F}_+(\mathbb{E},\mathbb{F})$ ne soit pas fermé,    compte tenu de la preuve du Corollaire 18.2 de \cite{Mul07}, l'ensemble   $\mathfrak{F}_+(\mathbb{E},\mathbb{F})\setminus \mathfrak{F}(\mathbb{E},\mathbb{F})$ est fermé dans $\mathcal{L}(\mathbb{E},\mathbb{F})$. Pour $x\in M$ fixé, si $ P_x\in \mathfrak{F}_+(T_x^\flat M, T_xM)\setminus\mathfrak{F}(T_x^\flat M, T_xM)$,  nous avons vu qu'il existe $
\varepsilon > 0$ tel que $P_x+\lambda Q_x\in \mathfrak{F}_+(T_x^\flat M, T_xM)$. Comme  $\mathfrak{F}_+(\mathbb{E},\mathbb{F})\setminus \mathfrak{F}(\mathbb{E},\mathbb{F}) $ est fermé, 
 $\{\lambda\in [0,\varepsilon[\;: \;  P_x+\lambda Q_x\in \mathfrak{F}_+(T_x^\flat M, T_xM)\setminus\mathfrak{F}(T_x^\flat M, T_xM)\}$ est fermé et
\[
\exists \lambda_0 \in [0,\varepsilon[: \;
\forall \lambda\in [0,\lambda_0], \;
P_x+\lambda Q_x \in \mathfrak{F}_+ (T_x^\flat M, T_xM)\setminus\mathfrak{F}(T_x^\flat M, T_xM).
\]
De plus, puisque $\mathfrak{F}(\mathbb{E},\mathbb{F})$ est ouvert, il existe un intervalle $]\lambda_0,\lambda_1[ \subset  [0,\varepsilon[$ tel que, sur cet intervalle, $P_x+\lambda Q_x\in \mathfrak{F}(T_x^\flat M, T_xM)$. Mais alors, pour toute suite $\lambda_n \in ]\lambda_0,\lambda_1[$ qui converge vers $\lambda_0$ la suite d'opérateurs $ \left( P_x+\lambda_n Q_x \right) $ converge vers $P_x+\lambda_0 Q_x$. Comme l'indice des opérateurs de Fredholm est constant sur chaque composante connexe de $\mathfrak{F}(T_x^\flat M, T_xM)$, il en résulte que l'indice de l'opérateur $P_x+\lambda_0 Q_x$ est fini et donc ne peut pas appartenir à  $\mathfrak{F}_+(T_x^\flat M, T_xM)\setminus\mathfrak{F}(T_x^\flat M, T_xM)$ d'où une contradiction. \\
Cela implique donc que  $P_x+\lambda Q_x$ appartient à  $\mathfrak{F}_+(T_x^\flat M, T_xM)\setminus\mathfrak{F}(T_x^\flat M, T_xM)$ pour tout $\lambda\in [0,\varepsilon[$ et comme cet 
ensemble est fermé,  le résultat est encore vrai pour $\lambda\in [0,\epsilon]$.  Le même type de raisonnement permet de montrer que l'ensemble des  $\lambda\in [0,\infty[$ tel que  $P_x+\lambda Q_x$ appartient à  $\mathfrak{F}_+(T_x^\flat M, T_xM)\setminus\mathfrak{F}(T_x^\flat M, T_xM)$ est à la fois ouvert et fermé.  Par connexité, cette propriété est donc vraie pour tout $\lambda\in [0,\infty[$.\\
Par suite,   on peut donc  appliquer le Théorème  \ref{T_FeuilletageSurUneVarieteDeBanachPoissonPartielle} à chaque $P_\lambda$ pour $\lambda\in [0,\infty[\cup\{\infty\}$.

Dans la situation (3), en se référant au contexte de l'Exemple \ref{Ex_LimiteDirecteDeStructuresDeBanachPoissonCompatibles}, on applique le Théorème \ref{T_LimiteDirecteDeVarietesDePoisson}.
\end{proof}
 
\begin{remark}
\label{R_Faisceau2dimensions}  
\'Etant données deux structures de Poissons compatibles  $P$ et $Q$ sur une variété $M$ de  dimension finie, certains  auteurs considèrent un faisceau de type $\lambda P+\mu Q$ à deux paramètres réels $(\lambda,\mu) \in \mathbb{R}\setminus \{(0,0)\}$.   Si $P$ et $Q$ sont des ancres de Poisson partielles sur une variété adaptée $M$, on pourrait aussi définir une ancre de Poisson partielle à deux  paramètres $\lambda P+ \mu Q$ sur $M$. Mais   pour  $\lambda\not=0$ (resp. $\mu\not=0$) le faisceau $\lambda P+ \mu Q$ a les mêmes propriétés que $P+\frac{\mu}{\lambda} Q$ (resp. $P+\frac{\lambda}{\mu} Q$ ). En fait,  si $P$ et $Q$ sont compatibles,  le même type de raisonnement que dans le paragraphe \ref{___FamilleDeCrochetsDUneVarietePQPartielle} entraine que, pour tout $(\lambda,\mu) \in \mathbb{R}\setminus \{(0,0)\}$, $\lambda P+ \mu Q$  est aussi une ancre de Poisson partielle sur $M$. A noter que  les mêmes types de raisonnement que dans la preuve  du Théorème \ref{T_IntegrabiliteDistributionAssocieeAFaisceauStructuresPQ} permet de montrer que dans les même situations la distribution caractéristique de  $\lambda P+ \mu Q$ est aussi {\bf FSI} pour tout $(\lambda,\mu) \in \mathbb{R}\setminus \{(0,0)\}$.
\end{remark}
 
\begin{example}
\label{Ex_LimiteDirecteDeStructuresVariétésfiniesPoissonCompatibles}
{\sf Suite croissante de variétés de Poisson compatibles de dimension finie}.\\
Si  $\left( T^\flat M_i,P_i,\{.,.\}_i \right) _{i \in \mathbb{N}}$ et $\left( T^\flat M_i,Q_i,\{.,.\}_i \right) _{i \in \mathbb{N}}$  est une suite croissante de variétés de Poisson compatibles, alors  pour tout entier $i$, les situations  (1) et (2)  du Théorème \ref{T_IntegrabiliteDistributionAssocieeAFaisceauStructuresPQ} sont vérifiées et ainsi  la situation (3) est aussi  vérifiée. Ainsi sur  $M = \underrightarrow{\lim} M_i$, on a un faisceau $P_\lambda$ de  structures de Poisson partielles dont la distribution caractéristique est {\bf FSI}.
\end{example}

\begin{example}\label{Ex_FaisceauPoissonAssociéDeformationAlgebroide}
{\sf Faisceaux d'ancres de Poisson partielles associé à une déformation d'algébroides de Lie}.\\
Rappelons qu'une déformation d'un algébroïde de Lie $ \left( E,M,\rho,[.,.]_E \right) $ est la donnée d'une famille à un paramètre  $ \{ \left( E,M, \rho_\lambda, [.,.]_\lambda \right) \}\lambda\in [0,\infty[$  de structure 
d'algebroïdes sur $E$ telle que $[.,.]_0=[.,.]_E$ (cf.\cite{CrMo08}). Cette définition s'étend naturellement aux algebroïdes adaptés. Par ailleurs, un tenseur $N$ de $E$ de type $(1,1)$ est un tenseur de Nijenhuis si 
\[
T(N)(s,s'):=[Ns,Ns']_E-N[Ns,s']_E- N[s,Ns']_E+N^2[s,s']_E=0
\]
où $T(N)$ est la torsion de $N$ (cf. \ref{____TenseurDeNijenhuis}).\\
Dans ces conditions, comme en dimension finie,  on peut définir une structure d'algébroïde de Lie $ \left( E,M,\rho_N:=\rho\circ N,[.,.]_N \right) $ où 
\begin{center}
$[s,s']_N:=[Ns,s']_E+[s,Ns']_E-[Ns,s']_E$.
\end{center} 
(cf. \cite{CrFe11} par exemple). 
On pose $N_\lambda=\operatorname{Id}+\lambda N$  pour tout $\lambda\in [0,\infty[$. Considérons le crochet 
$[s,s']_{N_\lambda}$. Comme $T(N_\lambda)(s,s')=0$ pour tout $\lambda \in [0,+\infty[$, il en résulte que $ \left( E,M, \rho_{N_\lambda}, [.,.]_{N_\lambda} \right) $ est un algébroïde de Lie. 
Soit $P$ (resp. $P_N$) l'ancre de Poisson sur $T^\flat E^\star$ associée à la structure d'algébroïde  de Lie $ \left( E,M,\rho,[.,.]_E \right) $ (resp.  $ \left( E,M,\rho_N,[.,.]_N \right) $. Alors $P_\lambda=P+\lambda P_N$ , $\lambda\in [0,\infty[\cup\{\infty\}$ est un faisceau d'ancres de Poisson sur $T^\flat E^\star$. 
En particulier, si la fibre  $\mathbb{E}$ de $E$ est un espace de Hilbert ou si $N$ et $P$ sont des opérateurs semi-Fredholm alors la distribution caractéristique de $P_\lambda$ est {\bf FSI}.
\end{example}

\subsubsection{Fonctions de Casimir et chaîne de Magri-Lénard} 

Soit pour $\lambda \in [0,+\infty[ \cup \{\infty\}$, l'ancre de Poisson $P_\lambda = P + \lambda Q$ où $P : T^{\flat}M \to TM$ et $Q : T^{\flat}M \to TM$ sont deux ancres de Poisson compatibles.  Pour tout $\lambda,\in [0,+\infty[\cup\{\infty\}$, considérons le faisceau associé $P_\lambda$. On suppose dans ce paragraphe que la distribution caractéristique $\mathcal{D}_{P_\lambda}$ de $P_\lambda$ est {\bf FSI} et supplémentée \footnote{Cette dernière hypothèse est toujours vérifiée si $\mathbb{E}$ est un espace de Hilbert  et  aussi dans les situations (2) et (3) du Théorème \ref{T_IntegrabiliteDistributionAssocieeAFaisceauStructuresPQ}.}, et  on note $\Sigma_\lambda$ l'ouvert  dense sur lequel cette distribution est maximale.

\begin{proposition}\label{P_FonctionsCasimirFaisceau} 
Pour tout  $x_0\in \Sigma_{\lambda}$, sous les hypothèses précédentes,  il existe un voisinage  $I\times U\subset\mathbb{R}\times \Sigma_{\lambda_0}$ de $(\lambda_0,x_0)$ tel que pour tout $(\lambda,x) \in I\times U$  on a:\\
toute  fonction Casimir $f$ sur $U$  pour $P_{\lambda_0}$ s'étend en une fonction de Casimir $f_\lambda$  de $P_\lambda$ sur $U$,  pour tout $\lambda\in I$.\\
 \end{proposition}

\begin{proof} 
Sur chaque composante connexe de $\Sigma_{\lambda}$ la distribution caractéristique $\mathcal{D}_{P_{\lambda}}$  définit un sous fibré intégrable $E_\lambda$ sur  $\Sigma_\lambda$ dont chaque feuille est symplectique. Fixons $x_0\in \Sigma_{\lambda_0}$. Puisque l'application $\lambda \mapsto P_\lambda$ est linéaire, il existe un voisinage   $I$ de $\lambda_0$  tel que $x_0$ appartient à $\Sigma_\lambda$ pour tout $\lambda\in I$. Par ailleurs, pour $\lambda_0$ fixé, il existe une décomposition $\mathbb{M}=\mathbb{E}\times\mathbb{T}$ et une carte feuilletée $(U,\psi)$ autour de $x_0$, c'est-à-dire que  le feuilletage induit sur $U$ est l'image réciproque du feuilletage trivial sur $\mathbb{M}$ dont le feuilles sont du type $\mathbb{E}\times\{t\}$ avec  $t\in \mathbb{T}$. Sans perte de généralité, on peut supposer que  $M=U:= V\times W\subset \mathbb{E}\times\mathbb{T}$,  $x_0=(0,0)\in \mathbb{E}\times \mathbb{T}$ et $U$ est simplement connexe. \\
 
Puisque l'application $\lambda\mapsto P_\lambda$ est linéaire, quitte à restreindre $U$ et $I$, il en résulte que $(\lambda, x)\mapsto \mathcal{D}_{P_\lambda}(x)$ est une application lisse $I\times U$ dans 
la grassmannienne des sous espaces de $\mathbb{M}$ isomorphes à $\mathbb{E}$ et transverses à $\{0\}\times \mathbb{T}$. Cela implique que  pour tout  $\lambda\in I$, l'ouvert $U$ est contenu dans  $\Sigma_\lambda$.  Comme $U$ est simplement connexe, le fibré $E_\lambda$ est trivialisable et, quitte à restreindre $U$, il existe une carte feuilletée $(U,\Psi_\lambda)$ telle que si $
\mathbb{E}_\lambda$ est le sous espace $\Psi_\lambda(0)$ alors le feuilletage symplectique associé à $E_\lambda$ sur $U$ est l'image inverse du feuilletage $\{\mathbb{E}_\lambda\times\{t\}\}_{t\in \mathbb{T}}$. D'autre part, pour chaque $x=(u,t)\in V\times W$, il existe une famille d'isomorphismes linéaires $T_{(x,\lambda)}$ 
%$\mathbb{E}\times\mathbb{T}$ 
 qui fixent $\mathbb{T}$  et tels 
que $T_{(x,\lambda)}(\mathbb{E})=\mathbb{E}_\lambda$. Dans ces mêmes conditions, si $\mathbb{M}^\flat$ est la fibre type de $T^\flat M$,  la restriction de  $T^\flat M$ à $U$ est $U\times 
\mathbb{M}^\flat\subset U\times (\mathbb{E}_\lambda^{\operatorname{ann}}\times \mathbb{T}^{\operatorname{ann}})$. Comme $T_{(x,\lambda)}(\mathbb{E})=\mathbb{E}_\lambda$, on a 
\begin{equation}
\label{Eq_PlambdaP}
P_\lambda(x)
=
T_{(x,\lambda)}\circ P_{\lambda_0}\circ T_{(x,\lambda)}^\star.
\end{equation}
Rappelons qu'une fonction  admissible $f$ est une  fonction de Casimir pour $P_\lambda$  si et seulement si sa différentielle appartient à $\ker P_\lambda(x)$ (cf. $\S$~\ref{____FonctionsDeCasimirPourVarieteDePoissonPartielle}). Par suite de (\ref{Eq_PlambdaP}), il résulte que si $d_xf$  appartient à $\ker P_{\lambda_0}(x)$ alors 
$d \left( f\circ T_{(x,\lambda)}^{-1} \right) $ appartient à $\ker P_\lambda(x)$, ce qui achève la démonstration en prenant 
$f_\lambda= f\circ T_{(x,\lambda)}^{-1} $.
\end{proof}

\emph{Envisageons le principe de construction d'une chaîne de récursion de Magri-Lenard}. \\
On considère, si elle existe, une fonction de Casimir $H_\lambda$ associé à l'ancre $P_\lambda$.
On a donc, pour toute fonction $f \in\mathfrak{A} $, $\{ H_\lambda,f \}=0$.\\
Supposons que $H_\lambda$ puisse être écrit sous forme d'une série
\[
H_\lambda = H_0 + H_1 \lambda + H_2 \lambda^2 + \cdots
\]
En utilisant les même arguments qu'en dimension finie, on peut vérifier que $H_0$ est une fonction de Casimir de $P$ et que, pour tout entier naturel $k$, le champ de vecteurs Hamiltonien associé à $H_{k+1}$ relativement à $P$ coïncide avec le champ Hamiltonien associé à $H_k$ relativement à $Q$.\\
De plus, tous les Hamiltoniens $H_k$ sont en involution pour les crochets associés à $P$ et $Q$ et les champs de vecteurs Hamiltoniens associés commutent deux à deux. 

\begin{example}
On reprend les éléments de l'Exemple \ref{Ex_CrochetsCompatiblesSurDualAlgebreLieBanachNonNecessairementReflexive}. 
On considère alors pour $P$ le tenseur de Poisson constant associé au crochet $\{.,.\}_{m_0}$ et pour $Q$ le tenseur de Lie-Poisson associé au crochet $\{.,.\}_{\operatorname{LP}}$.\\
Pour obtenir la fonction de Casimir $H_\lambda$, on choisit une fonction de Casimir $H$ de la structure de Lie-Poisson et on utilise la méthode de translation de l'argument (cf. \cite{Kole07}, 4.5)
\[
H_\lambda (m) = H \left( m_0 + \lambda m \right).
\]
Cette technique peut être appliquée à l'équation de KdV vue comme système Hamiltonien sur le dual de l'algèbre de Virasoro (cf. \cite{KhMi03}).
\end{example} 

\subsubsection{Fonctions de Casimir et feuilletages pour l'équation de KdV}
\label{___Ex_FonctionDeCasimirFeuilletagesPourEquationKdVSurTourDeHilbert}
Dans \cite{KaMa98}, Kappeler et Makarov font apparaître l'équation de KdV sur le cercle $\mathbb{S}^1$ comme un système bihamiltonien sur la tour de Hilbert d'espaces de Sobolev $ \left( H_n \left( \mathbb{S}^1 \right)  \right) $ où 
\[
H^{n} \left( \mathbb{S}^{1} \right) 
=
\left\{ 
q\in L^{2} \left( \mathbb{S}^{1} \right):\;
\forall k\in \left\{ 0,\dots,n\right\},
q^k\in L^{2} \left( \mathbb{S}^{1} \right)  
\right\}
\]
%Une base orthonormale pour cette tour de Hilbert est alors 
%$ \left( e_{0},e_{1},e_{-1},\dots,e_{k},e_{-k},\dots \right)  _{k\in\mathbb{N}}$ où $e_{k}:x\mapsto e^{i2k\pi x}$.\\
L'opérateur
\[
\begin{array}
[c]{cccc}
\partial_{x}: & U\cap H_{n} & \longrightarrow & \mathcal{L}\left(
H_{n+1},H_{n}\right)  \\
& q & \longmapsto & \left(  \partial_{x}\right)  _{q}%
\end{array}
\]
où $U = H_{0} = H^{0} \left( \mathbb{S}^{1} \right) $ correspond au premier tenseur de Poisson
\footnote{
L'équation de KdV\index{KdV equation ($u_t=-u_{xxx}+6u u_{xx}$)}
\index{utKdV@$u_t=-u_{xxx}+6u u_x$ (équation de KdV)}\index{equation!KdV ($u_t=-u_{xxx}+6u u_{xx}$)} 
\[
\partial_{t}u=-\partial_{x}^{3}u+6u\partial_{x}u
\]
apparaît comme un \emph{système complètement intégrable} au sens suivant~: il existe une suite de variables action-angle\index{variables action-angle} et le système peut être linéarisé (cf. \cite{Zub97} et \cite{KoKu16})~: 
\[
\left\{
\begin{array}
[c]{rcl}
A_1		&=&		0	\\
I_t     &=&     0   \\
\phi_t  &=&     \dfrac{\partial}{\partial I} K(I)
\end{array}
\right.
\]
où $K$ correspond à l'Hamiltonien.\\
Le lecteur trouvera dans \cite{LMV11} un travail sur le théorème action-angle dans le contexte des systèmes intégrables sur des variétés de Poisson de dimension finie.
}
 qui admet pour fonction de Casimir la moyenne 
\[
 A_1 : u \mapsto \int_{\mathbb{S}^1} u(x) dx
 \]
Les feuilles symplectiques pour le feuilletage correspondant  sont alors données par les espaces affines
\[
H_c^n \left( \mathbb{S}^1 \right) 
=
\{
q \in H^n \left( \mathbb{S}^1 \right):\; A_n(q)=c
\}
 \]
D'autre part, l'opérateur
\[
\begin{array}
[c]{cccc}
L: & U \cap H_{n} 	& \longrightarrow & 
			\mathcal{L} \left( H_{n+2},H_{n-1} \right)  \\
		& q 			& \longmapsto &
			 L_q
\end{array}
\]
où 
\[
L_q = -\dfrac{1}{2}\partial_x^3 + q\partial_x+\partial_x q
\]
correspond au second tenseur de Poisson pour cette EDP.\\
Pour déterminer une fonction de Casimir pour ce second tenseur, on considère l'équation de Schrödinger
\[
-y''  + qy = \lambda y
\]
et l'on note, par abus, $y_1(x,\lambda,q)$ et $y_2(x,\lambda,q)$ les solutions fondamentales de cette équation.\\
Le discriminant correspondant est alors
\[
\Delta(\lambda,q) = y_1(1,\lambda,q) + y_2'(1,\lambda,q)
\]
$A_2:q \mapsto \Delta(0,q)$ est alors une fonction de Casimir pour cette seconde structure de Poisson.\\
Le feuilletage associé n'est pas régulier et admet des feuilles singulières qui peuvent être décrites à partir du spectre de l'opérateur $-d_x^2 + q$ (cf. \cite{KaMa98}, Theorem 0.1).

\subsection{Variétés $\operatorname{PN}$ partielles}
\label{___VarietesPNPartielles}

La notion de variété $\operatorname{PN}$ introduite par Y. Kosmann-Schwarzbach et  F. Magri dans \cite{KSMag90} constitue le cadre naturel pour les systèmes Hamiltoniens intégrables, notamment sur les duaux d'algèbre de Lie. Ces systèmes peuvent être décrits \textit{via} un champs de vecteurs $X$ laissant invariant à la fois $P$ et $N$~:
\[
L_X P = 0	\textrm{  et  }		L_X N = 0
\] 

\subsubsection{Tenseur de Nijenhuis}
\label{____TenseurDeNijenhuis}

La torsion d'un tenseur $T(N)$ d'un tenseur $N : \mathfrak{X}(M) \to \mathfrak{X}(M)$ de type $(1,1)$ sur la variété adaptée $M$ est défini, comme dans \cite{MaMo84}, (2.6), pour tout couple $(X,Y)$ par
\begin{equation}
\label{eq_Torsion1-1Tenseur}
T(N)(X,Y) = [NX,NY] - N \left( [NX,Y]+[X,NY]-N[X,Y] \right)
\end{equation} 
Le tenseur $N$ de type $(1,1)$ est appelé \emph{tenseur de Nijenhuis}\index{tenseur de Nijenhuis}\index{Nijenhuis!tenseur} si $T(N)=0$.

\subsubsection{Variétés $\operatorname{PN}$ partielles adaptées}
\label{____VarietesPNPartiellesAdaptees}
 
\begin{definition}
\label{D_VarietePNPartielle}
Une variété $\operatorname{PN}$ partielle adaptée\index{variété!$\operatorname{PN}$ partielle} est une variété de Poisson partielle $ \left( M,\mathfrak{A}(M),\{.,.\}_{P} \right) $ muni d'un tenseur de Nijenhuis $N$ vérifiant les conditions
\begin{description}
\item[\textbf{(pPN~1)}]
$N \circ P = P \circ N^t$
\item[\textbf{(pPN~2)}]
$R(P,N) = 0$
\end{description}
où le concomitant $R(P,N)$ est défini pour toute section $\alpha \in \Gamma \left( T^{\flat}M \right) $ et tout champ de vecteurs $X \in \operatorname{im} P$ par
\footnote{\cite{MaMo84}, (B.3.5) ou \cite{MMR85}, (2.5).}
\begin{equation}
\label{eq_pR(P,N)}
R(P,N)(\alpha,X) 
= L_{P(\alpha)}(N)X 
- P \left( L_X \left( N^t \alpha\right) \right) 
+ P \left( L_{NX} \alpha \right) .
\end{equation}
\end{definition}

Si $M$ est une variété $\operatorname{PN}$ alors $P_1=NP$ est un tenseur de Poisson couplé à $N$. En itérant le processus on obtient une suite $ \left( P_k \right) _{k \in \mathbb{N}^\star}$ de tenseurs de Poisson (où $P_k = N^k P$) en involution, i.e. $\left[ P_i,P_j \right] = 0$ en adaptant le résultat \cite{MaMo84}, B.3, iii). 

\subsubsection{Exemples}
\label{____Ex_VarietesPNPartielles}

\begin{example}
\label{Ex_VarietesPN}
{\sf Variétés $\operatorname{PN}$.}\\
Les variétés $\operatorname{PN}$ de dimension finie ou encore modelées sur des espaces de Banach sont bien entendu des variétés $\operatorname{PN}$ partielles. 
\end{example}

\begin{example}
\label{Ex_StructurePNPartielleInvarianteAGaucheSurUnGroupeDeLie}
{\sf Variétés $\operatorname{PN}$ partielles invariantes à droite sur un groupe de Lie.}\\
La notion de structure $\operatorname{PN}$ invariante à droite sur un groupe de Lie $G$ modelé sur une structure de Banach est défini de manière naturelle (cf. \cite{RRH18}). Une structure $\operatorname{PN}$ partielle est obtenue en restreignant la fibre $\mathfrak{g}^\star $ du fibré cotangent $T^\star  G$ à un sous-espace vectoriel de $\mathfrak{g}^\star $.  
\end{example}

\subsubsection{Processus de restriction}
\label{___ProcessusDeRestrictionDUneVarietePNPartielle}
On adapte le processus de restriction développé en \ref{____ProcessusDeRestrictionVarietesPoissonPartielles}
pour les variétés de Poisson partielles aux variétés $\operatorname{PN}$ partielles tout en conservant les notations introduites.

\begin{theorem}
\label{T_RestrictionVarietePoissonPartielleASousVarieteImmergee}
Soit $S$ une sous-variété immergée de la variété $\operatorname{PN}$ partielle $M$ où $S$ est paramétrée par $( \check{M}, \imath : \check{M} \to M )$.\\
Supposons que soient vérifiées les relations \emph{\textbf{(RpP~1)}} et \emph{\textbf{(RpP~2)}} ainsi que
\begin{description}
\item[\textbf{(RpN)}]
{\hfil
$N \left( \mathfrak{X}(S) \right) \subset \mathfrak{X}(S)$
}
\end{description}
$S$ hérite alors d'une structure de variété $\operatorname{PN}$ partielle pour les tenseurs
\[
\check{P} = d\imath^{-1} \circ P \circ \delta \imath|_{\mathfrak{X}^\flat_P(S)}^{-1}
\]
et
\[
\check{N} = d\imath^{-1} \circ N \circ d\imath
\]
\end{theorem}

On vérifie en sus que $\check{N}$ est un tenseur de Nijenhuis 
\footnote{\cite{MaMo84}, (6.13).}
et que $\check{P}$ et $\check{N}$ sont compatibles
\footnote{\cite{MaMo84}, (6.15) et (6.16).}.

\subsubsection{Processus de projection}
\label{___ProcessusProjectionDUneVarietePNPartiellee}

On adapte ici le processus de projection développé en \ref{____ProcessusDeProjectionVarietesPoissonPartielles}
pour les variétés de Poisson partielles aux variétés $\operatorname{PN}$ partielles tout en conservant les notations introduites.

\begin{theorem}
\label{T_ProjectionVarietePNPartielleLeLongDistributionIntegrable}
Soit $\mathcal{D}$ une distribution intégrable d'une variété $\operatorname{PN}$ partielle $M$. Supposons que les feuilles soient connexes et que l'espace quotient $M/\mathcal{D}$ soit une variété quotient $\check{M}$ où la projection canonique $\jmath : M \to \check{M}$ est une submersion surjective. Si la propriété \emph{\textbf{(PpP)}} est vérifiée et si, de plus, on a
\begin{description}
\item[\textbf{(PpN)}]
{\hfil
$ N^\star  \left( \mathfrak{X}^\star _D(M) \right) 
\subset \mathfrak{X}^\star _D(M)$
}
\end{description}
alors $\check{M}$ hérite de $M$ d'une structure de variété $\operatorname{PN}$ partielle  définie par les tenseurs
\[
\check{P} = d\jmath \circ P \circ \delta \jmath\]
et
\[
\check{N}^\star  = d\jmath^{-1} \circ N^\star  \circ d\jmath
\]
\end{theorem}

\subsection{Variétés $\operatorname{P\Omega}$ partielles}
\label{__VarietesPOmegaPartielles}

Adaptant la définition de tenseur présymplectique donnée dans  \cite{MaMo84},~2 à notre contexte $c^\infty$-complet et voyant, par abus, un tenseur $2$ fois covariant\footnote{
$ \omega(X,Y) = \left\langle \Omega X,Y \right\rangle $ où $\omega : \mathfrak{X}(M) \times \mathfrak{X}(M) \to C^\infty(M)$.
} comme un morphisme de $TM$ dans $T'M$ on introduit~:
\begin{definition}
\label{D_FormeSymplectiqueFaible}
Un tenseur $\Omega : TM \to T'M$ de type $(0,2)$  antisymétrique injectif
\footnote{Classiquement, $\Omega : TM \to T'M$ est antisymétrique si $ \Omega + \Omega^\star = 0$.}  est appelé \emph{forme symplectique faible}\index{forme symplectique faible} si $d\Omega=0$ où pour tout couple de champ de vecteurs $(X,Y)$, on a
\[
d\Omega(X,Y)
=
L_X(\Omega)Y - L_Y(\Omega)X - \Omega([X,Y]) + d \left\langle \Omega X, Y \right\rangle
\]
\end{definition}

\begin{definition}
\label{D_AncresPoissonPartiellesCompatibles}
Soit $p^{\flat}:T^{\flat}M\to M$ un sous-fibré faible de $p_{M}^{\prime}:T^{\prime}M\to M$.\\
Soient $P : T^{\flat}M \to TM$ une ancre de Poisson partielle et $\Omega : TM \to T'M$ une forme symplectique partielle à valeurs dans $T^{\flat}M$. 
$M$ est dite \emph{variété $\operatorname{P\Omega}$ partielle}\index{variété!$\operatorname{P\Omega}$ partielle} si l'on a la condition de compatibilité
\begin{eqnarray}
\label{eq_dOmegaPOmegaNul}
d(\Omega P \Omega) =0
\end{eqnarray}
\end{definition}

\subsubsection{Exemples}
\label{Ex_VarietesPOmegaPartielles}

\begin{example}
{\sf Structure $\operatorname{P\Omega}$ sur un groupe de Lie admettant un tenseur de Poisson invariant à droite et une forme faiblement symplectique invariante à gauche.}\\ 
Soit $G$ un groupe de Banach-Lie de neutre $e$~; on désigne par $\mathfrak{g}$ son algèbre de Lie et par $\mathfrak{g}^\star$ le dual de $\mathfrak{g}$.\\
Pour $g \in G$, $L_g : h \mapsto g.h$ (resp. $R_g : h \mapsto h.g$) désigne la translation à gauche (resp. à droite).\\
Toute application linéaire $\omega : \mathfrak{g} \to \mathfrak{g}^\star$ injective vérifiant, pour tout triplet $(X,Y,Z)$ d'éléments de l'algèbre de Lie $\mathfrak{g}$, les conditions
\begin{enumerate}
\item[\emph{\textbf{(AS)}}] \emph{[Antisymétrie]}
\[
\langle \omega X,Y \rangle
= - \langle \omega Y,X \rangle
\]
\item[\emph{\textbf{(CS)}}] \emph{[Cocycle symplectique]}
\[
\langle \omega [X,Y],Z \rangle
+ \langle \omega [Y,Z],X \rangle 
+ \langle \omega [Z,X],Y \rangle
=0 
\]
\end{enumerate}
définit sur $G$ un tenseur faiblement symplectique 
$\Omega : \mathfrak{X}(G) \to \mathfrak{X}^\star (G)$ invariant à gauche.\\
%!+ ajout du tenseur de Poisson $P$ invariant à droite compatible avec $\Omega$
De manière analogue, on définit un tenseur de Poisson $P$ invariant à droite à partir d'un cocycle de Poisson\footnote{$P_e$ vérifie les propriétés (12.32) et (12.33) dans \cite{MaMo84}.} $P_e$.\\

$G$ est alors muni d'une structure  $\operatorname{P\Omega}$ (\cite{MaMo84}, 12).
\end{example}

\subsubsection{Opérateur de récursion}
\label{___OperateurDeRecursion}

Dans le cadre bihamiltonien, la notion d'opérateur de récursion est intimement lié aux deux structures de Poisson compatibles dont l'une est inversible. Un des intérêts majeur de cette opérateur est qu'il permet d'obtenir des intégrales du mouvement pour le système dynamique associé.\\
Ce processus bien connu en dimension finie peut être étendu au cadre d'espaces de Hilbert séparables où la notion de système Hamiltonien complètement intégrable est associée à la recherche de variables action-angle dans lequel le système peut être linéarisé (e.g. l'équation de Korteveg-de Vries apparaît comme un système intégrable sur une tour d'espaces de Hilbert comme décrit dans \cite{KoKu16}).\\
On trouvera dans l'ouvrage de référence \cite{MaMo84} des exemples d'opérateurs de récursion dans le cadre de structures de Banach. S'agissant de l'existence de variables action-angle, on se réfèrera à \ref{___Structure de Poisson partielles polarisées}.\\
Il est à noter que ces opérateurs de récursion apparaissent aussi dans le cadre plus général des variétés de Fréchet comme indiqué dans \cite{MMR85}. 

\begin{definition}
\label{D_OperateurRecursionPourVarietesPOmegaPartielles}
Soit une variété $\operatorname{P\Omega}$ partielle. On appelle \emph{opérateur de récursion}\index{opérateur de récursion} le morphisme $N = P \circ \Omega : TM \to TM$.
\end{definition}

Pour une variété P$\Omega$ partielle, l'opérateur de récursion $N$ est un \emph{tenseur de Nijenhuis}\index{tenseur de Nijenhuis}, i.e. vérifie la condition\index{TN@$T(N)$ (torsion du tenseur $N$)}
\[
T(N) = 0
\]
où la \emph{torsion}\index{torsion} du tenseur $N$ est donnée, pour tout couple de champs de vecteurs $(X,Y)$ par
\begin{eqnarray}
\label{eq_TorsionTenseurN}
T(N)(X,Y)
=
[NX,NY] - N \left( [NX,Y] + [X,NY] - N[X,Y] \right)
\end{eqnarray}
Cette relation
\footnote{cf. \cite{MaMo84}, (B.2.3)} 
est alors équivalente à
\begin{eqnarray}
\label{eq_TorsionTenseurNDeriveeDeLie}
L_{NX}(N) - N.L_X(N) = 0
\end{eqnarray}
%%%%%%%%%%%%%%%%%%%%%%%%%%%%%%%%%%%%%%%
%In 1967, Gardner, Greene, Kruskal, and Miura presented in \cite{GGKM} a method, known
%as the inverse scattering transform, to solve the initial-value problem for the KdV.
%They showed that $u(x, t)$ can be obtained from $u(x, 0)$ with the help of the solution
%to the inverse scattering problem for the $1$-D Schrödinger equation.
%[GGKM67] C. S. Gardner, J. M. Greene, M. D. Kruskal, and R. M. Miura, Method for solving the Korteweg-de Vries equation, Phys. Rev. Lett. 19 (1967), 1095--1097.

% extract from
% T. Aktosun, 
%Solitons and Inverse Scattering Transform

%%%%%%%%%%%%%%%%%%%%%%%%%%%%%%%%%%%%%%%

%\section*{References}

\printindex

\begin{minipage}[t]{10cm}
\begin{flushleft}
\small{
\textsc{Patrick Cabau}
\\*e-mail: patrickcabau@gmail.com
\\[0.4cm]

\textsc{Fernand Pelletier}
\\UMR 5127 CNRS, Universit\'e de Savoie Mont Blanc, LAMA
\\Campus Scientifique
Le Bourget-du-Lac, 73370, France
\\*e-mail: fernand.pelletier@univ-smb.fr
\\[0.4cm]
}
\end{flushleft}
\end{minipage}

\end{document}